\documentclass[12pt,french,english]{amsart}

\usepackage[T1]{fontenc}
\usepackage[latin9]{inputenc}
\usepackage{geometry}
\usepackage{textcomp}
\usepackage{amsthm}
\usepackage[dvips]{graphicx}
\usepackage{amssymb}

\numberwithin{equation}{section} %% Comment out for sequentially-numbered
\numberwithin{figure}{section} %% Comment out for sequentially-numbered
\theoremstyle{plain}
\theoremstyle{plain}
\newtheorem{thm}{Theorem}
  \theoremstyle{plain}
  \newtheorem{conjecture}[thm]{Conjecture}
  \theoremstyle{plain}
  \newtheorem{prop}[thm]{Proposition}
  \theoremstyle{plain}
  \newtheorem{fact}[thm]{Fact}
  \theoremstyle{remark}
  \newtheorem*{acknowledgement*}{Acknowledgement}
  \theoremstyle{remark}
  \newtheorem{notation}[thm]{Notation}
  \theoremstyle{definition}
  \newtheorem{defn}[thm]{Definition}
  \theoremstyle{remark}
  \newtheorem{rem}[thm]{Remark}
  \theoremstyle{plain}
  \newtheorem{lem}[thm]{Lemma}
  \theoremstyle{plain}
  \newtheorem{cor}[thm]{Corollary}
  \theoremstyle{remark}
  \newtheorem*{rem*}{Remark}

\newcommand{\nwc}{\newcommand}
\nwc{\nwt}{\newtheorem}
\nwt{coro}{Corollary}
\nwt{ex}{Example}

\nwc{\mf}{\mathbf} %Latex (as in \bf not tilted math letters)
\nwc{\blds}{\boldsymbol} %Latex 
\nwc{\ml}{\mathcal} %Latex

\nwc{\lam}{\lambda}
\nwc{\del}{\delta}
\nwc{\Del}{\Delta}
\nwc{\Lam}{\Lambda}
\nwc{\elll}{\ell}
\nwc{\IA}{\mathbb{A}} %algebraic
\nwc{\IB}{\mathbb{B}} %ball
\nwc{\IC}{\mathbb{C}} %complex
\nwc{\ID}{\mathbb{D}} %Dedekind
\nwc{\IE}{\mathbb{E}} %Euklides
\nwc{\IF}{\mathbb{F}} %finite field
\nwc{\IG}{\mathbb{G}} %Gauss
\nwc{\IH}{\mathbb{H}} %Hilbert\N-subgroup
\nwc{\IN}{\mathbb{N}} %natural
\nwc{\IP}{\mathbb{P}} %prime
\nwc{\IQ}{\mathbb{Q}} %rational
\nwc{\IR}{\mathbb{R}} %real
\nwc{\IS}{\mathbb{S}} %sphere
\nwc{\IT}{\mathbb{T}} %torus
\nwc{\IZ}{\mathbb{Z}} %integers
\def\bbbone{{\mathchoice {1\mskip-4mu {\rm{l}}} {1\mskip-4mu {\rm{l}}}
{ 1\mskip-4.5mu {\rm{l}}} { 1\mskip-5mu {\rm{l}}}}}
\def\bbleft{{\mathchoice {[\mskip-3mu {[}} {[\mskip-3mu {[}}{[\mskip-4mu {[}}{[\mskip-5mu {[}}}}
\def\bbright{{\mathchoice {]\mskip-3mu {]}} {]\mskip-3mu {]}}{]\mskip-4mu {]}}{]\mskip-5mu {]}}}}
\nwc{\setK}{\bbleft 1,K \bbright}
\nwc{\setN}{\bbleft 1,\cN \bbright}
\nwc{\va}{{\bf a}}
\nwc{\vb}{{\bf b}}
\nwc{\vc}{{\bf c}}
\nwc{\vd}{{\bf d}}
\nwc{\ve}{{\bf e}}
\nwc{\vf}{{\bf f}}
\nwc{\vg}{{\bf g}}
\nwc{\vh}{{\bf h}}
\nwc{\vi}{{\bf i}}
\nwc{\vI}{{\bf I}}
\nwc{\vj}{{\bf j}}
\nwc{\vk}{{\bf k}}
\nwc{\vl}{{\bf l}}
\nwc{\vm}{{\bf m}}
\nwc{\vM}{{\bf M}}
\nwc{\vn}{{\bf n}}
\nwc{\vo}{{\it o}}
\nwc{\vp}{{\bf p}}
\nwc{\vq}{{\bf q}}
\nwc{\vr}{{\bf r}}
\nwc{\vs}{{\bf s}}
\nwc{\vt}{{\bf t}}
\nwc{\vu}{{\bf u}}
\nwc{\vv}{{\bf v}}
\nwc{\vw}{{\bf w}}
\nwc{\vx}{{\bf x}}
\nwc{\vy}{{\bf y}}
\nwc{\vz}{{\bf z}}
\nwc{\balpha}{\blds{\alpha}}
\nwc{\bal}{\blds{\alpha}}
\nwc{\bbeta}{\blds{\beta}}
\nwc{\bgamma}{\blds{\gamma}}
\nwc{\bep}{\blds{\epsilon}}
\nwc{\barbep}{\overline{\blds{\epsilon}}}
\nwc{\bnu}{\blds{\nu}}
\nwc{\bmu}{\blds{\mu}}

\nwc{\bk}{\blds{k}}
\nwc{\bm}{\blds{m}}
\nwc{\bM}{\blds{M}}
\nwc{\bp}{\blds{p}}
\nwc{\bq}{\blds{q}}
\nwc{\bn}{\blds{n}}
\nwc{\bv}{\blds{v}}
\nwc{\bw}{\blds{w}}
\nwc{\bx}{\blds{x}}
\nwc{\bxi}{\blds{\xi}}
\nwc{\by}{\blds{y}}
\nwc{\bz}{\blds{z}}

\nwc{\cA}{\ml{A}}
\nwc{\cB}{\ml{B}}
\nwc{\cC}{\ml{C}}
\nwc{\cD}{\ml{D}}
\nwc{\cE}{\ml{E}}
\nwc{\cF}{\ml{F}}
\nwc{\cG}{\ml{G}}
\nwc{\cH}{\ml{H}}
\nwc{\cI}{\ml{I}}
\nwc{\cJ}{\ml{J}}
\nwc{\cK}{\ml{K}}
\nwc{\cL}{\ml{L}}
\nwc{\cM}{\ml{M}}
\nwc{\cN}{\ml{N}}
\nwc{\cO}{\ml{O}}
\nwc{\cP}{\ml{P}}
\nwc{\cQ}{\ml{Q}}
\nwc{\cR}{\ml{R}}
\nwc{\cS}{\ml{S}}
\nwc{\cT}{\ml{T}}
\nwc{\cU}{\ml{U}}
\nwc{\cV}{\ml{V}}
\nwc{\cW}{\ml{W}}
\nwc{\cX}{\ml{X}}
\nwc{\cY}{\ml{Y}}
\nwc{\cZ}{\ml{Z}}

\nwc{\tA}{\widetilde{A}}
\nwc{\tB}{\widetilde{B}}
\nwc{\tE}{E^{\vareps}}
\nwc{\tk}{\tilde k}
\nwc{\tN}{\tilde N}
\nwc{\tP}{\widetilde{P}}
\nwc{\tQ}{\widetilde{Q}}
\nwc{\tR}{\widetilde{R}}
\nwc{\tV}{\widetilde{V}}
\nwc{\tW}{\widetilde{W}}
\nwc{\ty}{\tilde y}
\nwc{\teta}{\tilde \eta}
\nwc{\tdelta}{\tilde \delta}
\nwc{\tlambda}{\tilde \lambda}
\nwc{\ttheta}{\tilde \theta}
\nwc{\tvartheta}{\tilde \vartheta}
\nwc{\tPhi}{\widetilde \Phi}
\nwc{\tpsi}{\tilde \psi}

\nwc{\To}{\longrightarrow} %limits

\nwc{\ad}{\rm ad}
\nwc{\eps}{\epsilon}
\nwc{\ep}{\epsilon}
\nwc{\vareps}{\varepsilon}

\nwc{\err}{\epsilon}
\def\ep{\epsilon}

\def\sq2{\sqrt{2}}

\def\t2{{\mathbb T}^2}
\def\s2{{\mathbb S}^2}
\def\hn{\mathcal{H}_{N}}

\def\T{\mathbb{T}}
\def\R{\mathbb{R}}

\def\Z{\mathbb{Z}}

\nwc{\lap}{\bigtriangleup}
\nwc{\rest}{\restriction}
\nwc{\Diff}{\operatorname{Diff}}
\nwc{\dist}{\operatorname{dist}}
\nwc{\diam}{\operatorname{diam}}
\nwc{\Res}{\operatorname{Res}}
\nwc{\Spec}{\operatorname{Spec}}
\nwc{\Vol}{\operatorname{Vol}}
\nwc{\Op}{\operatorname{Op}}
\nwc{\OpN}{\operatorname{Op}_N}
\nwc{\Oph}{\operatorname{Op}_\hbar}
\nwc{\supp}{\operatorname{supp}}
\nwc{\Span}{\operatorname{span}}
\nwc{\neigh}{\operatorname{neigh}}

\nwc{\dia}{\varepsilon}
\nwc{\cut}{f}
\nwc{\qm}{u_\hbar}
\nwc{\cPr}{\mathcal{P}r}

\def\hto0{\xrightarrow{\hbar\to 0}}

\def\rto0{\xrightarrow{r\to 0}}

\providecommand{\abs}[1]{\lvert#1\rvert}
\providecommand{\norm}[1]{\lVert#1\rVert}
\providecommand{\set}[1]{\left\{#1\right\}}

\nwc{\la}{\langle}
\nwc{\ra}{\rangle}
\nwc{\lp}{\left(}
\nwc{\rp}{\right)}

\nwc{\bequ}{\begin{equation}}
\nwc{\be}{\begin{equation}}
\nwc{\ben}{\begin{equation*}}
\nwc{\bea}{\begin{eqnarray}}
\nwc{\bean}{\begin{eqnarray*}}
\nwc{\bit}{\begin{itemize}}
\nwc{\bver}{\begin{verbatim}}

\nwc{\eequ}{\end{equation}}
\nwc{\ee}{\end{equation}}
\nwc{\een}{\end{equation*}}
\nwc{\eea}{\end{eqnarray}}
\nwc{\eean}{\end{eqnarray*}}
\nwc{\eit}{\end{itemize}}
\nwc{\ever}{\end{verbatim}}

\newcommand{\defeq}{\stackrel{\rm{def}}{=}}

\usepackage{babel}
\addto\extrasfrench{\providecommand{\fg}{\ifdim\lastskip>\z@\unskip\fi~\frqq}}

\begin{document}

\title{Entropy of chaotic eigenstates}
\thanks{Notes of the minicourse given at the Centre de Recherches Math\'ematiques, Montr\'eal, during the workshop "Spectrum and Dynamics", April 2008. }

\author{Stéphane Nonnenmacher}

\address{Institut de Physique Théorique, CEA-Saclay, 91191 Gif-sur-Yvette,
France}

\email{snonnenmacher@cea.fr}

\maketitle

\section{Overview and Statement of results}

\subsection{Introduction}

These lectures present a recent approach, mainly developed by Nalini Anantharaman and the author, aimed at studying the high-frequency
eigenmodes of the Laplace-Beltrami operator $\Delta=\Delta_{X}$ on
compact riemannian manifolds $(X,g)$ for which the sectional curvature
is everywhere negative. It is well-known that the geodesic flow on
such a manifold (which takes place on the unit cotangent bundle $S^{*}X$)
is strongly chaotic, in the sense that it is uniformly hyperbolic
(Anosov). This flow leaves invariant the natural smooth measure on
$S^{*}X$, namely the Liouville measure $\mu_{L}$ (which is also
the lift of the Lebesgue measure on $X$). Studying the eigenstates
of $\Delta$ is thus a part of {}``quantum chaos''. In these notes
we extend the study to more general Schrödinger-like operators
in the semiclassical limit, such that the corresponding Hamitonian
flow on some compact energy shell $\cE$ has the Anosov property.
We also consider the case of quantized Anosov diffeomorphisms on
the torus, which are popular toy models in the quantum chaos literature.
To set up the problem we first stick to the Laplacian.

For a general riemannian manifold $(X,g)$ of dimension $d\geq2$,
there exist no explicit, not even approximate expression for the eigenmodes
of the Laplacian. One way to ``describe'' these modes consists
in comparing them (as {}``quantum'' invariants) with {}``classical''
invariants, namely probability measures on $S^{*}X$, invariant w.r.to
the geodesic flow. To this aim, starting from the full sequence of
eigenmodes $(\psi_{n})_{n\geq0}$ one can construct a family of invariant
probability measures on $S^{*}X$, called \emph{semiclassical measures}.
Each such measure $\mu_{sc}$ can be associated with a subsequence
of eigenmodes $(\psi_{n_{j}})_{j\geq1}$ which share, in the limit
$j\to\infty$, the same macroscopic localization properties, both
on the manifold $X$ and in the velocity (or momentum) space: these
macroscopic localization properties are {}``represented'' by $\mu_{sc}$.
One (far-reaching) aim would be a complete classification of the semiclassical
measures associated with a given manifold $(X,g)$. 

This goal is much too ambitious, starting from the fact that the set
of invariant measures is itself not always well-understood. We will
thus restrict ourselves to the class of manifolds described above,
namely manifolds $(X,g)$ of negative sectional curvature. One advantage
is that the classical dynamics is at the same time {}``irregular''
(in the sense of {}``chaotic''), and {}``homogeneous''.
The geodesic flow on such a manifold is (semi)conjugated with a suspended
flow over a simple symbolic dynamics (a subshift of finite type over
a finite alphabet), which allows one to explicitly construct many
different invariant measures. For instance, such a flow admits infinitely
many isolated (unstable) periodic orbits $\gamma$, each of which
carries a natural probability invariant measure $\mu_{\gamma}$. The
set of periodic orbits is so large that the measures $\left\{ \mu_{\gamma}\right\} $
form a dense subset (in the weak-{*} topology) of the set of invariant
probability measures. Hence, it would be interesting to know whether
some high-frequency eigenmodes can be asymptotically localized near
certain periodic orbits, leading to semiclassical measures of the
form \begin{equation}
\mu=\sum_{\gamma}p_{\gamma}\mu_{\gamma},\qquad\mbox{with }\:\sum_{\gamma}p_{\gamma}=1,\quad p_{\gamma}\in[0,1].\label{eq:strong-scar}\end{equation}
This possibility was named {}``strong scarring'' by Rudnick-Sarnak
\cite{RudSar94}, in analogy with a weaker form of {}``scarring''
observed by Heller on some numerically computed eigenmodes \cite{Hel84},
namely a {}``nonrandom enhancement'' of the wavefunction in the
vicinity of a certain periodic orbit. In the same paper, Rudnick and
Sarnak conjectured that such semiclassical measures do not exist for
manifolds of negative curvature. More precisely, they formulated the
Quantum Unique Ergodicity conjecture
\begin{conjecture}
{[}Quantum Unique Ergodicity{]} \cite{RudSar94}\label{con:QUE}

Let $(X,g)$ be a compact riemannian manifold of negative curvature.
Then there exist only one semiclassical measure, namely the Liouville
measure $\mu_{L}$.
\end{conjecture}
This conjecture rules out any semiclassical measure of the type \ref{eq:strong-scar}.
It also rules out linear combination of the form \begin{equation}
\mu=\alpha\mu_{L}+(1-\alpha)\sum_{\gamma}p_{\gamma}\mu_{\gamma},\quad\alpha\in[0,1).\label{eq:partial-scar}\end{equation}
The name {}``quantum unique ergodicity'' reminds of a classical
notion: a dynamical system (map or flow) is \emph{uniquely ergodic}
if and only if it admits a unique invariant measure. In the present
case, the classical system is not uniquely ergodic, but the conjecture
is that its quantum analogue conspires to be so. 

This conjecture was formulated several years after the proof of a
general result describing {}``almost all'' the eigenstates $(\psi_{n})$. 
\begin{thm}
{[}Quantum Ergodicity{]} \cite{Schnirel74,Zel87,CdV85}

Let $(X,g)$ be a compact riemannian manifold such that the geodesic
flow is ergodic w.r.to the Liouville measure $\mu_{L}$. Then, there
exists a subsequence $S\subset\IN$ of density $1$, such that the
subsequence $(\psi_{n})_{n\in S}$ is associated with $\mu_{L}$.%
\footnote{A subsequence $S\subset\mathbb{N}$ is said to be of density $1$
iff $\lim_{N\to\infty}\frac{\sharp\left\{ n\in S,\: n\leq N\right\} }{N}=1$.%
}
\end{thm}
The manifolds encompassed by this theorem include the case of negative
curvature, but also more general ones (like manifolds where the curvature
is negative outside a flat cylindrical part ). The proof of this theorem
is quite {}``robust''. It has been generalized to many different
ergodic systems: Hamiltonian flows ergodic on some compact energy
shell \cite{HMR87}, broken geodesic flows on some Euclidean domains
\cite{GL93,ZZ96}, symplectic diffeomorphism (possibly with discontinuities)
on a compact phase space \cite{BDB96}. This result leaves open the
possibility of \emph{exceptional subsequences} $(\psi_{n_{j}})_{j\geq1}$
(necessarily of density zero) of eigenmodes with different localization
properties. 

The QUE conjecture was motivated by partial results concerning a much
more restricted class of manifolds, namely compact quotients of the
hyperbolic disk $\Gamma\backslash\mathbb{H}$ for which $\Gamma$
is an arithmetic subgroup of $SL_{2}(\IR)$ %
\footnote{More precisely, $\Gamma$ is derived from an Eichler order in a quaternion
algebra.%
}. Such manifolds admit a commutative algebra of selfadjoint Hecke
operators $(T_{k})_{k\geq2}$ which all commute with the Laplacian.
It thus makes sense to preferably consider an eigenbasis $(\psi_{n})_{n\geq0}$
made of joint eigenmodes of the Laplacian and the Hecke operators
(called Hecke eigenmodes), and the associated semiclassical measures
(called Hecke semiclassical measures)%
\footnote{It is widely believed that the spectrum of $\Delta$ on such a manifold
is simple, in which case the restriction to Hecke eigenmodes is not
necessary.%
}. Rudnick-Sarnak proved that the only Hecke semiclassical measure
of the form \ref{eq:partial-scar} is the Liouville measure ($\alpha=1$).
Finally, Lindenstrauss \cite{Linden06} showed that for such manifolds,
the only Hecke semiclassical measure is the Liouville measure, thus
proving an arithmetic form of QUE. He used as an intermediate step
a lower bound for the \emph{Kolmogorov-Sinai (KS) entropy} of Hecke
semiclassical measures, which he proved in a joint work with Bourgain.
\begin{prop}
\cite{BourLin03} Let $X$ be an arithmetic quotient $\Gamma\backslash\mathbb{H}$.
Consider a Hecke semiclassical measure $\mu_{sc}$. Then for any $\rho\in S^{*}X$
any small $\tau,\eps>0$, the measure of the tube $B(\rho,\eps,\tau)$
of diameter $\eps$ around the stretch of trajectory $[\rho,g^{\tau}\rho]$
is bounded by\[
\mu_{sc}\left(B(\rho,\eps,\tau)\right)\leq C_{\tau}\,\eps^{2/9}.\]
As a consequence, for almost any ergodic component $\mu_{erg}$ of
$\mu_{sc}$, one has \begin{equation}
H_{KS}(\mu_{erg})\geq2/9.\label{eq:bound-arithmetic}\end{equation}

\end{prop}
As we will see below, the KS entropy is an affine quantity, therefore
$H_{KS}(\mu_{sc})$ also satisfies the same lower bound.

\subsection{Entropy as a measure of localization\label{sub:Entropy-1}}

In the previous section we already noticed a relationship between
phase space localization and entropy: a uniform lower bound on the
measure of thin tubes implies a positive lower bound on the entropy
of the measure. For this reason, it is meaningful to consider the
KS entropy of a given invariant measure as a {}``quantitative indicator
of localization'' of that measure. In section \ref{sub:KSentropy}
we will give a precise definition of the entropy. For now, let us
only provide a few properties valid in the case of Anosov flows \cite[Chap. 4 ]{KatHas95}.
\begin{enumerate}
\item $H_{KS}(\bullet)$ is a real function defined on the set of invariant
probability measures. It takes values in a finite interval $[0,H_{\max}]$
and is upper semicontinuous. In information-theoretic language, it
measures the average complexity of the flow w.r.to that measure. The
maximum entropy $H_{\max}=H_{top}(S^{*}X)$ is also the \emph{topological
entropy} of the flow on $S^{*}X$, which is a standard measure of
the \emph{complexity} of the flow.
\item a measure $\mu_{\gamma}$ supported on a single periodic orbit has
zero entropy.
\item Since the flow is Anosov, at each point $\rho\in S^{*}X$ the tangent
space splits into $T_{\rho}S^{*}X=E_{\rho}^{u}\oplus E_{\rho}^{s}\oplus E_{\rho}^{0}$,
respectively the unstable, stable subspaces and the flow direction.
Each of these subspaces is flow-invariant. Let us call $J^{u}(\rho)=J_{1}^{u}(\rho)=\left|\det g_{\restriction E^{u}(\rho)}^{1}\right|$
the unstable Jacobian at the point $\rho$. Then, any invariant measure
satisfies \begin{equation}
H_{KS}(\mu)\leq\int\log J^{u}(\rho)\, d\mu(\rho).\qquad\mbox{(Ruelle inequality)}\label{eq:Ruelle}\end{equation}
 The equality is reached iff $\mu=\mu_{L}$. In constant curvature,
one has $H_{KS}(\mu_{L})=H_{\max}$. 
\item the entropy is affine: $H_{KS}\left(\alpha\mu_{1}+(1-\alpha)\mu_{2}\right)=\alpha H_{KS}(\mu_{1})+(1-\alpha)H_{KS}(\mu_{2})$.
\end{enumerate}
Apart from the result of Bourgain-Lindenstrauss (relative to arithmetic
surfaces), the first result on the entropy of semiclassical measures
was obtained by Anantharaman:
\begin{thm}
\cite{Ana08} Let $(X,g)$ be a manifold of negative sectional curvature.
Then, there exists $c>0$ such that any semiclassical measure $\mu_{sc}$
satisfies \[
H_{KS}(\mu_{sc})\geq c.\]
Furthermore, the flow restricted to the support of $\mu_{sc}$ has
a nontrivial complexity: its topological entropy satisfies \[
H_{top}(\supp\mu_{sc})\geq\frac{\Lambda_{\min}^{u}}{2},\]
where $\Lambda_{\min}^{u}=\lim_{t\to\infty}\inf_{\rho}\frac{1}{t}\log J_{t}^{u}(\rho)$
is the minimal volume expanding rate of the unstable manifold.
\end{thm}
The lower bound $c>0$ is not very explicit and is rather {}``small''.
This is to be opposed to the lower bound controlling the complexity
of the flow on $\supp\mu_{sc}$, given in terms of the hyperbolicity
of the flow. The lower bound on the KS entropy was improved by Anantharaman, Koch and the author:
\begin{thm}
\cite{AnaNo07-2,AnaKoNo06} \label{thm:bound-general-AKN}Let $(X,g)$
be a $d$-dimensional manifold of negative sectional curvature. Then,
any semiclassical measure $\mu_{sc}$ satisfies \begin{equation}
H_{KS}(\mu_{sc})\geq\int\log J^{u}(\rho)\, d\mu_{sc}(\rho)-\frac{(d-1)\lambda_{\max}}{2},\label{eq:bound-general-AKN}\end{equation}
where $\lambda_{\max}=\lim_{t\to\infty}\sup_{\rho}\frac{1}{t}\log\left|dg^{t}(\rho)\right|$
is the maximal expansion rate of the flow.

In the particular case where $X$ has constant curvature $-1$, this
bound reads\begin{equation}
H_{KS}(\mu_{sc})\geq\frac{d-1}{2}=\frac{H_{top}(S^{*}X)}{2}.\label{eq:bound--1}\end{equation}

\end{thm}
In the constant curvature case, the above bound roughly means that
high-frequency eigenmodes of the Laplacian are \emph{at least half-delocalized}.
Still, the bound (\ref{eq:bound-general-AKN}) is not very satisfactory
when the curvature varies much across $X$; since $\int\log J^{u}d\mu$
may be as small as $\Lambda_{\min}$, the right hand side in (\ref{eq:bound-general-AKN})
can become negative (therefore trivial) in case $\Lambda_{\min}<\frac{(d-1)\lambda_{\max}}{2}$.
The following lower bound seems more natural:
\begin{conjecture}
\label{con:optimal?}Let $(X,g)$ be a manifold of negative sectional
curvature. Then, any semiclassical measure $\mu_{sc}$ satisfies \begin{equation}
H_{KS}(\mu_{sc})\geq\frac{1}{2}\int\log J^{u}(\rho)\, d\mu_{sc}(\rho).\label{eq:bound-optimal}\end{equation}

\end{conjecture}
This bound is identical with (\ref{eq:bound--1}) in the case of curvature
$-1$. Using a nontrivial extension of the methods developed in \cite{AnaKoNo06},
it has been recently proved by G.Rivière in the case of surfaces ($d=2$)
of \emph{nonpositive curvature} \cite{Riv08,Riv09}, and also by B.Gutkin
for a certain class of quantized interval maps \cite{Gutkin08}.

\subsection{Generalization to Anosov Hamiltonian flows and symplectic maps}

The conjecture \ref{con:optimal?} is weaker than the QUE conjecture
\ref{con:QUE}. We expect the bound (\ref{eq:bound-optimal}) to apply
as well to more general classes of quantized chaotic dynamical systems,
like Anosov Hamiltonian flows or symplectic diffeomorphisms on a compact
phase space. In these notes we will extend the bound (\ref{eq:bound-general-AKN})
to these more general Anosov systems (the first instance of this entropic
bound actually appeared when studying the Walsh-quantized baker's
map \cite{AnaNo07-1}). The central result of these notes is the following
theorem.
\begin{thm}
i)\label{thm:bound-general2} Let $p(x,\xi;\hbar)$ be a Hamilton
function on some phase space $T^{*}X$ with principal symbol $p_{0}$,
such that the energy shell $\cE=p_{0}^{-1}(0)$ is compact, and the
Hamiltonian flow $g^{t}=e^{tH_{p_{0}}}$ on $\cE$ is Anosov (see
$\S$\ref{sub:Hamiltonians}). Let $P(\hbar)=\Oph(p)$ be the $\hbar$-quantization
of $p$. Then, any semiclassical measure $\mu_{sc}$ associated with
a sequence of null eigenmodes $\left(\psi_{\hbar}\right)_{\hbar\to0}$
of $P(\hbar)$ satisfies the following entropic bound: \begin{equation}
H_{KS}(\mu_{sc})\geq\int_{\cE}\log J^{u}(\rho)\, d\mu_{sc}(\rho)-\frac{(d-1)\lambda_{\max}}{2}.\label{eq:bound-general2}\end{equation}

ii) Let $\cE=\IT^{2d}$ be the $2d$-dimensional torus, equipped with
its standard symplectic structure. Let $\kappa:\cE\to\cE$ be an Anosov
diffeomorphism, which can be quantized into a family of unitary propagators
$\left(U_{\hbar}(\kappa)\right)_{\hbar\to0}$ defined on (finite dimensional)
quantum Hilbert spaces $\left(\cH_{\hbar}\right)$. Then, any semiclassical
measure $\mu_{sc}$ associated with a sequence of eigenmodes $(\psi_{\hbar}\in\cH_{\hbar})_{\hbar\to0}$
of $U_{\hbar}(\kappa)$ satisfies the entropic bound\begin{equation}
H_{KS}(\mu_{sc})\geq\int\log J^{u}(\rho)\, d\mu_{sc}(\rho)-\frac{d\lambda_{\max}}{2}.\label{eq:bound-maps}\end{equation}

\end{thm}
In §\ref{sec:Anosov-maps} we will state more precisely what we meant
by a {}``quantized torus diffeomorphism''. Let us mention that the
same proof could apply as well to Anosov symplectic maps on more general
symplectic manifolds admitting some form of quantization. We restricted
the statement to the $2d$-torus because simple Anosov diffeomorphisms
on $\IT^{2d}$ can be constructed, and their quantization is by now
rather standard. As we explain below, their study has revealed interesting
features regarding the QUE conjecture.

\subsection{Counterexamples to QUE for Anosov diffeomorphisms\label{sub:Counterexamples}}

For the simplest Anosov diffeos on $\IT^{2}$, namely the hyperbolic
symplectomorphisms of the torus (colloquially known as {}``quantum
Arnold's cat map''), the QUE conjecture is known to fail. Indeed,
in \cite{FNDB03} counterexamples to QUE for {}``cat maps'' on the
$2$-dimensional torus were exhibited, in the form of explicit semiclassical
sequences of eigenstates of the quantized map, associated with semiclassical
measures of type (\ref{eq:partial-scar}) with $\alpha=1/2$. In \cite{FNDB04} we also
showed that, for this particular map, semiclassical measures of the
form (\ref{eq:partial-scar}) necessarily satisfy $\alpha\geq1/2$. 

In the case of toral symplectomorphisms on higher dimensional tori
$\mathbb{T}^{2d}$, Kelmer \cite{Kelm07,Kelm10} has exhibited semiclassical
measures in the form of the Lebesgue measure on certain co-isotropic
affine subspaces $\Lambda\subset\mathbb{T}^{2d}$ of the torus invariant
through the map:\[
\mu_{sc}=\mu_{L\restriction\Lambda}.\]
For another example of a chaotic map (the baker's map quantized \emph{à
la Walsh}), we were able to construct semiclassical measures of purely
fractal (self-similar) nature.
\begin{fact}
The above counterexamples to QUE all satisfy the entropy bound (\ref{eq:bound-optimal}),
and some of them (like the measure (\ref{eq:partial-scar}) with $\alpha=1/2$)
saturate that bound. 
\end{fact}
It is worth mentioning an interesting result obtained by S.Brooks
\cite{Broo08} in the case of the {}``quantum Arnold's cat map''.
Brooks takes into accout the possibility to split any invariant measure
$\mu$ into ergodic components: \[
\mu=\int_{\t2}\mu_{x}\, d\mu(x),\]
where the probability measure $\mu_{x}$, defined for $\mu$-almost
every point $x$, is ergodic. The affineness of the KS entropy ensures
that \begin{equation}
H_{KS}(\mu)=\int_{\t2}H_{KS}(\mu_{x})\, d\mu(x),\label{eq:ergodic-decompo}\end{equation}
so to get a lower bound on $H_{KS}(\mu)$ it is sufficient to show
that {}``high-entropy'' components have a positive weight in $\mu$.
Brooks's result reads as follows:
\begin{thm}
\cite{Broo08}\label{thm:Brooks}Let $\kappa:\t2\circlearrowleft$
be a linear hyperbolic symplectomorphism, with positive Lyapunov exponent
$\lambda$ ($\lambda$ is also equal to the topological entropy of $\kappa$ on $\IT^2$).
Fix any $0<H_{0}<\frac{\lambda}{2}$, and consider any associated
semiclassical mesure $\mu_{sc}$. Then the following inequality holds:\[
\mu_{sc}\left\{ x\,:\, H_{KS}(\mu_{x})<H_{0}\right\} \leq\mu_{sc}\left\{ x\,:\, H_{KS}(\mu_{x})>\lambda-H_{0}\right\} .\]

\end{thm}
This result directly implies (through (\ref{eq:ergodic-decompo}))
the bound $H_{KS}(\mu_{sc})\geq\frac{\lambda}{2}$, but it also implies
(by sending $H_{0}\to0$) the above-mentioned fact that the weight
of atomic components of $\mu_{sc}$ is smaller or equal to the weight
of its Lebesgue component.

\subsection{Plan of the paper}

These lectures reproduce most of the proofs of \cite{AnaNo07-2,AnaKoNo06}
dealing with eigenstates of the Laplacian on manifolds of negative
curvature. Yet, we extend the proofs in order to deal with more general
Hamiltonian flows of Anosov type (for instance, adding some small
potential to the free motion on $X$). This can be done at the price
of using more general, {}``microlocal'' partitions of unity, as
opposed to the {}``local'' partition of unity used in \cite{AnaNo07-2,AnaKoNo06}
(which was given in terms of functions $\pi_{k}(x)$ only depending
on the position variable). This microlocal setting is somehow
more natural, since it does not depend on the way unstable manifolds
project down to the manifold $X$. It is also more natural in the
case of Anosov maps.

In §\ref{sec:Preliminaries} we recall the semiclassical tools we
will need, starting with the $\hbar$-pseudodifferential calculus
on a compact manifold, and including some exotic classes of
symbols. We also define the main object of study, namely the semiclassical
measures associated with sequences of null eigenstates $(\psi_{\hbar})_{\hbar\to0}$
of a family of Hamiltonians $(P(\hbar))_{\hbar\to0}$. In the central
section §\ref{sec:Entropies} we provide the proof of Thm. \ref{thm:bound-general2},
\emph{i)}, that is in the case of an Anosov Hamiltonian flow on a
compact manifold $X$. We first recall the definition of entropies
and pressures associated with invariant measures. We then introduce
microlocal quantum partitions in §\ref{sub:quantum-partition},
and their refinements used to define quantum entropies and pressures
associated with the eigenstates $\psi_{\hbar}$. We try to provide
some geometric intuition on the operators $\Pi_{\bal}$ defining
these partitions. We then state the central \emph{hyperbolic dispersive
estimate} on the norms of these operators, deferring the proof to
§\ref{sec:Hyperbolic dispersive estimate}. We then introduce several
versions of {}``entropic uncertainty principles'', from the simplest
to the most complex, microlocal form (Prop. \ref{pro:EUP,microlocal})
. We then apply this microlocal EUP in order to bound from below quantum
and classical pressures associated with our eigenmodes. §\ref{sec:Hyperbolic dispersive estimate}
is devoted to the proof of the hyperbolic dispersive estimate. Here
we adapt the proof of \cite{NoZw09}, which is valid in more general
situations than the case of geodesic flows. Finally, in §\ref{sec:Anosov-maps}
we briefly recall the framework of quantized maps on the torus, and
provide the details necessary to obtain Thm.\ref{thm:bound-general2},
\emph{ii).}
\begin{acknowledgement*}
I am grateful to D.Jakobson and I.Polterovich who invited me to give
this minicourse in Montréal and to write these notes. Most of the
material of these notes were obtained through from collaborations
with N.Anantharaman, H.Koch and M.Zworski. I have also been partially
supported by the project ANR-05-JCJC-0107091 of the Agence Nationale
de la Recherche.
\end{acknowledgement*}

\section{Preliminaries and problematics\label{sec:Preliminaries}}

\subsection{Semiclassical calculus on $X$}

The original application of the methods presented below concern the
Laplace-Beltrami operator on a smooth compact manifold $X$ of negative
sectional curvature. To deal with this problem, one needs to define
a certain number of auxiliary operators on $L^{2}(X)$, which are
$\hbar$-pseudodifferential operators on $X$ ($\Psi$DOs), or $\hbar$-Fourier
integral operators on $X$. We will only recall the definition and
construction of the former class.

The Hamiltonians mentioned in Theorem \ref{thm:bound-general2} also
belong to some class of $\hbar$-pseudodifferential operators, but
the manifold $X$ on which they are defined is not necessarily compact
any more. In this setting, the smooth manifold $X$ can be taken as
the Euclidean space $X=\IR^{d}$, or be Euclidean near infinity, that
is $X=X_{0}\sqcup\left(\IR^{d}\setminus B(0,R_{0})\right)$, where
$B(0,R_{0})$ is the ball of radius $R_{0}$ in $\IR^{d}$, and $X_{0}$
is a compact manifold, the boundary of which is smoothly glued to
$\partial B(R_{0})$.

\subsubsection{Symbol classes on $T^{*}X$ and $\hbar$-pseudodifferential calculus}

Let us construct an $\hbar$-quantization procedure on a Riemannian
manifold $X$. To a certain class of well-behaved functions $\left(f(\hbar)\right)_{\hbar\to0}$
on $T^{*}X$ (the physical \emph{observables}, referred to as \emph{symbols}
in mathematics) one can associate, through a well-defined \emph{quantization
procedure $\Oph$,} a corresponding set of operators $\Op_{\hbar}(f(\hbar))$
acting on $C_{c}^{\infty}(X)$. By {}``well-behaved'' one generally
refers to certain conditions on the regularity and growth of the function.
There are many different types of classes of {}``well-behaved symbols'';
we will be using the class\[
S^{m,k}(T^{*}X)=\left\{ f(\hbar)\in C^{\infty}(T^{*}X),\;\left|\partial_{x}^{\alpha}\partial_{\xi}^{\beta}f(\hbar)\right|\leq C_{\alpha,\beta}\hbar^{-k}\la\xi\ra^{m-|\beta|}\right\} .\]
Here we use the {}``japanese brackets'' notation $\la\xi\ra\defeq\sqrt{1+|\xi|^{2}}$.
The estimates are supposed to hold uniformly for $\hbar\in(0,1]$
and $(x,\xi)\in T^{*}X$. The seminorms can be defined locally on
coordinate charts of $X$; due to the factor $\la\xi\ra^{-|\beta|}$,
this class is invariant w.r.to changes of coordinate charts on $X$,
and thus makes sense intrinsically on the manifold $T^*X$. 

Some (important) symbols in this class are of the form \[
f(x,\xi;\hbar)=\hbar^{-k}f_{0}(x,\xi)+f_{1}(x,\xi,\hbar),\quad f_{1}\in S^{m,k'},\quad k'<k.\]
In that case, $\hbar^{-k}f_{0}(x,\xi)$ is called the \emph{principal
symbol} of $f$.

For $X=\IR^{d}$, a symbol $f\in S^{m,k}$ can be quantized
using the Weyl quantization: it acts on $\varphi\in\cS(\R^{d})$ through
as the integral operator\begin{equation}
\Op_{\hbar}^{W}(f)\varphi(x)=\int f\left(\frac{x+y}{2},\xi\right)\, e^{i\la x-y,\xi\ra/\hbar}\,\varphi(y)\,\frac{dy\, d\xi}{(2\pi\hbar)^{d}}.\label{eq:Weyl quant.}\end{equation}
If $f$ is a real function, this operator is essentially selfadjoint
on $L^{2}(\IR^{d})$. 

If $X$ is a more complicated manifold, one can quantize $f$ by first
splitting it into pieces localized on various coordinate charts $V_{\ell}\subset X$,
through a finite partition of unity $1=\sum_{\ell}\phi_{\ell}$, $\supp\phi_{\ell}\subset V_{\ell}$:\[
f=\sum_{\ell}f_{\ell},\quad f_{\ell}=f\times\phi_{\ell}.\]
Each component $f_{\ell}$ can be considered as a function on $T^{*}\IR^{d}$,
and be quantized through (\ref{eq:Weyl quant.}), producing an operator
$\Oph^{\IR}(f_{\ell}$) acting on $C^{\infty}_c(\IR^{d})$. A wavefunction
$\varphi\in C^{\infty}(X)$ will be cut into pieces $\tilde{\phi}_{\ell}\times\varphi$,
where the cutoffs $\supp\tilde{\phi}_{\ell}\subset V_{\ell}$ satisfy
$\tilde{\phi}_{\ell}\phi_{\ell}=\phi_{\ell}$%
\footnote{Throughout the text we will ofen encouter such {}``embedded cutoffs''.
The property $\tilde{\phi}_{\ell}\phi_{\ell}=\phi_{\ell}$ will be
denoted by $\tilde{\phi}_{\ell}\succ\phi_{\ell}$.$ $%
}. Our final quantization is then defined as:\begin{equation}
\Oph(f)\varphi=\sum_{\ell}\tilde{\phi}_{\ell}\times\Oph^{\IR}(f_{\ell})\left(\tilde{\phi}_{\ell}\times\varphi\right).\label{eq:Quantization-mfold}\end{equation}

The image of the class $S^{m,k}(T^{*}X)$ through quantization is
an operator algebra acting on $C_{c}^{\infty}(X)$, denoted by $\Psi^{m,k}(T^{*}X)$.
This algebra has {}``nice'' properties in the semiclassical limit
$\hbar\ll1$. The product of two such operators behaves as a {}``decoration''
of the usual multiplication:\begin{equation}
\Oph(f)\Oph(g)=\Oph(f\sharp g),\label{eq:composition}\end{equation}
where $f\sharp g\in S^{m+m',k+k'}$ admits an asymptotic expansion
of the form \begin{equation}
f\sharp g\sim\sum_{j\geq0}\hbar^{j}(f\sharp g)_{j}.\label{eq:expansion}\end{equation}
Here the first component $(f\sharp g)_{0}=f\times g$, while each
$(f\sharp g)_{j}\in S^{m+m',k+k'}$ is a linear combination of derivatives
$\partial^{\gamma}f\partial^{\gamma'}g$, with $|\gamma|,\,|\gamma'|\leq j$.
In the case $X=\IR^{d}$ and $\Oph$ is the Weyl quantization (\ref{eq:Weyl quant.}),
$\sharp$ is called the Moyal product. In the case $f\in S(T^{*}X)=S^{0,0}(T^{*}X)$,
$\Op_{\hbar}(f)$ can be extended into a continuous operator on $L^{2}(X)$,
and the sharp Gårding inequality ensures that \[
\left\Vert \Op_{\hbar}(f)\right\Vert _{L^{2}}=\left\Vert f\right\Vert _{\infty}+\cO_{f}(\hbar).\]

The quantization procedure $f\mapsto\Oph(f)$ is obviously not unique:
it depends on the choice of coordinates on each chart, on the choice
of quantization $\Oph^{\IR}$, on the choice of cutoffs $\phi_{j},\,\tilde{\phi}_{j}$.
Fortunately, this non-uniqueness becomes irrelevant in the semiclassical
limit.
\begin{prop}
\label{pro:quantization-equivalence}In the semiclassical limit $\hbar\to0$,
two $\hbar$-quantizations differ at most at subprincipal order:\[
\forall f\in S^{m,k}(T^{*}X),\qquad\Op_{\hbar}^{1}(f)-\Op_{\hbar}^{2}(f)\in\Psi^{m,k-1}(T^{*}X).\]

\end{prop}

\subsection{From the Laplacian to more general quantum Hamiltonians\label{sub:Hamiltonians}}

\subsubsection{Rescaling the Laplacian}

One of our objectives is to study an eigenbasis $\left(\psi_{n}\right)_{n\geq0}$
of the Laplace-Beltrami operator on some compact Riemannian manifold
$(X,g)$. To deal with the high-frequency limit $n\gg1$, it turns
out convenient to use a {}``quantum mechanics'' point of view, namely
rewrite the eigenmode equation \[
(\Delta+\lambda_{n}^{2})\psi_{n}=0,\quad\lambda_{n}>0\]
in the form \begin{equation}
\left(-\frac{\hbar_{n}^{2}\,\Delta}{2}\psi_{n}-\frac{1}{2}\right)\,\psi_{n}=0,\quad\hbar_{n}=\lambda_{n}^{-1}.\label{eq:eigenmode-n}\end{equation}
This way, $\hbar_{n}$ appears as an effective Planck's constant (which
is of the order of the \emph{wavelength} of the state $\psi_{n}$).
The rescaled Laplacian operator \[
-\frac{\hbar^{2}\,\Delta}{2}-\frac{1}{2}=P(\hbar)\]
is the $\hbar$-quantization $P(\hbar)=\Oph(p)$ of a certain classical
Hamiltonian \[
p(x,\xi;\hbar)=p_{0}(x,\xi)+\hbar p_{1}(x,\xi)+\ldots.\in S^{2,0}(T^{*}X).\]
The\emph{ }principal symbol $p_{0}(x,\xi)=\frac{|\xi|_{x}^{2}}{2}-1/2$
generates (through Hamilton's equations) the motion of a free particle
on $X$. In particular, the Hamilton flow $g^{t}=\exp tX_{p_{0}}$
restricted to the energy shell $p_{0}^{-1}(0)=S^{*}X$ is the geodesic
flow (in the following, we will often denote by $\cE$ this energy
shell).
\begin{notation}
The Laplacian eigenmodes will often be denoted by $\psi_{\hbar}$
instead of $\psi_{n}$, with the convention that the state $\psi_{\hbar}$
satisfies the eigenvalue equation \begin{equation}
P(\hbar)\,\psi_{\hbar}=\left(-\frac{\hbar^{2}\,\Delta}{2}-\frac{1}{2}\right)\,\psi_{\hbar}=0.\label{eq:eigenmode-P}\end{equation}
\end{notation}
\begin{defn}
We will call $S\subset(0,1]$ a countable set of scales $\hbar$,
with only accumulation point at the origin. A sequence of states indexed
by $\hbar\in S$ will be denoted by $(\varphi_{\hbar})_{\hbar\in S}$,
or sometimes, omitting the reference to a specific $S$, by $(\varphi_{\hbar})_{\hbar\to0}$.
\end{defn}

\subsubsection{Anosov Hamiltonian flows}

In Theorem \ref{thm:bound-general2} we deal with more general Hamiltonians
$p(x,\xi;\hbar)$ on $T^{*}X$, where $X$ is compact or could also
be the Euclidean space $\R^{d}$. The (real) symbol $p$ is assumed
to belong to a class $S^{m,0}(T^{*}X)$, and admit the expansion \[
p(x,\xi,\hbar)=p_{0}(x,\xi)+\hbar^{\nu}p_{1}(x,\xi,\hbar),\qquad p_{0},p_{1}\in S^{m,0}(T^{*}X),\quad\nu>0,\]
that is $p_{0}$ is the principal symbol of $p$. We could as well
consider more general symbol classes, see for instance \cite[Sec. 4.3]{EvZw09}
in the case $X=\IR^{d}$. We assume that 
\begin{enumerate}
\item the energy shell $\cE=p_{0}^{-1}(0)$ is compact, so that $\cE_{\eps}\defeq p_{0}^{-1}([-\eps,\eps])$
is compact as well for $\eps>0$ small enough. 
\item the Hamiltonian flow $g^{t}=e^{tH_{p_{0}}}$ restricted to the energy
shell $\cE$ does not admit fixed points, and is of Anosov type.
\end{enumerate}
The Hamiltonian $p$ is quantized into an operator $P(\hbar)=\Oph(p)\in\Psi^{m,0}$.
The first assumption above implies that, for $\hbar>0$ small enough,
the spectrum of $P(\hbar)$ near zero is purely discrete. We will
focus on sequences of normalized null eigenstates $\left(\psi_{\hbar}\right)_{\hbar\to0}$:
\begin{equation}
P(\hbar)\psi_{\hbar}=0.\label{eq:Energy-condition}\end{equation}

\begin{rem}
If $\psi_{\hbar}$ is a {}``quasi-null'' eigenstate of $P(\hbar)$,
that is if $P(\hbar)\psi_{\hbar}=E(\hbar)\psi_{\hbar}$ with $E(\hbar)=\cO(\hbar^{\nu})$,
then it is a null eigenstate of $\tilde{P}(\hbar)\defeq P(\hbar)-E(\hbar)$,
which admits the same principal symbol $p_{0}$ as $P(\hbar)$. As
a result, Thm. \ref{thm:bound-general2} is also valid for such sequences
of states.
\end{rem}

\subsection{$\hbar$-dependent singular' observables\label{sub:exotic-symbols}}

In the following we will have to use some classes of {}``singular''
$\hbar$-dependent symbols.

\subsubsection{{}``Isotropically singular'' observables}

For $\nu\in[0,1/2)$, we will consider the class \begin{equation}
S_{\nu}^{m,k}(T^{*}X)=\left\{ f(\hbar)\in C^{\infty}(T^{*}X),\;\left|\partial_{x}^{\alpha}\partial_{\xi}^{\beta}f(\hbar)\right|\leq C_{\alpha,\beta}\hbar^{-k-\nu|\alpha+\beta|}\la\xi\ra^{m-|\beta|}\right\} .\label{eq:S_nu}\end{equation}
Such functions can strongly oscillate on scales $\geq\hbar^{\nu}$.
The corresponding operators belong to an algebra $\Psi_{\nu}^{m,k}(T^{*}X)$
which can still be analyzed using an $\hbar$-expansion of the type
(\ref{eq:expansion}). The main difference is that the higher-order
terms $(f\#g)_{j}\in S_{\nu}^{m+m',k+k'+2j\nu}$. Similarly, the Garding
inequality reads, for $f\in S_{\nu}^{0,0}$: \[
\left\Vert \Op_{\hbar}(f)\right\Vert _{L^{2}}=\left\Vert f\right\Vert _{\infty}+\cO_{f}(\hbar^{1-2\nu}),\]
where the implicit constant depends on a certain seminorm of $f$.

\subsubsection{{}``Anisotropically singular'' observables}

We will also need to quantize observables which are {}``very singular''
along certain directions, away from some specific submanifold (see
for instance \cite{SjoZwo99} for a presentation). Consider $\Sigma\subset T^{*}X$
a compact co-isotropic manifold of dimension $2d-D$ (with $D\leq d$).
Near each point $\rho\in\Sigma$, there exist local canonical coordinates
$(y_{i},\eta_{i})$ such that $\Sigma=\left\{ \eta_{1}=\eta_{2}=\cdots=\eta_{D}=0\right\} $.
For some index $\nu\in[0,1)$, we define as follows a class of smooth
symbols $f\in S_{\Sigma,\nu}^{m,k}(T^{*}X)\subset C^{\infty}(T^{*}X\times(0,1])$: 
\begin{itemize}
\item for any family of smooth vector fields $V_{1},\ldots,V_{l_{1}}$ tangent
to $\Sigma$ and of smooth vector fields $W_{1},\ldots,W_{l_{2}}$,
we have in any neighbourhood $\Sigma_{\eps}$ of $\Sigma$: \[
\sup_{\rho\in\Sigma_{\eps}}\left|V_{1}\cdots V_{l_{1}}W_{1}\cdots W_{l_{2}}f(\rho)\right|\leq C\,\hbar^{-k-\nu l_{2}}.\]

\item away from $\Sigma$ we require $|\partial_{x}^{\alpha}\partial_{\xi}^{\beta}|=\cO(\hbar^{-k}\la\xi\ra^{m-|\beta|})$.
\end{itemize}
Such a symbol $f$ can be split into components $f_{j}$ localized
in neighbourhoods $\cV_{j}\subset\Sigma_{\eps}$, plus an {}``external''
piece $f_{\infty}\in S^{m,k}(T^{*}X)$ vanishing near $\Sigma$. Each
piece $f_{j}$ is Weyl-quantized in local adapted canonical coordinates
$(y,\eta)$ on $\cV_{j}$ (as in (\ref{eq:Weyl quant.})), and then
brought back to the original coordinates $(x,\xi)$ using Fourier
integral operators. On the other hand, $f_{\infty}$ is quantized
as in (\ref{eq:Quantization-mfold}). Finally, $\Op_{\Sigma,\hbar}(f)$
is obtained by summing the various contributions. The resulting class
of operators is denoted by $\Psi_{\Sigma,\nu}^{m,k}(T^{*}X)$.

\subsubsection{Sharp energy cutoffs\label{sub:Sharp-energy-cutoffs}}

We will mostly use this quantization relative to the energy layer
$\Sigma\defeq\cE=p_{0}^{-1}(0)$, in order to define a family of sharp
energy cutoffs. Namely, for some small $\delta>0$ we will start from
a cutoff $\chi_{\delta}\in C^{\infty}(\IR)$ such that $\chi_{\delta}(s)=1$
for $|s|\leq e^{-\delta/2}$, $\chi_{\delta}(s)=0$ for $|s|\geq1$.
From there, we define, for each $\hbar\in(0,1]$ and each $n\geq0$,
the rescaled function $\chi^{(n)}\in C_{c}^{\infty}(\IR\times(0,1])$
by \begin{equation}
\chi^{(n)}(s,\hbar)\defeq\chi_{\delta}\left(e^{-n\delta}\,\hbar^{-1+\delta}\, s\right).\label{eq:sharp-cutoff}\end{equation}
The functions $\chi^{(n)}\circ p_{0}$ are {}``sharp'' energy cutoffs,
they belong to the class $S_{\cE,1-\delta}^{-\infty,0}$. We will
always consider $n\leq n_{\max}=C_{\delta}|\log\hbar|$, where the
constant $C_{\delta}<\delta^{-1}-1$, such that $\supp\chi^{(n)}$
is microscopic. 

These cutoffs can be quantized in two ways:
\begin{enumerate}
\item we may directly quantize the function $\chi^{(n)}\circ p_{0}$, into
$\Op_{\cE,\hbar}(\chi^{(n)}\circ p_{0})\in\Psi_{\cE,1-\delta}^{-\infty,0}$. 
\item or we can consider, using functional calculus, the operators $\chi^{(n)}(P(\hbar))$.
These operators (which generally differ from the previous ones) also
belongs to $\Psi_{\cE,1-\delta}^{-\infty,0}$.
\end{enumerate}
The sequence $(\chi^{(n)})_{0\leq n\leq n_{\max}}$ is an increasing
sequence of embedded cutoffs: for each $n$, we have $\chi^{(n+1)}\chi^{(n)}=\chi^{(n)}$
(equivalently, $\chi^{(n+1)}\succ\chi^{(n)}$). More precisely, we
have here 
\begin{equation}
dist\left(\supp\chi^{(n)},,\supp(1-\chi^{(n+1)})\right)\geq\hbar^{1-\delta}e^{\delta n}(e^{\delta/2}-1).\label{eq:cutoff-embedding-class}\end{equation}
This distance between the supports implies the following
\begin{lem}
For any symbol $f\in S_{\cE,1-\delta}^{m,0}$ and any $0\leq n\leq n_{\max}$,
one has\begin{equation}
\left(Id-\Op_{\cE,\hbar}(\chi^{(n+1)}\circ p_{0})\right)\Op_{\cE,\hbar}(f)\,\Op_{\cE,\hbar}(\chi^{(n)}\circ p_{0})=\cO(\hbar^{\infty}).\label{eq:cutoff-embedding}\end{equation}
The same property holds if we replace $\Op_{\cE,\hbar}(\chi^{(n)}\circ p_{0})$
by $\chi^{(n)}(P(\hbar))$.
\end{lem}
Using the calculus of the class $S_{\cE,1-\delta}^{m,0}$, one can
use the ellipticity of $P(\hbar)$ away from $\cE$ to show that,
if $(\psi_{\hbar})$ is a sequence of null eigenstates of $P(\hbar)$,
then\begin{equation}
\left(Id-\chi^{(0)}(P(\hbar))\right)\psi_{\hbar}=\cO(\hbar^{\infty}),\quad\hbar\to0.\label{eq:local-psi-hbar}\end{equation}
That is, in the semiclassical limit the eigenstate $\psi_{\hbar}$
is microlocalized inside the energy layer of width $\hbar^{1-\delta}$
around $\cE$.

\subsection{Semiclassical measures\label{sub:Semiclassical-measures}}

The $\hbar$-semiclassical calculus allows us to define what we mean
by {}``phase space distribution of the eigenstate $\psi_{\hbar}$'',
through the notion of \emph{semiclassical measure}. A Borel measure
$\mu$ on the phase space $T^{*}X$ can be fully characterized by
the set of its values \[
\mu(f)=\int_{T^{*}X}f\, d\mu,\]
over \emph{smooth test functions $f\in C_{c}(T^{*}X)$.} For each
semiclassical scale \emph{$\hbar$}, one can quantize a test function
into a \emph{test operator} $\Op_{\hbar}(f)$ (which is, as mentioned
above, a continuous operator on $L^{2}(X)$). To any normalized state
$\varphi\in L^{2}(X)$ we can then associate the linear functional
\[
f\in C_{c}^{\infty}(T^{*}X)\mapsto\mu_{\hbar,\varphi}(f)\defeq\la\varphi,\Op_{\hbar}(f)\varphi\ra.\]
$\mu_{\hbar,\varphi}$ is a distribution on $T^{*}X$, which encodes
the localization properties of the state $\varphi$ in the phase space,
at the scale $\hbar$. Let us give an example. Using some local coordinate
chart near $x_{0}\in X$ and a function $\hbar^{2}\ll c(\hbar)\ll1$,
we can define a Gaussian wavepacket by \[
\varphi_{\hbar}(x)\defeq C_{\hbar}\,\chi(x)\,\exp\left\{ -\frac{|x-x_{0}|^{2}}{c(\hbar)}+i\frac{x\cdot\xi_{0}}{\hbar}\right\} .\]
Here $\chi$ is a smooth cutoff equal to unity near $x_{0}$, which
vanishes outside the coordinate chart, $C_{\hbar}$ is a normalization
factor. When $\hbar\ll1$, the distribution $\mu_{\hbar,\varphi_{\hbar}}$
associated with this wavefunction gets very peaked around the point
$(x_{0},\xi_{0})\in T^{*}X$. If we had used the quantization at the
scale $2\hbar$, the measure $\mu_{2\hbar,\varphi_{\hbar}}$ would
have been peaked around $(x_{0},\xi_{0}/2)$ instead.

Since the distribution $\mu_{\hbar,\varphi}$ is defined by duality
w.r.to the quantization $f\mapsto\Oph(f)$, it depends on the precise
quantization scheme $\Oph$. In the case $X=\IR^{d}$ and $\Oph$
is the Weyl quantization, the distribution $\mu_{\varphi,\hbar}$
is called the \emph{Wigner distribution} associated with the state
$\varphi$ and the scale $\hbar$. Fortunately, as shown by proposition
\ref{pro:quantization-equivalence}, this scheme-dependence is irrelevant
in the semiclassical limit.
\begin{cor}
For any $\varphi\in L^{2}$, consider the distributions $\mu_{\hbar,\varphi}^{1}$,
$\mu_{\hbar,\varphi}^{2}$ defined by duality with two $\hbar$-quantizations
$\Oph^{1}$, $\Oph^{2}$. Then, the following estimate holds in the
semiclassical limit, uniformly w.r.to $\varphi\in L^{2}$:\[
\forall f\in C_{c}^{\infty}(T^{*}X),\qquad\mu_{\hbar,\varphi}^{1}(f)-\mu_{\hbar,\varphi}^{2}(f)=\cO_{f}(\hbar\left\Vert \varphi\right\Vert ).\]
 
\end{cor}
Let $S\subset(0,1]$ be a set of scales. For a given family of $L^{2}$-normalized
states $(\varphi_{\hbar})_{\hbar\in S}$, we consider the sequence
of distributions $(\mu_{\hbar,\varphi_{\hbar}})_{\hbar\in S}$ on
$T^{*}X$. It is always possible to extract a subset of scales $S'\subset S$,
such that \[
\forall f\in C_{c}^{\infty}(T^{*}X),\qquad\mu_{\hbar,\varphi_{\hbar}}(f)\stackrel{S'\ni\hbar\to0}{\longrightarrow}\mu_{sc}(f),\]
with $\mu_{sc}$ a certain distribution on $T^{*}X$. One can show
that $\mu_{sc}$ is a Radon measure on $T^{*}X$ \cite[Thm 5.2]{EvZw09}.
From the above remarks, $\mu_{sc}$ does not depend on the precise
scheme of quantization. 
\begin{defn}
The measure $\mu_{sc}$ is called \emph{the }semiclassical measure
associated with the subsequence $(\varphi_{\hbar})_{\hbar\in S'}$.
It is also \emph{a }semiclassical measure associated with the sequence\emph{
}$(\varphi_{\hbar})_{\hbar\in S}$.
\end{defn}
From now on, we will assume that $\varphi_{\hbar}=\psi_{\hbar}$ is
a null eigenstate of the quantum Hamiltonian $P(\hbar)$ in §\ref{sub:Hamiltonians}:
we will then call $\mu_{sc}$ a semiclassical measure of the Hamiltonian
$P(\hbar)$. 
\begin{prop}
\label{pro:quasimode-basic}Any semiclassical measure associated with a sequence $(\psi_{\hbar})_{\hbar\to0}$ of
eigenstates of the Hamiltonian $P(\hbar)$
is a probability measure supported on the energy layer $\cE$, which
is invariant w.r.to the geodesic flow $g^{t}$ on $\cE$.\end{prop}
\begin{proof}
Possibly after extracting a subsequence, we assume that $\mu_{sc}$
is the semiclassical measure associated with a sequence of eigenstates
$(\psi_{\hbar})_{\hbar\in S}$. The support property of $\mu_{sc}$
comes from the fact that the operator $P(\hbar)$ is elliptic outside
$\cE$. As a result, for any $f\in C_{c}^{\infty}(T^{*}X)$ vanishing
near $\cE$, one can construct a symbol $g\in S^{-\infty,0}(T^{*}X)$
such that \[
\Oph(f)=\Oph(g)P(\hbar)+\cO_{L^{2}\to L^{2}}(\hbar^{\infty}).\]
Applying this equality to the eigenstates $\psi_{\hbar}$, we get
$\left\Vert \Oph(f)\psi_{\hbar}\right\Vert =\cO(\hbar^{\infty})$,
proving the support property of $\mu_{sc}$.

To prove the flow invariance, we need to compare the quantum time
evolution with the classical one. Denote by $U_{\hbar}^{t}=\exp\left\{ -i\frac{t\, P(\hbar)}{\hbar}\right\} $
the propagator generated by the Hamiltonian $P(\hbar)$: it solves
the time-dependent Schrödinger equation, and thus provides the quantum
evolution. Let us state \textbf{Egorov's theorem}, which is a rigorous
form of quantum-classical correspondence in terms of observables:\begin{equation}
\forall f\in C_{c}^{\infty}(T^{*}X),\;\forall t\in\IR,\qquad U_{\hbar}^{-t}\Op_{\hbar}(f)U_{\hbar}^{t}=\Op_{\hbar}(f\circ g^{t})+\cO_{f,t}(\hbar),\qquad\hbar\to0.\label{eq:Egorov}\end{equation}
Since $\psi_{\hbar}$ is an eigenstate of $U_{\hbar}^{t}$, we directly
get \[
\mu_{\hbar,\varphi_{\hbar}}(f)=\mu_{\hbar,\varphi_{\hbar}}(f\circ g^{t})+\cO_{f,t}\left(\hbar\right)\Longrightarrow\mu_{sc}(f)=\mu_{sc}(f\circ g^{t}).\]

\end{proof}
These properties of semiclassical measures naturally lead to the following
question:
\begin{quote}
Among all flow-invariant probability measures supported on $\cE$,
which ones appear as semiclassical measures associated with eigenstates
of $P(\hbar)$?
\end{quote}
To start answering this question, we will investigate the Kolmogorov-Sinai
entropy of semiclassical measures. We will show that, in the case
of an Anosov flow, the requirement of being a semiclassical measure
implies a nontrivial lower bound on the entropy.

\section{From classical to quantum entropies\label{sec:Entropies}}

\subsection{Entropies and pressures \cite{KatHas95}\label{sub:KSentropy}}

\subsubsection{Kolmogorov-Sinai entropy of an invariant measure}

In this paper we will deal with several types of entropies. All of
them are defined in terms of certain discrete probability distributions,
that is \emph{finite} sets of real numbers $\left\{ p_{i},\: i\in I\right\} $
satisfying \[
p_{i}\in[0,1],\quad\sum_{i\in I}p_{i}=1\,.\]
The entropy associated with such a set is the real number 
\begin{equation}
H(\left\{ p_{i}\right\} )=\sum_{i\in I}\eta(p_{i}),\qquad\mbox{where}\qquad\eta(s)\defeq-s\log s,\quad s\in[0,1].\label{eq:entropy--1}\end{equation}

Our first example is the entropy $H(\mu,\cP)$ associated with a $g^{t}$-invariant
probability measure $\mu$ on the energy shell $\cE$ and a finite
measurable partition $\cP=(E_{1},\ldots,E_{K})$ of $\cE$. That entropy
is given by \begin{equation}
H(\mu,\cP)=H(\left\{ \mu(E_{k})\right\} )=\sum_{k=1}^{K}\eta(\mu(E_{k})).\label{eq:entropy-0}\end{equation}
One can then use the flow $g^{t}$ in order to \emph{refine} the partition
$\cP$. For each integer $n\geq1$ we define the $n$-th refinement
$\cP^{\vee n}=[\cP]_{0}^{n-1}$ as the partition composed of the sets
\[
E_{\bal}\defeq g^{-n+1}E_{\alpha_{n-1}}\cap\cdots\cap g^{-1}E_{\alpha_{1}}\cap E_{\alpha_{0}},\]
where $\bal=\alpha_{0}\cdots\alpha_{n-1}$ can be any sequence of
length $n$ with symbols $\alpha_{i}\in\left\{ 1,\ldots,K\right\} $.
In general many of the sets $E_{\bal}$ may be empty, but we will
nonetheless sum over all sequences of a given length $n$. More generally,
for any $m\in\Z,\: n\geq1$, we consider the partition $[\cP]_{m}^{m+n-1}$
made of the sets

\[
g^{-m}E_{\bal}=g^{-m-n+1}E_{\alpha_{n-1}}\cap\cdots\cap g^{-m}E_{\alpha_{0}},\qquad|\bal|=n.\]
From this refined partition we obtain the entropy $H(\mu,[\cP]_{m}^{m+n-1})=H_{m}^{m+n-1}(\mu)$. 

From the concavity of the logarithm, one easily gets\begin{equation}
\forall m,n\geq1,\qquad H(\mu,[\cP]_{0}^{n+m-1})\leq H(\mu,[\cP]_{0}^{n-1})+H(\mu,[\cP]_{n}^{n+m-1}).\label{eq:subadditivity}\end{equation}
If the measure $\mu$ is $g^{t}$-invariant, this has for consequence
the \emph{subadditivity property}: \begin{equation}
H(\mu,\cP^{\vee(n+m)})\leq H(\mu,\cP^{\vee n})+H(\mu,\cP^{\vee m}).\label{eq:subadd-2}\end{equation}
It thus makes sense to consider the limit\[
H_{KS}(\mu,\cP)\defeq\lim_{n\to\infty}\frac{1}{n}H(\mu,\cP^{\vee n})=\lim_{n\to\infty}\frac{1}{2n}H(\mu,[\cP]_{-n+1}^{n}),\]
the \emph{Kolmogorov-Sinai entropy} of the invariant measure $\mu$,
associated with the partition $\cP$. The KS entropy \emph{per se}
is defined by maximizing over the initial (finite) partition $\cP$:
\[
H_{KS}(\mu)\defeq\sup_{\cP}H_{KS}(\mu,\cP).\]
For an Anosov flow, this supremum is actually reached as soon as the
partition $\cP$ has a sufficiently small diameter (that is, its elements
$E_{k}$ have uniformly small diameters).

\subsubsection{Pressures associated with invariant measures\label{sub:Pressures}}

Let us come back to our probability distribution $\left\{ p_{i},\: i\in I\right\} $.
We may associate to it a set of weights, that is of positive real
numbers $\left\{ w_{i}>0,\: i\in I\right\} $, making up a weighted
probability distribution. The pressure $p(\left\{ p_{i}\right\} ,\left\{ w_{i}\right\} )$
associated with this weighted distribution is the real number%
\footnote{The factor $-2$ appearing in front of the second term is convenient
for our future aims.%
} \[
p(\left\{ p_{i}\right\} ,\left\{ w_{i}\right\} )\defeq-\sum_{i}p_{i}\log(w_{i}^{2}p_{i})=H(\left\{ p_{i}\right\} )-2\sum_{i\in I}p_{i}\,\log w_{i}.\]
For instance, in the case of a flow-invariant measure on $\cE$ and
a partition $\cP$, we can select weights $w_{k}$ on each component
$E_{k}$, and define the pressure \[
p(\mu,\cP,w)\defeq H(\mu,\cP)-2\sum_{k}\mu(E_{k})\,\log w_{k}.\]
We want to refine this pressure using the flow. The weights corresponding
to the $n$-th refinement can be simply defined as \[
w_{\bal}=\prod_{j=0}^{n-1}w_{\alpha_{j}},\quad|\bal|=n.\]
The refined pressure is denoted by $p_{0}^{n-1}(\mu,\cP,w)$. From
the subadditivity of the entropies (\ref{eq:subadd-pressure}) one
easily draws the subadditivity of the pressures:\begin{equation}
p_{0}^{n+m-1}(\mu,\cP,w)\leq p_{0}^{n-1}(\mu,\cP,w)+p_{0}^{m-1}(\mu,\cP,w).\label{eq:subadd-pressure}\end{equation}

\subsubsection{Smoothed partitions near $\cE$}

The definition of $H(\mu,\cP)$ can be expressed in terms of characteristic
functions over the partition $\cP$. Indeed, if $\bbbone_{k}$ is
the characteristic function on $E_{k}$, then the function \begin{equation}
\bbbone_{\bal}=\left(\bbbone_{\alpha_{n-1}}\circ g^{n-1}\right)\times\cdots\times\left(\bbbone_{\alpha_{1}}\circ g\right)\times\bbbone_{\alpha_{0}}\label{eq:refined-2}\end{equation}
 is the characteristic function of $E_{\bal}$. In formula (\ref{eq:entropy-0})
we can then replace $\mu(E_{k})$ by \[
\mu(\bbbone_{k})\defeq\int_{\cE}\bbbone_{k}\, d\mu.\]
Let us assume that the invariant measure $\mu$ does not charge the
boundary of the $E_{k}$ (this is always possible by slightly shifting
the boundaries of the $E_{k}$). Then, for any $\eps>0$, we can approximate
the characteristic function $\bbbone_{k}$ by a smooth function $\pi_{k}\in C_{c}^{\infty}(\cE_{\eps},[0,1])$
supported in a small neighbourhood $\tilde{E}_{k}$ of $E_{k}$, and
such that these $K$ functions form a \emph{smooth partition of unity
near }$\cE$:\begin{equation}
\sum_{k=1}^{K}\pi_{k}(\rho)=\chi_{\eps/2}(\rho),\qquad\supp\chi_{\eps/2}\subset\cE_{\eps},\qquad\chi_{\eps/2}=1\quad\mbox{near}\quad\cE_{\eps/2}.\label{eq:smooth-partition}\end{equation}
One can extend the definition of the entropy to the smooth partition\emph{
}$\cP_{sm}=\left\{ \pi_{k}\right\} _{k=1,\ldots,K}$ and its refinements
through the flow. From the assumption $\mu(\partial\cP)=0$, for any
$\eps'>0$ and any $n\geq1$ we can choose $\cP_{sm}$ such that 
\[
\left|H(\mu,\cP_{sm}^{\vee n})-H(\mu,\cP^{\vee n})\right|\leq\eps'.
\]

To prove a lower bound on the entropy $H_{KS}(\mu)$ therefore amounts
to proving a lower bound on $\frac{1}{n}H(\mu,\cP_{sm}^{\vee n})$,
uniform w.r.to $n\geq1$ and the smoothing $\cP_{sm}$ of $\cP$.
The advantage of using a smoothed partition $\cP_{sm}$ is that it
is fit for quantization.

\subsection{Quantum partitions of unity\label{sub:quantum-partition}}

\subsubsection{Definition}

From the smoothed partition $\cP_{sm}=\left\{ \pi_{k}\right\} _{k=1,\ldots,K}$
we form a \emph{quantum partition of unity} $\cP_{sm,q}=\left\{ \Pi_{k}=\Oph(\tilde{\pi}_{k})\right\} _{k=1,\ldots,K}$,
where $\Oph$ is the $\hbar$-quantization (\ref{eq:Quantization-mfold}),
and the symbols $\tilde{\pi}_{k}\in S^{-\infty,0}(T^{*}X)$ satisfy
the following properties:
\begin{enumerate}
\item for each $k$, the symbol $\tilde{\pi}_{k}$ is real, supported on $\tilde{E}_{k}$,
and admits $\sqrt{\pi_{k}}$ as principal symbol. The operator $\Pi_{k}$
is thus selfadjoint.
\item the family $\cP_{sm,q}=\left\{ \Pi_{k}\right\} _{k=1,\ldots,K}$ is
a quantum partition of unity microlocally near $\cE$: \begin{equation}
\sum_{k=1}^{K}\Pi_{k}^{2}=\Oph(\tilde{\chi}_{\eps/2})+\cO(\hbar^{\infty}),\label{eq:Quantum Partition}\end{equation}
where $\tilde{\chi}_{\eps/2}\in S^{-\infty,0}(T^{*}X)$ satisfies
\[
\widetilde{\chi}_{\eps/2}(\rho)\!\equiv\!1\ \text{near}\ \cE_{\eps/2},\qquad 
\supp\widetilde{\chi}_{\eps/2}\!\subset\!\cE_{\eps},\qquad \lVert \Oph(\widetilde{\chi}_{\eps/2})\rVert=1+\cO(\hbar^{\infty})\,.
\]
Notice that $\tilde{\chi}_{\eps/2}$ has for principal symbol $\chi_{\eps/2}$
of (\ref{eq:smooth-partition}).\end{enumerate}
\begin{rem*}
Had we simply taken $\Pi_{k}=\Oph(\sqrt{\pi_{k}})$, the above properties
would hold only up to remainders $\cO(\hbar)$. By iteratively adjusting
the higher-order symbols in $\tilde{\pi}_{k}$ (and $\tilde{\chi}_{\eps/2}$),
we can enforce these properties to any order in $\hbar$. 
\end{rem*}

\subsubsection{Refined quantum partitions}

In the classical framework, the $n$-refinement of the partition $\cP_{sm}=\left\{ \pi_{k}\right\} _{k=1,\ldots,K}$
was obtained by considering the products $\pi_{\alpha_{n-1}}\circ g^{n-1}\times\cdots\times\pi_{\alpha_{0}}$,
for all sequences $\bal$ of length $n$. Egorov's theorem shows that
the quantum observable $\Oph\left(\pi_{\alpha_{j}}\circ g^{j}\right)$
resembles the quantum evolution $U^{-j}\Oph(\pi_{\alpha_{j}})U^{j}$,
where $U=U_{\hbar}=e^{-iP(\hbar)/\hbar}$ is the Schrödinger propagator
(at time unity). For this reason, we define as follows the elements
of the $n$-refined quantum projection:\begin{equation}
\Pi_{\bal}\defeq U^{-n+1}\Pi_{\alpha_{n-1}}U\cdots U\Pi_{\alpha_{2}}U\Pi_{\alpha_{1}}U\Pi_{\alpha_{0}},\qquad\bal=\alpha_{0}\cdots\alpha_{n-1}.\label{eq:quantum-refined}\end{equation}
 We first need to check that these operators still make up a quantum
partition of unity near $\cE$.
\begin{prop}
\label{pro:qu-part-local}Take $n_{\max}=[C_{\delta}|\log\hbar|]$
as in section \ref{sub:Sharp-energy-cutoffs}. Then, for each $1\leq n\leq n_{\max}$,
the family of operators $\cP_{sm,q}^{\vee n}=\left\{ \Pi_{\bal},\quad|\bal|=n\right\} $
forms a quantum partition of unity microlocally near $\cE$, in the
following sense. For any symbol $\chi\in S^{-\infty,0}(T^{*}X)$ supported
inside $\cE_{\eps/2}$, we have 
\begin{equation}\label{eq:Quant-partition-n}
\begin{gathered}
\forall n\leq n_{\max},\qquad\sum_{\alpha_{0},\ldots,\alpha_{n-1}}\Pi_{\bal}^{*}\Pi_{\bal}=S_{n},\qquad\left\Vert S_{n}\right\Vert =1+\cO(\hbar^{\infty}),\\
\left(Id-S_{n}\right)\,\Oph(\chi)=\cO(\hbar^{\infty}).
\end{gathered}
\end{equation}
\end{prop}
\begin{proof}
The statement is obvious in the case $\left\{ \Pi_{k}\right\} _{k=1,\ldots,K}$
forms a full resolution of identity (that is, if the left hand side
in (\ref{eq:Quantum Partition}) is equal to the identity modulo $\cO(\hbar^{\infty})$),
as was the case in \cite{Ana08,AnaKoNo06,AnaNo07-2} and will be the
case in §\ref{sec:Anosov-maps}. One can then sum over $\sum_{\alpha_{n-1}}\Pi_{\alpha_{n-1}}^{2}=Id+\cO(\hbar^{\infty})$,
then over $\alpha_{n-2}$, etc, to finally obtain $S_{n}=Id+\cO(\hbar^{\infty})$. 

In the case of a microlocal partition near $\cE$, the sum over the
index $\alpha_{n-1}$ leads to a product $\Pi_{\alpha_{n-2}}\Oph(\tilde{\chi}_{\eps/2})\Pi_{\alpha_{n-2}}$,
where $\tilde{\chi}_{\eps/2}$ is the symbol appearing in (\ref{eq:Quantum Partition}).
To {}``absorb'' the factor $\Oph(\tilde{\chi}_{\eps/2})$, we will
insert intermediate cutoffs at each time. We recall that $\chi\prec\bbbone_{\cE_{\eps/2}}\prec\tilde{\chi}_{\eps/2}$
. We consider a sequence of cutoffs $(\chi_{j}\circ p_{0})_{1\leq j\leq n_{\max}}$
such that $\bbbone_{\cE_{\eps/2}}\prec\chi_{1}\circ p_{0}\prec\chi_{2}\circ p_{0}\cdots\prec\chi_{n_{max}}\circ p_{0}\prec\tilde{\chi}_{\eps/2}$
. Since $n_{\max}\sim|\log\hbar|$, the $\chi_{j}$ will necessarily
depend on $\hbar$, their derivatives growing like $\hbar^{-\nu}$
for some small $\nu>0$, so that $\chi^{j}\circ p_{0}\in S_{\nu}^{-\infty,0}$.
The calculus in $S_{\nu}^{-\infty,0}$ and $\chi_{j}\prec\chi_{j+1}$
show that, for any $1\leq k\leq K$:

\begin{equation}
\forall 1\leq j\leq n_{\max}-1,\quad\left(Id-\chi_{j+1}(P(\hbar))\right)\,\Pi_{k}\,\chi_{j}(P(\hbar))=\cO(\hbar^{\infty}).\label{eq:embedding2}\end{equation}
This implies that we can indeed insert intermediate cutoffs in $\Pi_{\bal}$
with no harm: 
{\multlinegap0pt\begin{multline*}
\Pi_{\bal}\Oph(\chi)=
=U^{-n+1}\Pi_{\alpha_{n-1}}U\Pi_{\alpha_{n-2}}U\chi_{n-2}(P(\hbar))\cdots U\chi_{2}(P(\hbar))\Pi_{\alpha_{1}}\\
\times U\chi_{1}(P(\hbar))\Pi_{\alpha_{0}}\Oph(\chi)+\cO(\hbar^{\infty}).
\end{multline*}}
We also have\begin{equation}
\forall j,\quad\left(Id-\Oph(\tilde{\chi}_{\eps/2})\right)\,\Pi_{k}\,\chi_{j}(P(\hbar))=\cO(\hbar^{\infty}).\label{eq:embedding3}\end{equation}
This equation, and the fact that $\chi_{j}(P(\hbar))$ commutes with
the propagator $U$, results in
\begin{multline*}
\smash[b]{\sum_{\alpha_{n-2}}} U^*\Pi_{\alpha_{n-2}}\Oph(\widetilde{\chi}_{\eps/2})\Pi_{\alpha_{n-2}}U \chi_{n-2}\bigl(P(\hbar)\bigr) \\
\begin{aligned}
& =\sum_{\alpha_{n-2}}U^*\Pi_{\alpha_{n-2}}^2\chi_{n-2}\bigl(P(\hbar)\bigr)U+\cO(\hbar^{\infty})\\
 & =U^*\Oph(\widetilde{\chi}_{\eps/2})\chi_{n-2}\bigl(P(\hbar)\bigr)U+\cO(\hbar^{\infty})\\
 & =U^*\chi_{n-2}\bigl(P(\hbar)\bigr)U+\cO(\hbar^{\infty})\\
 & =\chi_{n-2}\bigl(P(\hbar)\bigr)+\cO(\hbar^{\infty}).
 \end{aligned}
 \end{multline*}
Using (\ref{eq:embedding2}) at each step, the summation over $\alpha_{n-3},\alpha_{n-4},\ldots$
finally brings us to \[
\sum_{\bal}\Pi_{\bal}^{*}\Pi_{\bal}\Oph(\chi)=\sum_{\alpha_{0}}\Pi_{\alpha_{0}}\chi_{1}(P(\hbar))\Pi_{\alpha_{0}}\Oph(\chi)+\cO(h^{\infty}).\]
The equation $\left(Id-\chi_{1}(P(\hbar))\right)\Pi_{\alpha_{0}}\Oph(\chi)=\cO(\hbar^{\infty})$
leads to the proof.
\end{proof}

\subsection{From the refined operators $\Pi_{\bal}$ to a quantum symbolic measure\label{sub:symbolic-measure}}

Let us turn back to the sequence of eigenstates $(\psi_{\hbar})_{\hbar\to0}$
associated with the semiclassical measure $\mu_{sc}$. The proof of
Prop. \ref{pro:quasimode-basic} shows that for any cutoff $\chi\in S^{-\infty,0}$
, $\chi\equiv1$ near $\cE$, we have \[
\Oph(\chi)\psi_{\hbar}=\psi_{\hbar}+\cO(\hbar^{\infty}).\]
As a result, for any $1\leq n\leq n_{\max}$ we have $\sum_{|\bal|=n}\left\Vert \Pi_{\bal}\psi_{\hbar}\right\Vert ^{2}=1+\cO(\hbar^{\infty})$.
Therefore, modulo an $\cO(\hbar^{\infty})$ error, the set $\left\{ \left\Vert \Pi_{\bal}\psi_{\hbar}\right\Vert ^{2},\:|\bal|=n\right\} $
forms a discrete probability distribution. The proof shows that these
weights also satisfy a compatibility condition (up to a negligible
error): \begin{equation}
\forall n\leq n_{\max},\;\forall\bal=\alpha_{0}\cdots\alpha_{n-1},\qquad\left\Vert \Pi_{\alpha_{0}\cdots\alpha_{n-2}}\psi_{\hbar}\right\Vert ^{2}=\sum_{\alpha_{n-1}}\left\Vert \Pi_{\alpha_{0}\cdots\alpha_{n-1}}\psi_{\hbar}\right\Vert ^{2}+\cO(\hbar^{\infty}).\label{eq:compatibility}\end{equation}
If we forget the errors $\cO(\hbar^{\infty})$, we can interpret the
weights $\left\Vert \Pi_{\bal}\psi_{\hbar}\right\Vert ^{2}$ in terms
of a certain probability measure $\mu_{\hbar}$ on the symbolic space
$\Sigma=\left\{ 1,\ldots,K\right\} ^{\IZ}$. Namely, each $\left\Vert \Pi_{\bal}\psi_{\hbar}\right\Vert ^{2}$
corresponds to the weight of that measure on the \emph{cylinder} $[\cdot\alpha_{0}\alpha_{1}\cdots\alpha_{n-1}]$:\begin{equation}
\mu_{\hbar}([\cdot\bal])\defeq\left\Vert \Pi_{\bal}\psi_{\hbar}\right\Vert ^{2}.\label{eq:mu^u-definition}\end{equation}
Formally, $\mu_{\hbar}$ can be defined as an equivalence class of
$\hbar$-dependent probability measures taking values on cylinders
of lengths $n\leq n_{\max}=C_{\delta}|\log\hbar|$, the equivalence
relation consisting in equality up to errors $\cO(\hbar^{\infty})$.
We will call such a measure a \textbf{symbolic measure}.

The defining property $\Pi_{k}=\Oph(\tilde{\pi}_{k})$ shows that
each element $\mu_{\hbar}([\cdot\alpha_{0}])=\left\Vert \Pi_{\alpha_{0}}\psi_{\hbar}\right\Vert ^{2}$
approximately represents the microlocal weight of the state $\psi_{\hbar}$
inside the element $E_{k}$ of the partition. Further on, for any
fixed $n\geq1$, Egorov's theorem (\ref{eq:Egorov}) and the composition
rule (\ref{eq:composition}) show that the refined operators $\Pi_{\bal}$
are still {}``good'' pseudodifferential operators: \[
\Pi_{\bal}=\Oph(\pi_{\bal})+\cO_{n}(\hbar),\]
where $\cP_{sm}^{\vee n}=\left\{ \pi_{\bal}\right\} $ is the $n$-refinement
of the smooth partition $\cP_{sm}$, as in (\ref{eq:refined-2}).
As a result, $\mu_{\hbar}([\cdot\bal])$ approximately represents
the weight of $\psi_{\hbar}$ inside the refined partition
element $E_{\bal}$. 

From the assumption on the sequence $(\psi_{\hbar})$, the symbolic
measure $\mu_{\hbar}$ is obviously related with the semiclassical
measure $\mu_{sc}$: for any fixed $n\geq1$ we have \begin{equation}
\forall\bal,\:|\bal|=n,\qquad\mu_{\hbar}([\cdot\bal])\hto0\mu_{sc}(\pi_{\bal}).\label{eq:weight-quant-class}\end{equation}

This limit assumes that the {}``time'' $n$ is fixed when taking
$\hbar\to0$. For our purposes, it will be crucial to extend the analysis
to times $n$ of logarithmic order in $\hbar$. Before doing so, let
us give a crude description of the sets $E_{\bal}$ when $ $$n=|\bal|$
grows. Let us assume that the elements $E_{k}$ are approximately
{}``isotropic'' w.r.to the stable and unstable directions. The inverse
flow $g^{-t}$ has the effect to compress along the unstable directions,
and expand along the stable ones. As a result, the set $E_{\alpha_{0}\alpha_{1}}=g^{-1}E_{\alpha_{1}}\cap E_{\alpha_{0}}$
will be narrower than $E_{\alpha_{0}}$ along the unstable direction,
but keep approximately the same size in the stable one. Iterating
this procedure, the set $E_{\alpha_{0}\cdots\alpha_{n-1}}$ will become
very anisotropic for large $n$: its size along the stable directions
will remain comparable with that of $E_{\alpha_{0}}$, while its {}``unstable
volume'' will be diminished by a factor $J_{n}^{u}(\alpha_{0}\cdots\alpha_{n-1})^{-1}$,
where we use the \textbf{coarse-grained unstable Jacobian} \begin{equation}
J_{n}^{u}(\alpha_{0}\cdots\alpha_{n-1})\defeq\prod_{i=0}^{n-1}J^{u}(\alpha_{i}),\qquad J^{u}(k)\defeq\min_{\rho\in E_{k}}J^{u}(\rho).\label{eq:coarse-grained Jacobian}\end{equation}
The factor $J_{n}^{u}(\bal)^{-1}$ decreases exponentially with $n$
according to the minimal $(d-1)$-dimensional unstable expansion rate
$\Lambda_{\min}^{u}$: \begin{equation}
\forall n,\;\forall\bal,\,|\bal|=n,\qquad J_{n}^{u}(\bal)^{-1}\leq Ce^{-n(\Lambda_{\min}^{u}-\eps)},\label{eq:Jacob-bound}\end{equation}
so the sets $E_{\bal}$ become very thin along the unstable direction.
This anisotropy is as well visible on the refined smooth functions
$\pi_{\bal}$ or the refined symbols $\tilde{\pi}_{\bal}$.

\subsection{Egorov theorem up to the Ehrenfest time\label{sub:Egorov-Ehrenfest}}

The Egorov theorem (\ref{eq:Egorov}) can be extended up to times
$t\sim C|\log\hbar|$, provided the constant $C$ is not too large.
The breakdown occurs when the classically evolved function $f\circ g^{t}$
shows fluctuations of size unity across a distance $\sim h^{1/2}$:
such a function is no more a {}``nice quantizable observable'' (see
$\S$\ref{sub:exotic-symbols} ). 

Let us start from a symbol $f\in S^{-\infty,0}(T^{*}X)$ supported
on the energy layer $\cE_{\eps}$. We have called $\lambda_{\max}$
the maximal expansion rate of the flow on $\cE$. Assume that $\lambda_{\eps}=\lambda_{\max}+\cO(\eps)$
is larger than the maximal expansion rate on $\cE_{\eps}$. It implies
that the derivatives of the flow are controlled as follows:\[
\forall t\in\IR,\quad\forall\rho\in\cE_{\eps},\qquad\left\Vert \partial^{\alpha}g^{t}(\rho)\right\Vert \leq C_{\alpha}\, e^{\lambda_{\eps}|\alpha t|}.\]
As a result, for any symbol $f\in S^{-\infty,0}$ supported inside
$\cE_{\eps}$, its classical evolution $f_{t}=f\circ g^{t}$ satisfy\[
\forall t\in\IR,\quad\forall\rho\in\cE_{\eps},\qquad\left\Vert \partial^{\alpha}f_{t}(\rho)\right\Vert \leq C_{f,\alpha}\, e^{\lambda_{\eps}|\alpha t|}.\]
For $t\sim C|\log\hbar|$ the right hand sides become of order $\hbar^{-C\lambda_{\eps}|\alpha|}$.
Therefore, if we want $f_{t}$ to belong to a reasonable symbol class
(see section \ref{sub:exotic-symbols}), we must restrict the values
of $C$. Let us define the time\[
T_{\eps,\hbar}\defeq\frac{(1-\eps)|\log\hbar|}{2\lambda_{\eps}},\]
which is about half of what is generally called the \emph{Ehrenfest
time} $T_{E}=\frac{|\log\hbar|}{\lambda_{\max}}$. Take any $\nu\in[\frac{1-\eps}{2},\frac{1}{2})$.
The above estimates show that for any choice of sequence $(t(\hbar))_{\hbar\to0}$
satisfying $|t(\hbar)|\leq T_{\eps,\hbar}$, the family of functions
$\left(f_{t(\hbar)}\right)_{\hbar\to0}$ belongs to the class $S_{\nu}^{-\infty,0}(T^{*}X)$
defined in (\ref{eq:S_nu}). In other words, any seminorm of that
class is uniformly bounded over the set $\left\{ f_{t},\;|t|\leq T_{\eps,\hbar},\;\hbar\in(0,1]\right\} $.
It is then not surprising that Egorov's theorem holds up to the
time $T_{\eps,\hbar}$.
\begin{prop}
\cite[Prop.5.1]{AnaNo07-2}Fix $\eps>0$ and $\nu\in(\frac{1-\eps}{2},\frac{1}{2})$.
Take $f\in S^{-\infty,0}$ supported inside $\cE_{\eps}$. Then, for
any $\hbar\in(0,1]$ and any time $t=t(\hbar)$ in the range $|t|\leq T_{\eps,\hbar}$,
we have 
\begin{multline}\label{eq:Egorov-long}
U^{-t}\Oph(f)U^t=\Oph(\tilde f_t)+\cO(h^{\infty}),\\
\text{with $\tilde f_t-f_t\in S_{\nu}^{-\infty,-(1+\eps)/2}$, $f_t=f\circ g^t\in S_{\nu}^{-\infty,0}$.}
\end{multline}
\end{prop}
\begin{proof}
This proposition was essentially proved in \cite{BouzRob02} in the
case of symbols on $T^{*}\IR^{d}$ driven by some (appropriate) Hamiltonian
flow. In that paper, the $\hbar$-expansion of the symbol $\tilde{f}_{t}$
was explicitly computed up to any fixed order $\hbar^{L}$, and the
$L^{2}$ norm of the remainder was estimated. In \cite[Sec. 5.2]{AnaNo07-2}
we used the fact that \[
U^{-t}\Oph(f)U^{t}-\Oph(f_{t})=\int_{0}^{t}ds\, U^{-s}\,\Diff f_{t-s}\, U^{s},\]
where \[
\Diff f_{s}=\frac{i}{h}[P(\hbar),\Oph(f_{s})]-\Oph\left(\left\{ p,f_{s}\right\} \right)\in\Psi_{\nu}^{-\infty,-\frac{1+\eps}{2}},\]
uniformly for $|t|\leq T_{\eps,\hbar}$. The Calderon-Vaillancourt
theorem on $\Psi_{\nu}^{-\infty,-\frac{1+\eps}{2}}$ then implies
that \begin{equation}
\left\Vert U^{-t}\Oph(f)U^{t}-\Oph(f_{t})\right\Vert \leq C\,|t|\, h^{\frac{1+\eps}{2}}.\label{eqlongtime-bound}\end{equation}
In order to prove that $\tilde{f}_{t}\in S_{\nu}^{-\infty,0}$ one
can proceed as in \cite{BouzRob02}, that is compute the $\hbar$
expansion of $\tilde{f}_{t}$ order by order, taking into account
that the quantization is performed on the manifold $X$, so that higher-order
terms also depend on the various choices of local charts and cutoffs.
We will not do so in any detail here, since we will mostly use the
inequality (\ref{eqlongtime-bound}).
\end{proof}
We will apply this proposition to the operators $U^{-j}\Pi_{k}U^{j}$:
in the range $|j|\leq T_{\eps,\hbar}$ they are still pseudodifferential
operators in some class $S_{\nu}^{-\infty,0}$. The products of these
operators can also be analyzed:
\begin{prop}
\label{pro:P_alpha-symbol}Take any $1>\eps>0$ and $\nu\in[\frac{1-\eps}{2},\frac{1}{2})$.
Then the family of symbols $\left\{ \tilde{\pi}_{\bal},\;|\bal|\leq T_{\eps,\hbar}\right\} $
belongs to a bounded set in the class $S_{\nu}^{-\infty,0}$. Furthermore,
the product operators $\Pi_{\bal}$ satisfy $\Pi_{\bal}-\Oph(\tilde{\pi}_{\bal})\in\Psi_{\nu}^{-\infty,-\frac{1+\eps}{2}}$.\end{prop}
\begin{proof}
A similar result was proved in \cite[Thm 7.1]{Riv08}. We already
know that the symbols $\tilde{\pi}_{\alpha_{j}}\circ g^{j}$ composing
$\tilde{\pi}_{\bal}$ belong to the class $\Psi_{\nu}^{-\infty,0}$.
Any finite product of those symbols also remains in that class. We
need to check that the symbol $\tilde{\pi}_{\alpha_{0}\cdots\alpha_{j}}$
remains uniformly bounded in the class when increasing $j$ until
$T_{\eps,\hbar}$. 

We start by applying Egorov's theorem to the operator $\Pi_{\alpha_{n-1}}$,
then multiply by $\Pi_{\alpha_{n-2}}$: \[
U^{-1}\Pi_{\alpha_{n-1}}U\Pi_{\alpha_{n-2}}=\Oph(\tilde{\pi}_{\alpha_{n-1}}\circ g\times\tilde{\pi}_{\alpha_{n-2}})+R_{2},\quad R_{2}\in\Psi^{-\infty,-1}.\]
The function $\tilde{\pi}_{\alpha_{n-2}\alpha_{n-1}}\defeq\tilde{\pi}_{\alpha_{n-1}}\circ g\times\tilde{\pi}_{\alpha_{n-2}}$
is supported in a {}``rectangle'' and satisfies $\left\Vert \partial^{\beta}\tilde{\pi}_{\alpha_{n-2}\alpha_{n-1}}\right\Vert \leq C_{\beta}\, e^{\lambda_{\eps}|\beta|}$.
Applying the same procedure (evolution and multiplication), we construct
a sequence of symbols \[
\tilde{\pi}_{\alpha_{n-j}\cdots\alpha_{n-1}}\defeq\tilde{\pi}_{\alpha_{n-j+1}\cdots\alpha_{n-1}}\circ g\times\tilde{\pi}_{\alpha_{n-j}}\]
and operators \[
U^{-j}\Pi_{\alpha_{n-1}}U\Pi_{\alpha_{n-2}}U\cdots\Pi_{\alpha_{n-j}}.\]
The symbols are supported in small rectangles, similar with the elements
$E_{\bal}$ of the refined partition $\cP^{\vee j}$. One iteratively
shows that \[
\left\Vert \partial^{\beta}\tilde{\pi}_{\alpha_{n-j}\cdots\alpha_{n-1}}\right\Vert \leq C_{\beta}\, e^{\lambda_{\eps}|\beta|j},\]
with constants $C_{\beta}$ uniform w.r.to $j$. Therefore, as long
as $j\leq T_{\eps,\hbar}$, the symbol $\tilde{\pi}_{\alpha_{n-j+1}\cdots\alpha_{n-1}}\in S_{\nu_{j}}^{-\infty,0}$
(with uniform constants), where $\nu_{j}=\frac{\lambda_{\eps}j}{\log(1/\hbar)}$.
At the same time, $\tilde{\pi}_{\alpha_{n-j}}\in S_{0}^{-\infty,0}$.
As a result, \[
U^{-1}\Oph(\tilde{\pi}_{\alpha_{n-j+1}\cdots\alpha_{n-1}})U\Oph(\tilde{\pi}_{\alpha_{n-j}})=\Oph(\tilde{\pi}_{\alpha_{n-j}\cdots\alpha_{n-1}})+R_{j},\]
and the remainder $R_{j}\in\Psi_{\nu_{j}}^{-\infty,-1+\nu_{j}}$ satisfies
\[
\left\Vert R_{j}\right\Vert _{L^{2}\circlearrowleft}\leq C\,\hbar^{1-\nu_{j}},\quad j=2,\ldots,n,\]
 with a uniform constant $C$. The sum of all remainders thus satisfies
\[
\sum_{j=2}^{n}R_{j}\in\Psi_{\nu_{n}}^{-\infty,-1+\nu_{n}},\qquad\left\Vert \sum_{j=2}^{n}R_{j}\right\Vert \leq\sum_{j=2}^{n}C\,\hbar^{1-\nu_{j}}\leq\tilde{C}\:\hbar^{1-\nu_{n}}.\]
\end{proof}
\begin{cor}
Take any $1>\eps>0$ and $\nu\in[\frac{1-\eps}{2},\frac{1}{2})$.
Let $\bal,\bbeta$ be two sequences of length $n\leq T_{\eps,\hbar}$.
Then the symbols $\tilde{\pi}_{\bal},\,\tilde{\pi}_{\bbeta}\circ g^{-n}$
belong to $S_{\nu}^{-\infty,0}$, and so does their product. The operator
\[
\Pi_{\bbeta\cdot\bal}\defeq\Pi_{\bal}U^{n}\Pi_{\bbeta}U^{-n}=U^{-n+1}\Pi_{\alpha_{n-1}}U\Pi_{\alpha_{n-2}}U\cdots\Pi_{\alpha_{0}}U\Pi_{\beta_{n-1}}U\cdots U\Pi_{\beta_{0}}U^{-n}\]
belongs to $\Psi_{\nu}^{-\infty,0}$, and satisfies \[
\Pi_{\bbeta\cdot\bal}=\Oph(\tilde{\pi}_{\bal}\times\tilde{\pi}_{\bbeta}\circ g^{-n})+\Psi_{\nu}^{-\infty,-1+2\nu}.\]

\end{cor}
These results show that, for times $n\leq T_{\eps,\hbar}$, the operators
$\Pi_{\bal}$ (resp. $\Pi_{\bbeta\cdot\bal}$) are {}``quasiprojectors''
on refined rectangles $E_{\bal}\in\cP^{\vee n}$ (resp. in the rectangles
$E_{\bal}\cap g^{n}(E_{\bbeta})$ of the {}``isotropic'' refined
partition $\cP_{-n}^{n-1}$). Using the fact that $\Pi_{\bbeta\cdot\bal}=U^{n}\Pi_{\bbeta\bal}U^{-n}$,
we also draw the
\begin{cor}
Take $\eps,\,\nu$ as above. Then, for any sequence $\bal$ of length
$|\bal|\leq2T_{\eps,\hbar}$, the operator norm $\left\Vert \Pi_{\bal}\right\Vert =\norm{\tilde{\pi}_{\bal}}_{\infty}+\cO(\hbar^{1-2\nu})$,
which can be close to unity.
\end{cor}

\subsection{Hyperbolic dispersive estimates}

We will now consider operators $\Pi_{\bal}$ for sequences $\bal$
\emph{longer than} $2T_{\eps,\hbar}$. We recall that $J_{n}^{u}(\bal)$
is the coarse-grained unstable Jacobian along orbits following the
path $\bal$ (see (\ref{eq:coarse-grained Jacobian})). Given some
small $\delta>0$, we have constructed in section \ref{sub:Sharp-energy-cutoffs}
cutoffs $\chi^{(m)}$ supported on intervals of lengths $2e^{m\delta}\hbar^{1-\delta}$,
from which we built up sharp energy cutoffs. Our major dynamical result
is the following dispersive estimate \cite{AnaNo07-2}. We provide
its proof in §\ref{sec:Hyperbolic dispersive estimate}.
\begin{prop}
\label{pro:hyperb-dispers}Choose $\delta>0$ small, leading to the
constant $C_{\delta}$ of $\S$\ref{sub:Sharp-energy-cutoffs}. Then, there exists $\hbar_\delta>0$ and $C>0$ such that,
for any $0<\hbar\leq h_{\delta}$, any integers $n,m\in [0, C_{\delta}\log(1/\hbar)]$
and any sequence $\bal$ of length $n$, the following estimate holds:
\begin{equation}\label{eq:hyperb-dispers}
\left\Vert \Pi_{\bal}\,\chi^{(m)}(P(\hbar))\right\Vert \leq C\, e^{m\delta/2}\,\hbar^{-\frac{d-1+\delta}{2}}\, J_{n}^{u}(\bal)^{-1/2}.
\end{equation}
\end{prop}
From the bound (\ref{eq:Jacob-bound}) on the coarse-grained Jacobians,
we see that (\ref{eq:hyperb-dispers}) becomes sharper than the obvious
bound $\norm{\Pi_{\bal}\chi^{(n)}(P(\hbar))}\leq1+\cO(h^{\infty})$
for times \begin{equation}
n\geq T_{1}\defeq\frac{(d-1)\log(1/h)}{\Lambda_{\min}^{u}}>2T_{\eps,\hbar}.\label{eq:T_1}\end{equation}
If we specialize Prop. \ref{pro:hyperb-dispers} to the case $n\approx4T_{\eps,\hbar}$
and insert $U^{-n/2}$ on the right, we obtain the following
\begin{cor}
\label{cor:hyper-disper-2}Take $\delta>0$ small. For $0<\hbar<h_{\delta}$,
take $\bal$, $\bbeta$ two arbitrary sequences of length $n=\left\lfloor 2T_{\eps,\hbar}\right\rfloor $.
Then,\begin{equation}
\left\Vert \Pi_{\bbeta\cdot\bal}\,\chi^{(n)}(P(\hbar))\right\Vert \leq C\,\hbar^{-\frac{d-1+c\delta}{2}}\, J_{n}^{u}(\bal)^{-1/2}J_{n}^{u}(\bbeta)^{-1/2},\label{eq:hyper-disper-2}\end{equation}
with uniform constants $C,c>0$.
\end{cor}
It is this estimate which we will use in $\S$\ref{sub:application-of-EUP}.

\subsubsection{A remark on the sharpness of (\ref{eq:hyper-disper-2})}

Let us give a handwaving argument to show that, in the case of a surface
($d=2$) of constant curvature, the upper bound (\ref{eq:hyper-disper-2})
is close to being sharp. This argument was made rigorous in the case
of the toy model studied in \cite{AnaNo07-1}. Since our argument
is sketchy, we set all {}``small constants'' ($\eps,\delta$) to
zero.

The operator $\Pi_{\bbeta\cdot\bal}$ is the product of two quasiprojectors,
$\Pi_{\bal}$ associated with the {}``thin stable'' rectangle $E_{\bal}$,
which has length $\lesssim\hbar$ along the unstable direction, and
$\Pi_{\bbeta\cdot}$ associated with the {}``thin unstable'' rectangle
$E_{\bbeta\cdot}$ which has length $\lesssim\hbar$ along the stable
direction. The intersection $E_{\bbeta\cdot\bal}$ has length $\lesssim\hbar$
along the two directions transverse to the flow, which are symplectically
conjugate to each other. As a result, the refined smoothed characteristic
function $\pi_{\bbeta\cdot\bal}$ does not belong to any {}``nice''
symbol class, and the norm of the operator $\Pi_{\bbeta\cdot\bal}$
is not connected with the sup-norm of $\pi_{\bbeta\cdot\bal}$. 

Since $E_{\bal}$ has symplectic volume $\lesssim\hbar$, the {}``essential
rank'' of $\Pi_{\bal}$ is of order $\cO(1)$: $\Pi_{\bal}$ resembles
a projector on a subspace spanned by \emph{finitely many} normalized
{}``stable states'' $s_{\bal}^{i}$ localized in $E_{\bal}$, \[
\Pi_{\bal}\approx\sum_{i}s_{\bal}^{i}\otimes s_{\bal}^{i*}.\]
Similarly, $\Pi_{\bbeta\cdot}$ effectively projects on $\cO(1)$
normalized {}``unstable states'' $u_{\bbeta}^{j}$ localized in
$E_{\bbeta\cdot}$. The stable and unstable directions are symplectically
conjugate to each other, so that stable and unstable states behave
like position vs momentum states in the phase space $\IR^{2}$. The
product operator \[
\Pi_{\bbeta\cdot\bal}=\Pi_{\bal}\Pi_{\bbeta\cdot}\approx\sum_{i,j}\la s_{\bal}^{i},u_{\bbeta}^{j}\ra s_{\bal}^{i}\otimes u_{\bbeta}^{j}\]
involves the overlaps between the two families of states, which are
all of order $\hbar^{1/2}$. It is thus natural to expect $\left\Vert \Pi_{\bbeta\cdot\bal}\right\Vert \sim\hbar^{\frac{1}{2}}$,
which is the order of the estimate (\ref{eq:hyper-disper-2}).

\subsection{Quantum entropy and pressure}

\subsubsection{Back to the symbolic measure $\mu_{\hbar}$}

We now turn back to the the symbolic measure $\mu_{\hbar}$ defined
in §\ref{sub:symbolic-measure}. We recall that for a fixed sequence
$\bal$, $\mu_{\hbar}([\cdot\bal])$ approximately measures the weight
of the state $\psi_{\hbar}$ inside the rectangle $E_{\bal}$. This
interpretation is actually possible as long as $\Pi_{\bal}$ can be
interpreted as a quasiprojector on this rectangle, that is for $n\leq2T_{\eps,\hbar}$.
Under this condition, we have seen that the only upper bounds at our
disposal are trivial:\[
\mu_{\hbar}([\cdot\bal])\leq1,\qquad n=|\bal|\leq2T_{\eps,\hbar}.\]
On the other hand, Proposition \ref{pro:hyperb-dispers} implies that
the weights of longer cylinders satisfy nontrivial bounds:\begin{align}
\mu_{\hbar}([\cdot\bal])\leq C\,\hbar^{-(d-1+c\delta)}\, J_{n}^{u}(\bal)^{-1}, & \qquad|\bal|=n\leq C_{\delta}\log\hbar^{-1}.\label{eq:hyper-disper-3}\end{align}

\begin{cor}
For times $n>T_{1}$ (see (\ref{eq:T_1})), the measure $\mu_{\hbar}$
is necessarily distributed over many cylinders of length $n$. This
corresponds to a dispersion phenomenon: the state $\psi_{\hbar}$
cannot be concentrated in $\cO(1)$ boxes $E_{\bal}$, since each
such box has a volume $\ll\hbar^{d-1}$.
\end{cor}
Following $\S$\ref{sub:KSentropy}, the distribution of the weights
$\left\{ \mu_{\hbar}([\cdot\bal]),\;|\bal|=n\right\} $ can be characterized
by an entropy. Since $\mu_{\hbar}$ was built from the quantum state
$\psi_{\hbar}$, it is natural to call this entropy a \textbf{quantum
entropy}: \begin{equation}
H_{0}^{n-1}(\psi_{\hbar},\cP_{sm,q})\defeq H_{0}^{n-1}(\mu_{\hbar})=\sum_{|\bal|=n}\eta\big(\mu_{\hbar}([\cdot\bal])\big).\label{eq:entropy-mu^u}\end{equation}
One can associate a \textbf{quantum pressure} to the state $\psi_{\hbar}$,
the quantum partition $\cP_{sm,q}$ and a set of weights $\left\{ w_{k},\: k=1,\ldots,K\right\} $.
Below we will be dealing with weights of the form $w_{\balpha}=J_{n}^{u}(\balpha)^{1/2}$.
The quantum pressures will also be denoted by $p_{0}^{n-1}(\psi_{\hbar},\cP_{sm,q},w)=p_{0}^{n-1}(\mu_{\hbar},w)$.

Upper bounds (\ref{eq:hyper-disper-3}) on the weights of {}``long''
cylinders have direct consequences on the values of the quantum entropies:\[
H_{0}^{n-1}(\psi_{\hbar},\cP_{sm,q})\geq n\Lambda_{\min}^{u}-(d-1+c\delta)\log\hbar^{-1}-\log C,\qquad n\leq C_{\delta}\log\hbar^{-1}.\]
The RHS is positive (and thus makes up a nontrivial lower bound) only
for $n>T_{1}$, that is for {}``long'' times. A similar lower bound
on the entropy of {}``long times'' was used in \cite{Ana08} to
deduce nontrivial information on the values of the entropies at {}``short''
times, and finally a lower bound on the KS entropy.

\subsection{Entropic uncertainty principles}

In \cite{AnaNo07-2,AnaKoNo06} we used a different strategy, which
we describe below. Instead of using the upper bounds (\ref{eq:hyper-disper-3})
at {}``long'' times $C_{\delta}\log\hbar^{-1}$, we rather use the
bound (\ref{eq:hyper-disper-2}) corresponding to {}``moderately
long'' times $4T_{\eps,\hbar}$. The strategy consists in interpreting
the operator on the LHS as a {}``block matrix element'' associated
with two different quantum partitions. Through a certain \emph{Entropic
Uncertainty Principle} (EUP), the bound (\ref{eq:hyper-disper-2})
then leads to a lower bound on the pressures $p_{0}^{n-1}(\mu_{\hbar},w)$
at {}``moderately long times'' $n\approx2T_{\eps,\hbar}$ (that
is, right below the Ehrenfest time). Using an approximate subadditivity
of those pressures, we then get a lower bound for finite time pressures. 

The central piece of this method resides in a certain\emph{ }entropic
uncertainty principle\emph{.} Before giving the precise version used
for our aims, we first give the simplest example of such a {}``principle'',
first proven by Maassen and Uffink \cite{MaaUff88}.
\begin{prop}
{[}EUP, level 1 (finite-dimensional projectors){]}

Consider two orthonormal bases in the Hilbert space $\IC^{N}$, $e=\left\{ e_{i}\right\} _{i=1,\ldots,N}$
and $f=\left\{ f_{j}\right\} _{j=1,\ldots,N}$. For any $\psi\in\IC^{N}$
of unit norm, consider the two probability distributions $\left\{ |\la e_{i}|\psi\ra|^{2},\, i=1,\ldots,N\right\} $,
$\left\{ |\la f_{j}|\psi\ra|^{2},\: j=1,\ldots,N\right\} $. Then
the entropies associated with these two distributions satisfy the
inequality \[
H(\psi,e)+H(\psi,f)\geq-2\log\left(\max_{i,j}|\la e_{i},f_{j}\ra|\right).\]

\end{prop}
For instance, take $e=\left\{ e_{i}\right\} $ the standard basis
on $\IC^{N}$, and for $f=\left\{ f_{j}\right\} $ the {}``discrete
momentum states'', related with $\left\{ e_{i}\right\} $ through
the discrete Fourier transform. All matrix elements satisfy $|\la e_{i},f_{j}\ra|=N^{-1/2}$,
so the inequality reads \[
H(\psi,e)+H(\psi,f)\geq\log N.\]
The inequality shows that the distributions of {}``position amplitudes''$\la e_{i},\psi\ra$
on the one hand, of {}``momentum amplitudes''$\la f_{j},\psi\ra$,
cannot be both arbitrarily localized. It is hence a form of {}``uncertainty
principle''.

If we call $\rho_{i}$ (resp. $\tau_{j}$) the orthogonal projector
on the state $e_{i}$ (resp. $f_{j}$), then each overlap $|\la e_{i},f_{j}\ra|$
can be interpreted as the operator norm $|\la e_{i},f_{j}\ra|=\left\Vert \tau_{j}\rho_{i}^{*}\right\Vert .$
This interpretation allows to generalize this {}``uncertainty principle''
as follows:
\begin{prop}
{[}EUP, level 2 (quantum partition){]}\label{pro:EUP2}

On a Hilbert space $\cH$, consider two quantum partitions of unity,
that is two finite sets of bounded operators $\rho=\left\{ \rho_{i},\: i\in I\right\} $,
$\pi=\left\{ \pi_{j},\: j\in J\right\} $ satisfying the identities\[
\sum_{i\in I}\rho_{i}^{*}\rho_{i}=Id,\qquad\sum_{j\in J}\tau_{j}^{*}\tau_{j}=Id\,.\]
To any normalized state $\psi\in\cH$ we associate the probability
distributions $\left\{ \left\Vert \rho_{i}\psi\right\Vert ^{2},\: i\in I\right\} $,
$\left\{ \left\Vert \tau_{j}\psi\right\Vert ^{2},\: j\in J\right\} $.
Then, the entropies associated with these two distributions satisfy\[
H(\psi,\rho)+H(\psi,\tau)\geq-2\log\left(\max_{i\in I,j\in J}\left\Vert \tau_{j}\rho_{i}^{*}\right\Vert \right).\]

\end{prop}
Because the instability of the flow may not be uniform, the coarse-grained
jacobians $J_{n}^{u}(\bal)$ may vary a lot among all $n$-sequences
$\bal$. For this reason, the estimates (\ref{eq:hyper-disper-2})
may also strongly depend on the sequences $\bal$, $\bbeta$. To counterbalance
these variations, it is convenient to use \emph{pressures} instead
of entropies (see $\S$\ref{sub:Pressures}). 
\begin{prop}
{[}EUP, level 2 (quantum weighted partition){]}

On a Hilbert space $\cH$, consider two quantum partitions of unity
$\rho=\left\{ \rho_{i},\: i\in I\right\} $, $\tau=\left\{ \tau_{j},\: j\in J\right\} $
as in Prop. \ref{pro:EUP2} and two families of weights $v=\left\{ v_{i}>0,\: i\in I\right\} $,
$w=\left\{ w_{j}>0,\: j\in J\right\} $. To any normalized state $\psi\in\cH$
correspond the probability distributions $\left\{ \left\Vert \rho_{i}\psi\right\Vert ^{2},\: i\in I\right\} $,
$\left\{ \left\Vert \tau_{j}\psi\right\Vert ^{2},\: j\in J\right\} $. 

Then, the pressures associated with these distributions and weights
satisfy\[
p(\psi,\rho,v)+p(\psi,\tau,w)\geq-2\log\left(\max_{i\in I,j\in J}v_{i}w_{j}\left\Vert \tau_{j}\rho_{i}^{*}\right\Vert \right).\]

\end{prop}
This version is almost sufficient for our aims. Yet, the quantum partitions
of unity we are using are localized near the energy shell $\cE$ (see
(\ref{eq:Quantum Partition},\ref{eq:Quant-partition-n})), and the
estimate (\ref{eq:hyperb-dispers}) starts with a sharp energy cutoff.
For these reasons, the version we will need is of the following form.
\begin{prop}
{[}EUP, level 3 (microlocal weighted partition){]}\label{pro:EUP,microlocal}

On a Hilbert space $\cH$, consider two approximate quantum partitions
of unity, that is two finite sets of bounded operators $\rho=\left\{ \rho_{i},\: i\in I\right\} $,
$\tau=\left\{ \tau_{j},\: j\in J\right\} $ satisfying the identities\[
\sum_{i\in I}\rho_{i}^{*}\rho_{i}=S_{\rho},\qquad\sum_{j\in J}\tau_{j}^{*}\tau_{j}=S_{\tau}\,,\]
and two families of weights $v=\left\{ v_{i},\: i\in I\right\} $,
$w=\left\{ w_{j},\: j\in J\right\} $ satisfying $V^{-1}\leq v_{i},w_{j}\leq V$
for some $V\geq1$. 

We assume that for some $0<\varepsilon\leq\min\left(|I|^{-2}V^{-2},|J|^{-2}V^{-2}\right)$
the above sum operators satisfy $0\leq S_{\rho/\tau}\leq1+\vareps$.
Besides, let $S_{c_{1}},S_{c_{2}}$ be two hermitian operators on
$ $$\cH$ satisfying $0\leq S_{c_{*}}\leq1+\vareps$, and related
with the above partitions as follows: \begin{align}
\left\Vert (S_{c_{2}}-1)\rho_{i}S_{c_{1}}\right\Vert  & \leq\varepsilon,\qquad\forall i\in I,\label{eq:Sc2-Sc1}\\
\left\Vert (S_{\rho/\tau}-1)S_{c_{1}}\right\Vert  & \leq\varepsilon.\label{eq:Spi-Sc1}\end{align}
Let us define the {}``cone norm'' \begin{equation}
c_{cone}\defeq\max_{i\in I,\, j\in J}v_{i}w_{j}\norm{\tau_{j}\rho_{i}^{*}S_{c_{2}}}.\label{eq:c_cone}\end{equation}

Then, for any $\psi\in\cH$ satisfying \begin{equation}
\norm{\psi}=1,\qquad\norm{(Id-S_{c_{1}})\psi}\leq\vareps,\label{eq:cone-c1}\end{equation}
the pressures of $\psi$ w.r.to the weighted partitions $(\rho,v)$
and $(\tau,w)$ satisfy the bound\[
p(\psi,\rho,v)+p(\psi,\tau,w)\geq-2\log\left(c_{cone}+3|I|V^{2}\varepsilon\right)+\cO\left(\epsilon^{1/5}\right).\]
The implied constant is independent of the weighted partitions or
the cutoff operators $S_{c_{*}}$.\end{prop}
\begin{proof}
The proof is a slight adaptation of the one given in \cite[Section 6]{AnaNo07-2}
(in the case $\cU=Id$). One considers a bounded operator $T:\cH^{|I|}\to\cH^{|J|}$
, and studies the norm of $T$ as an operator $l_{p}^{(v)}(\cH^{|I|})\mapsto l_{q}^{(w)}(\cH^{|J|})$,
with the weighted norms $\norm{\Psi}_{p}^{(v)}=\left(\sum_{i}v_{i}^{p-2}\norm{\Psi_{i}}^{p}\right)^{1/p}$
and similarly for $\norm{T\Psi}_{q}^{(w)}$. Notice that the weights
are {}``invisible'' when $p=q=2$. An auxiliary bounded operator
$O:\cH\to\cH$ is used to define a cone of states:\[
\cC(O,\vartheta)=\set{\Psi\in\cH^{|I|},\;\norm{O\Psi_{i}-\Psi_{i}}\leq\vartheta\norm{\Psi}_{2},\: i\in I}.\]
The proof of \cite[Thm. 6.3]{AnaNo07-2}, which uses a Riesz-Thorin
interpolation argument, shows that for any $\Psi$ in the cone $\cC(O,\vartheta)$,
one has:\begin{multline}
\forall t\in[0,1],\qquad\norm{T\Psi}_{\frac{2}{1-t}}^{(w)}\leq\left(c_{O}(T)+|I|V^{2}\vartheta\norm{T}_{2,2}\right)^{t}\,\norm{T}_{2,2}^{1-t}\,\norm{\Psi}_{\frac{2}{1+t}}^{(v)},\\
\mbox{where }\: c_{O}(T)=\max_{i,j}v_{i}w_{j}\norm{T_{ji}O}.\label{eq:Riesz-Thorin-1}\end{multline}
We now apply this result to the specific choice \[
\Psi_{i}\defeq\rho_{i}\psi,\qquad T_{ji}\defeq\tau_{j}\rho_{i}^{*},\qquad O\defeq S_{c_{2}},\]
where the state $\psi$ satisfies (\ref{eq:cone-c1}), that is, the
cone $\cC(S_{c_{1}},\vareps)$ is not empty (in the opposite case,
the statement of the theorem is empty). The relations (\ref{eq:Spi-Sc1})
then imply that \begin{equation}
\left\Vert (S_{\rho/\tau}-1)\psi\right\Vert \leq3\vareps.\label{eq:(S_rho-1)psi}\end{equation}
As a consequence, the state components $(T\Psi)_{i}=\tau_{i}S_{\rho}\psi=\tau_{i}\psi+\cO(\vareps)$. 

The same duality argument as in \cite[Lemma 6.5]{AnaNo07-2} shows
that the $l^{2}\to l^{2}$ norm of the operator $T$ takes the value
$\norm{T}_{2,2}=\norm{\sqrt{S_{\rho}}\sqrt{S_{\tau}}}$. Using the
spectral theorem, one easily deduces that $ $$\norm{\sqrt{S_{\rho/\tau}}\psi-\psi}\leq\sqrt{3\vareps}$,
so that $\left\Vert \sqrt{S_{\rho}}\sqrt{S_{\tau}}\psi-\psi\right\Vert \leq4\sqrt{\vareps}$,
and hence $\norm{T}_{2,2}\in[1-4\sqrt{\vareps},1+2\vareps]$. 

The $l^{2}$ norm of $\Psi$ is $\norm{\Psi}_{2}^{2}=\la\psi,S_{\rho}\psi\ra\in[1-3\vareps,1+\vareps]$.
From (\ref{eq:Sc2-Sc1}) and the fact that $\psi$ is in the cone
$\cC(S_{c_{1}},\eps)$, we easily get $\left\Vert (O-1)\rho_{i}\psi\right\Vert \leq2\vareps$,
so that $\Psi\in\cC(O,\vartheta)$ for $O=S_{c_{2}}$, $\vartheta=\frac{2\vareps}{1-3\vareps}$.
We are now in a position to apply (\ref{eq:Riesz-Thorin-1}) to the
above data. The constant $c_{O}(T)$ is equal to the $c_{cone}$ in
the statement of the proposition, so we get\[
\forall t\in[0,1],\qquad\norm{T\Psi}_{\frac{2}{1-t}}^{(w)}\leq\left(c_{cone}+3|I|V^{2}\vareps\right)^{t}\,\norm{T}_{2,2}^{1-t}\,\norm{\Psi}_{\frac{2}{1+t}}^{(v)}.\]
Let us now expand this expression when $t\to0$. We first split the
sum $\sum_{i}\norm{\Psi_{i}}^{\frac{2}{1+t}}$ between the terms $\norm{\Psi_{i}}\geq\eps$
and the remaining ones:
\begin{align*}
\sum_{i}v_{i}^{\frac{-2t}{1+t}}\norm{\Psi_{i}}^{\frac{2}{1+t}} & =\sum_{i,>}\norm{\Psi_{i}}^{2}-t\sum_{i,>}\left\Vert \Psi_{i}\right\Vert ^{2}\left(\log v_{i}^{2}\left\Vert \Psi_{i}\right\Vert ^{2}\right)+\cO\left(\left(t\log\vareps\right)^{2}\right)+\sum_{i,<}v_{i}^{\frac{-2t}{1+t}}\norm{\Psi_{i}}^{\frac{2}{1+t}}\\
 & =\left\Vert \Psi\right\Vert _{2}^{2}+t\, p(\Psi,v)+\cO\left(\left(t\log\vareps\right)^{2}\right)+\cO(|I|\vareps).
\end{align*}
We now take the logarithm of this expression, and use $\left\Vert \Psi\right\Vert _{2}^{2}=1+\cO(\vareps)$
to get
\begin{align*}
\log\norm{\Psi}_{\frac{2}{1+t}}^{(v)} & =\frac{t}{2}p(\Psi,v)+\cO\left(\left(t\log\vareps\right)^{2}+|I|\vareps+t^{2}p(\Psi,v)^{2}\right).\end{align*}
We can perform the same manipulations on the LHS of (\ref{eq:Riesz-Thorin-1}),
noticing that $\left\Vert T\Psi\right\Vert _{2}^{2}=\la S_{\tau}S_{\pi}\psi,S_{\pi}\psi\ra\in[1-10\vareps,1+10\vareps]$:
\begin{align*}
\log\norm{T\Psi}_{\frac{2}{1-t}}^{(w)} & =-\frac{t}{2}\, p(T\Psi,w)+\cO\left(\left(t\log\vareps\right)^{2}+|J|\vareps+t^{2}p(T\Psi,w)^{2}\right).
\end{align*}
We notice that both pressures satisfy simple bound $\abs{p(\bullet,\bullet)}\leq\log(V^{2}|I|)\leq3|\log\vareps|$,
and (from (\ref{eq:(S_rho-1)psi})) the estimate $(T\Psi)_{i}=\tau_{i}\psi_{\hbar}+\cO(\vareps)$.
Inserting these estimates in the logarithm of (\ref{eq:Riesz-Thorin-1})
and using $\left\Vert T\right\Vert _{2,2}\in[1-4\sqrt{\vareps},1]$,
we get\[
t\, p(\psi_{\hbar},\tau,w)+tp(\psi_{\hbar},\rho,v)\geq-2t\log\left(c_{cone}+3|I|V^{2}\vareps\right)+\cO\left(\sqrt{\vareps}+\left(t\log\vareps\right)^{2}+(|J|+|I|)\vareps\right).\]
We now need to make some assumptions on the values of $t$ to make
the remainder small. If we take $t=\vareps^{1/4}$, a simple power
counting shows that \[
p(\psi_{\hbar},\tau,w)+p(\psi_{\hbar},\rho,v)\geq-2\log\left(c_{cone}+3|I|V^{2}\vareps\right)+\cO\left(\vareps^{1/4}\left(\log\vareps\right)^{2}\right).\]

\end{proof}

\subsection{Our application of the entropy uncertainty principle\label{sub:application-of-EUP}}

We now apply the EUP to our semiclassical framework. Our choice of
quantum partitions is determined by the requirement that the operators
$\tau_{j}\rho_{i}^{*}S_{c_{2}}$ are of the form $\Pi_{\bbeta\cdot\bal}\,\chi^{(n)}(P(\hbar))$,
with $\bal,\bbeta$ two sequences of length $n$ close to the Ehrenfest
time. We will thus take\begin{gather}
\mbox{the time }\qquad n=\left\lfloor 2T_{\epsilon,\hbar}\right\rfloor ,\nonumber \\
\tau=\left\{ \Pi_{\bal}=\Pi_{\alpha_{n-1}}(n-1)\cdots\Pi_{\alpha_{0}},\:|\bal|=n\right\} ,\label{eq:application-partitions}\\
\rho=\left\{ \Pi_{\bbeta\cdot}^{*}=\Pi_{\beta_{-n}}(-n)\cdots\Pi_{\beta_{-1}}(-1),\:|\bbeta|=n\right\} ,\nonumber \\
v=w=\left\{ J_{n}^{u}(\bbeta)^{1/2},\:|\bbeta|=n\right\} ,\nonumber \\
S_{c_{1}}=\chi^{(0)}(P(\hbar)),\qquad S_{c_{2}}=\chi^{(n)}(P(\hbar)).\nonumber \end{gather}
The weights $v,w$ have been selected in order to balance the variations
of the upper bounds in (\ref{eq:hyper-disper-2}). Both the cardinals
$|I|=|J|=K^{n}$ and the upper bound $v_{i},w_{j}\leq e^{n\lambda_{\max}(d-1)/2}$
are $\cO(\hbar^{-M})$ for some $M>0$. Hence, we may take the small
paremeter $\vareps=\hbar^{L}$ for some fixed exponent $L\gg M$.

With this choice, the assumption (\ref{eq:Spi-Sc1}) is satisfied
for $\hbar$ small enough according to Prop. \ref{pro:qu-part-local}.
The assumption (\ref{eq:Sc2-Sc1}) can be checked by inserting the
increasing sequence of cutoffs $\left\{ \chi^{(j)}(P(\hbar)),\:1\leq j\leq n-1\right\} $
along the sequence $\Pi_{\bbeta}^{*}$, similarly as in the proof
of Prop. \ref{pro:qu-part-local}. The assumption (\ref{eq:cone-c1})
holds if one takes $\psi=\psi_{\hbar}$ a null eigenstate of $P(\hbar)$,
see (\ref{eq:local-psi-hbar}).

The coefficient $c_{cone}$ is then estimated by the hyperbolic dispersive
estimate of Corollary \ref{cor:hyper-disper-2}:\begin{equation}
c_{cone}\leq C_{cone}(\hbar)\defeq C\,\hbar^{-\frac{d-1+c\delta}{2}},\quad\hbar\leq\hbar_{\delta}.\label{eq:C_cone}\end{equation}
The application of Prop. \ref{pro:EUP,microlocal} to these data gives
the following result.
\begin{prop}
{[}Applied entropic uncertainty principle{]}

Take the weighted quantum partitions $(\rho,v)$, $(\tau,w)$ defined
in (\ref{eq:application-partitions}), and $L$ a large positive number.
Then, there exists $\hbar_{L}>0,\: C>0$ such that the pressures of
the eigenstates $(\psi_{\hbar})$ associated with these weighted partitions
satisfy the inequality:\begin{equation}
p(\psi_{\hbar},\rho,v)+p(\psi_{\hbar},\tau,w)\geq-2\log C_{cone}(\hbar)+C\hbar^{L/3},\qquad\hbar<\hbar_{L}.\label{eq:applied-EUP}\end{equation}

\end{prop}
Our next task will be to relate the pressures associated to the {}``long
time'' partitions ($n=\left\lfloor 2T_{\epsilon,\hbar}\right\rfloor $),
to pressures associated with {}``finite time'' partitions ($n_{0}$
independent of $\hbar$).

\subsection{Approximate subadditivity of the quantum pressure}

In (\ref{eq:applied-EUP}) appear two pressures, associated with two
types of refined quantum partitions. The partition $\tau=\left\{ \Pi_{\alpha_{n-1}}(n-1)\cdots\Pi_{\alpha_{0}},\:|\bal|=n\right\} $
corresponds to the definition of the symbolic measure $\mu_{\hbar}$
in (\ref{eq:mu^u-definition}), so that the pressure $p(\psi_{\hbar},\tau,w)$
can be expressed as the refined pressure $p_{0}^{n-1}(\mu_{\hbar},w)$.
On the other hand, the probability weights involved in $p(\psi_{\hbar},\rho,v)$
can be rewritten (after relabelling the sequences and using the fact
that $\psi_{\hbar}$ is an eigenmode) $\rho_{\bbeta}=\norm{\Pi_{\beta_{n-1}}(-n)\cdots\Pi_{\beta_{0}}(-1)\psi_{\hbar}}^{2}$,
so they correspond to a backwards evolution.%
\footnote{Notice that, using the fact that $\psi_{\hbar}$ is an eigenmode of $U$, these weights can also be written as $\lVert\Pi_{\beta_{n-1}}(0)\dotsm\Pi_{\beta_{0}}(n-1)\psi_{\hbar}\rVert^2$, which is of the same form as $\norm{\tau_{\beta_{n-1}\dotsm\beta_0}\psi_{\hbar}}^2$, except for the ordering of the operators.}
We express these weights in terms of a {}``backwards symbolic measure''
$\tilde{\mu}_{\hbar}$ similar with $\mu_{\hbar}$: 
\begin{equation}
\tilde{\mu}_{\hbar}([\cdot\beta_{0}\cdots\beta_{n-1}])\defeq\left\Vert \Pi_{\beta_{n-1}}(-n)\cdots\Pi_{\beta_{0}}(-1)\psi_{\hbar}\right\Vert ^{2},\label{eq:tilde-mu^u}
\end{equation}
so that $p(\psi_{\hbar},\rho,v)=p_{0}^{n-1}(\tilde{\mu}_{\hbar},v)$.
For fixed $n>0$, and any $|\bbeta|=n$, we have $\tilde{\mu}_{\hbar}([\cdot\bbeta])\hto0\mu_{\mathrm{sc}}(\pi_{\beta_{n-1}\dotsm\beta_0})$.

The bound (\ref{eq:applied-EUP}) concerns the quantum pressures of
$\mu_{\hbar}$ and $\tilde{\mu}_{\hbar}$ at the time $n=\left\lfloor 2T_{\epsilon,\hbar}\right\rfloor $
close to the Ehrenfest time. These pressures cannot be directly connected
with the pressures of the semiclassical measure $\mu_{sc}$. For this
aim, we need to deduce from (\ref{eq:applied-EUP}) a lower bound
for the pressures of $\mu_{\hbar}$ and $\tilde{\mu}_{\hbar}$ at
\emph{finite} times $n_{o}$. This connection will be done by using
an approximate version of the subadditivity property (\ref{eq:subadd-pressure}).
We present the computations in the case of $\mu_{\hbar}$, the case
of $\tilde{\mu}_{\hbar}$ being identical up to exchanging forward
and backward evolution.

The symbolic measure $\mu_{\hbar}$ was defined on cylinders $[\cdot\bal]$,
up to $\cO(\hbar^{\infty})$ errors which are uniformly controlled
as long as $|\bal|\leq C|\log\hbar|$. This definition was possible
thanks to the (approximate) compatibility condition (\ref{eq:compatibility}).
The property (\ref{eq:subadditivity}) can be extended to the measure
$\mu_{\hbar}$:\begin{equation}
\forall n,n_{0}\leq C|\log\hbar|,\quad H_{0}^{n+n_{0}-1}(\mu_{\hbar})\leq H_{0}^{n-1}(\mu_{\hbar})+H_{n}^{n+n_{0}-1}(\mu_{\hbar})+\cO(\hbar^{\infty}).\label{eq:subadd-quant-1}\end{equation}
The second term on the RHS can be written as $H_{n}^{n+n_{0}-1}(\mu_{\hbar})=H_{0}^{n_{0}-1}(\sigma_{*}^{-n}\mu_{\hbar})$,
where $\sigma$ refers to a shift of indices. Shift-invariance would
mean that $\sigma^{-n}\mu_{\hbar}=\mu_{\hbar}$, which would allow
us to replace this entropy with $H_{0}^{n_{0}-1}(\mu_{\hbar})$.
How far are we from this invariance? To answer this question we need
to compare the weights
\begin{equation}
\mu_{\hbar}(\sigma^{-n}[\cdot\beta_{0}\cdots\beta_{n_{0}-1}])\defeq\sum_{\alpha_{0},\ldots,\alpha_{n-1}}\mu_{\hbar}([\cdot\alpha_{0}\cdots\alpha_{n-1}\beta_{0}\cdots\beta_{n_{0}-1}]),\label{eq:shifted-measure}
\end{equation}
with the weights $\mu_{\hbar}([\cdot\beta_{0}\cdots\beta_{n_{0}-1}])$.
This is achieved in the following 
\begin{lem}
Let $(\mu_{\hbar})_{\hbar\to0}$, $(\tilde{\mu}_{\hbar})_{\hbar\to0}$
be the symbolic measures associated with the eigenstates $(\psi_{\hbar})_{\hbar\to0}$.
Fix some $n_{0}\geq0$, and take $n$ in the range $[0,2T_{\epsilon,\hbar}-n_{0}]$.
Then, for any $\bbeta$ of length $n_{0}$, we have\[
\mu_{\hbar}(\sigma^{-n}[\cdot\bbeta])=\mu_{\hbar}([\cdot\bbeta])+\cO(\hbar^{\eps/2}),\qquad\tilde{\mu}_{\hbar}(\sigma^{-n}[\cdot\bbeta])=\tilde{\mu}_{\hbar}([\cdot\bbeta])+\cO(\hbar^{\eps/2}).\]

\end{lem}
This lemma means that the measures $\mu_{\hbar}$, $\tilde{\mu}_{\hbar}$
are approximately shift-invariant in the semiclassical limit.
\begin{proof}
We only give the proof for the measure $\mu_{\hbar}$. Each term on
the RHS of (\ref{eq:shifted-measure}) reads\[
\la\Pi_{\bbeta}(n)\Pi_{\bal}\psi_{\hbar},\Pi_{\bbeta}(n)\Pi_{\bal}\psi_{\hbar}\ra=\la\Pi_{\bal}^{*}\left|\Pi_{\bbeta}(n)\right|^{2}\Pi_{\bal}\psi_{\hbar},\psi_{\hbar}\ra.\]
We first present a short (but false) proof. In order to use (\ref{eq:Quant-partition-n}),
we try to bring together the operator $\Pi_{\bal}$ and its hermitian
conjugate, such as to let appear the sum $\sum_{|\bal|=n}\Pi_{\bal}^{*}\Pi_{\bal}$.
An error appears while commuting $\Pi_{\bal}$ with $\left|\Pi_{\bbeta}(n)\right|^{2}$.
Still, from Proposition \ref{pro:P_alpha-symbol} we know that for
any $|\bal|=n$, $|\bbeta|=n_{0}$ with $n+n_{0}\leq2T_{\epsilon,\hbar}$,
the operators $\Pi_{\bal}(-\frac{n}{2})$ and $\left|\Pi_{\bbeta}(\frac{n}{2})\right|^{2}$
belong to the class $\Psi_{\nu}^{-\infty,0}$, with $\nu\in[\frac{1-\eps}{2},\frac{1}{2})$.
As a consequence, their commutator satisfies $\left[\left|\Pi_{\bbeta}(\frac{n}{2})\right|^{2},\Pi_{\bal}(-\frac{n}{2})\right]\in\Psi_{\nu}^{-\infty,-1+2\nu}$,
an operator of norm $\cO\left(\hbar^{1-2\nu}\right)$. By unitarity
of the evolution, the commutator $\left[\left|\Pi_{\bbeta}(n)\right|^{2},\Pi_{\bal}\right]$
has the same norm. We thus get 
\[
\sum_{\bal}\mu_{\hbar}([\cdot\bal\bbeta])=\la\left|\Pi_{\bbeta}(n)\right|^{2}S_{n}\psi_{\hbar},\psi_{\hbar}\ra+\sum_{\bal}\cO\left(\hbar^{1-2\nu}\right)=\mu_{\hbar}([\cdot\bbeta])+\cO\left(K^{n}\hbar^{1-2\nu}\right).
\]
The remainder in the RHS is small if $n$ is uniformly bounded, but
it becomes very large if $n\approx2T_{\eps,\hbar}$! 

To remedy this problem one actually needs to \emph{successively} group
together the pairs of operators $\Pi_{\alpha_{j}}(j)$ and perform
the sum over $\alpha_{j}$. One starts by grouping together the $\Pi_{\alpha_{n-1}}$:
{\multlinegap0pt
\begin{multline*}
\la\Pi_{\bbeta}(n)\Pi_{\bal}\psi_{\hbar},\Pi_{\bbeta}(n)\Pi_{\bal}\psi_{\hbar}\ra\\
\begin{aligned}
&=\la\Pi_{\alpha_{n-1}}^*(n-1)\lvert \Pi_{\bbeta}(n)\rvert^2\Pi_{\alpha_{n-1}}(n-1)\Pi_{\alpha_0\dotsm\alpha_{n-2}}\psi_{\hbar},\Pi_{\alpha_0\dotsm\alpha_{n-2}}\psi_{\hbar}\ra\\
&=\la\lvert \Pi_{\alpha_{n-1}}(n-1)\rvert^2\lvert \Pi_{\bbeta}(n)\rvert^2\Pi_{\alpha_0\dotsm\alpha_{n-2}}\psi_{\hbar},\Pi_{\alpha_0\dotsm\alpha_{n-2}}\psi_{\hbar}\ra
\end{aligned}\\
{}+\la\Pi_{\alpha_{n-1}}^*(n-1)[\lvert \Pi_{\bbeta}(n)\rvert^2,\Pi_{\alpha_{n-1}}(n-1)]\Pi_{\alpha_0\dotsm\alpha_{n-2}}\psi_{\hbar},\Pi_{\alpha_0\dotsm\alpha_{n-2}}\psi_{\hbar}\ra.\end{multline*}}%
The commutator in the RHS is an operator of norm $\cO\left(\hbar^{1-2\nu}\right)$
for the reasons indicated above. The second overlap is then bounded
from above by $\left\Vert \Pi_{\alpha_{0}\cdots\alpha_{n-2}}\psi_{\hbar}\right\Vert ^{2}\cO\left(\hbar^{1-2\nu}\right)$,
where the implied constant does not depend on $\bal$. Using the quantum
partition of order $n-1$, we can sum this second term over $\bal$,
to obtain an error $\cO\left(\hbar^{1-2\nu}\right)$. The first term
on the RHS can be summed over $\alpha_{n-1}$, to produce 
\begin{multline*}
\la\Oph(\tilde{\chi}_{\eps/2})(n-1)\left|\Pi_{\bbeta}(n)\right|^{2}\Pi_{\alpha_{0}\cdots\alpha_{n-2}}\psi_{\hbar},\Pi_{\alpha_{0}\cdots\alpha_{n-2}}\psi_{\hbar}\ra\\
=\la\left|\Pi_{\bbeta}(n)\right|^{2}\Pi_{\alpha_{0}\cdots\alpha_{n-2}}\psi_{\hbar},\Pi_{\alpha_{0}\cdots\alpha_{n-2}}\psi_{\hbar}\ra+\cO(\hbar^{\infty}),
\end{multline*}
where we used the fact that $\tilde{\chi}_{\eps/2}\equiv1$ on the
microsupport of $\Pi_{\alpha_{0}\cdots\alpha_{n-2}}\psi_{\hbar}$
(see the proof of Lemma \ref{pro:qu-part-local}). Apart from the
errors, we are now left with the sum \[
\sum_{\alpha_{0},\ldots,\alpha_{n-2}}\la\left|\Pi_{\bbeta}(n)\right|^{2}\Pi_{\alpha_{0}\cdots\alpha_{n-2}}\psi_{\hbar},\Pi_{\alpha_{0}\cdots\alpha_{n-2}}\psi_{\hbar}\ra.\]
This sum can be treated as above, by bringing $\Pi_{\alpha_{n-2}}(n-2)$
to the left, commuting it with $\left|\Pi_{\bbeta}(n)\right|^{2}$,
and summing over $\alpha_{n-2}$. It produces another error $\cO\left(\hbar^{1-2\nu}\right)$.
Iterating this procedure down to $\alpha_{0}$, we get \[
\sum_{\bal}\mu_{\hbar}([\cdot\bal\bbeta])=\mu_{\hbar}([\cdot\bbeta])+\cO\left(n\hbar^{1-2\nu}\right).\]
By selecting $\nu$ appropriately, we get the statement of the lemma.
\end{proof}
Coming back to the subadditivity equation (\ref{eq:subadd-quant-1})
in the case $n_{o}$ fixed, $n=\left\lfloor 2T_{\eps,\hbar}\right\rfloor -n_{0}$,
we get
\begin{align*}
H_{0}^{n-1}(\mu_{\hbar}) & \leq H_{0}^{n-n_{o}-1}(\mu_{\hbar})+H_{0}^{n_{o}-1}(\sigma_{*}^{-n}\mu_{\hbar})+\cO(\hbar^{\infty})\\
 & \leq H_{0}^{n-n_{o}-1}(\mu_{\hbar})+H_{0}^{n_{o}-1}(\mu_{\hbar})+\cO_{n_{o}}(\hbar^{\eps/3}).
\end{align*}
Here we used the fact that the function $\eta(s)$ satisfies $\left|\eta(s+s')-\eta(s)\right|\leq\eta(|s'|)$.
It is also easy to check that the {}``potential part'' of the pressure
$p(\mu_{\hbar},w)$ satisfies this inequality:
\begin{multline*}
\sum_{\lvert\bal\rvert =n}\sum_{\lvert \bbeta\rvert =n_0}\mu_{\hbar}([\cdot\bal\bbeta])\log w_{\bal\bbeta} \\
\begin{aligned}
& =\sum_{\lvert\bal\rvert =n}\mu_{\hbar}([\cdot\bal])\log w_{\bal}+\sum_{\lvert \bbeta\rvert =n_0}\sigma^{-n}\mu_{\hbar}([\cdot\bbeta])\log w_{\bbeta}+\cO(\hbar^{\infty})\\
 & =\sum_{\lvert\bal\rvert =n}\mu_{\hbar}([\cdot\bal])\log w_{\bal}+\sum_{\lvert\bbeta\rvert =n_0}\mu_{\hbar}([\cdot\bbeta])\log w_{\bbeta}+\cO(\hbar^{\eps/2}),
\end{aligned}
\end{multline*}
so we finally get the approximate pressure subadditivity
\begin{equation}
p_{0}^{n-1}(\mu_{\hbar},w)\leq p_{0}^{n-n_{o}-1}(\mu_{\hbar},w)+p_{0}^{n_{o}-1}(\mu_{\hbar},w)+\cO_{n_{o}}(\hbar^{\eps/3}).\label{eq:pressure-approx-subadd}
\end{equation}
Taking the Euclidian quotient $n=qn_{o}+r$, $r<n_{o}$, we can iterate
this process $q$ times and obtain: 
\[
p_{0}^{n+n_{0}-1}(\mu_{\hbar},w)\leq p_{0}^{r-1}(\mu_{\hbar},w)+q\, p_{0}^{n_{o}-1}(\mu_{\hbar},w)+\cO_{n_{o}}(q\hbar^{\eps/3}).
\]

A similar approximate subadditivity holds for the pressures $p(\tilde{\mu}_{\hbar},v)$
associated with the {}``backwards'' symbolic measure $\tilde{\mu}_{\hbar}$.
We have thus obtained the following 
\begin{prop}
Let $\mu_{\hbar},\,\tilde{\mu}_{\hbar}$ be the associated forward
and backwards symbolic measures associated with $(\psi_{\hbar})$.
Fix some $n_{o}>0$, and for $\hbar<\hbar_{0}$ split the Ehrenfest
time $n=\left\lfloor 2T_{\eps,\hbar}\right\rfloor $$ $ into $n=qn_{o}+r$,
$r\in[0,n_{o})$. Then, the pressures associated with these measures
and the weights $w_{k}=v_{k}=J^{u}(k)^{1/2}$ satisfy the following
lower bound:
\begin{multline}\label{eq:quant-entropy-bound}
q\bigl(p_0^{n_0-1}(\mu_{\hbar},w)+p_0^{n_0-1}(\tilde{\mu}_{\hbar},v)\bigr)+p_0^{r-1}(\mu_{\hbar},w)+p_0^{r-1}(\tilde{\mu}_{\hbar},v)\\
\geq-(d-1+c\delta)\lvert\log\hbar\rvert +\cO_{n_0}(\hbar^{\eps/4}).
\end{multline}
\end{prop}
Hence, approximate subadditivity has enabled us to transfer the bound
on pressures for $n=\left\lfloor 2T_{\eps,\hbar}\right\rfloor $ into
a bound on pressures at finite times $n_{0},r$.

\subsection{Back to the classical pressure}

Using the fact that the pressures $p_{0}^{r-1}(\bullet)$ are uniformly
bounded, we divide (\ref{eq:quant-entropy-bound}) by $qn_0=2T_{\eps,\hbar}-\cO(1)=\lvert\log\hbar\rvert \*\bigl({1+\cO(\eps)}\bigr)/\lambda_{\max}$. For $\hbar<\hbar_\eps$, we get
\begin{align*}
\frac{p_{0}^{n_{0}-1}(\mu_{\hbar},w)}{n_{0}}+\frac{p_{0}^{n_{0}-1}(\tilde{\mu}_{\hbar},v)}{n_{0}} & \geq\frac{-(d-1+c\delta)|\log\hbar|}{qn_{0}}+\cO_{n_{o}}(|\log\hbar|^{-1})\\
 & \geq-(d-1+c\delta)\lambda_{\max}(1+\cO(\eps)).
\end{align*}
The RHS does not depend on $\hbar$, so it is still valid once we
take the semiclassical limit of the LHS. From the properties of $\mu_{\hbar}$
and $\tilde{\mu}_{\hbar}$, the latter is equal to the ratio $2\frac{p_{0}^{n_{o}-1}(\mu_{sc},v,\cP_{sm})}{n_{0}}$,
where we recall that $\mu_{sc}$ is the semiclassical measure associated
with $(\psi_{\hbar})$, while $\cP_{sm}$ is the smoothed partition
(\ref{eq:smooth-partition}). We thus obtained the following lower
bound on the classical pressure:
\[
\frac{p_{0}^{n_{0}-1}(\mu_{sc},v,\cP_{sm})}{n_{o}}\geq-\frac{(d-1)\lambda_{\max}}{2}+\cO(\delta,\eps).
\]
Importantly, the implied constants do not depend on the degree of
smoothness of the partition (that is, on the derivatives of the functions
$\pi_{k}$), so we may send $\delta\to0$ and get rid of the smoothing.
We thus obtain the following lower bound on the pressure associated
with the \emph{sharp} partition $\cP$:
\[
\frac{p_{0}^{n_{0}-1}(\mu_{sc},v,\cP)}{n_{o}}\geq-\frac{(d-1)\lambda_{\max}}{2}+\cO(\eps).
\]
We recall that $\eps$ majorizes the diameter of $\cP$. The first
part of the pressure is the entropy $H(\mu_{sc},\cP^{\vee n_{0}})$,
while the second part is given by the sum 
\[
-\sum_{|\bal|=n_{0}}\mu_{sc}(E_{\bal})\log J_{n_{0}}^{u}(\bal)=-n_{0}\sum_{k=1}^{K}\mu_{sc}(E_{k})\log J^{u}(k),
\]
where we remind that $J^{u}(k)=\min_{\rho\in E_{k}}J^{u}(\rho)$.
We can now let $n_{o}\to\infty$, and get
\[
H_{KS}(\mu_{sc},\cP)\geq-\frac{(d-1)\lambda_{\max}}{2}+\sum_{k=1}^{K}\mu_{sc}(E_{k})\log J^{u}(k)+\cO(\eps).
\]
By taking finer and finer partitions (that is, $\eps\to0$), we finally
get the bound (\ref{eq:bound-general2}) for the semiclassical measure
$\mu_{sc}$. \hfill$\blacksquare$

\section{Proof of the hyperbolic dispersive estimate\label{sec:Hyperbolic dispersive estimate}}

\subsection{Decomposition into adapted Lagrangian states\label{sub:Decomposition-Lagrangian} }

The proof of Proposition~\ref{pro:hyperb-dispers} starts from an
arbitrary $\Psi\in L^{2}$ with $\left\Vert \Psi\right\Vert _{L^{2}}=1$.
The localized state $\psi_{\alpha_{0}}\defeq\Pi_{\alpha_{0}}\chi^{(n)}(P(\hbar))\Psi$
will then be decomposed into a linear combination of {}``nice''
Lagrangian states. To construct these {}``nice states'', we need
to consider, on each neighbourhood $\tilde{E}_{k}\supset E_{k}$,
a coordinate chart $\left\{ (y_{i},\eta_{i}),\; i=0,\ldots,d-1\right\} $
adapted to the dynamics of the geodesic flow. These coordinates are
required to satisfy the following properties:
\begin{enumerate}
\item $\tilde{E}_{k}$ is contained in the polydisk $D(\eps,\eps)=\left\{ (y,\eta),\,|y|\leq\eps,\,|\eta|\leq\eps\right\} $.
\item the coordinate $\eta_{0}=p_{0}(x,\xi)$, so that the energy shells
are given by $\left\{ \eta_{0}=const\right\} $, and the conjugate
variable $y_{0}$ represents the time along the trajectory. 
\item the planes $\left\{ \eta=const\right\} $ are close to the local weak
unstable manifolds $W_{\eps}^{u0}$ in $\tilde{E}_{k}$. For this
aim, we let the plane $\left\{ \eta=0\right\} $ coincide with the
local unstable manifold $W_{\eps}^{u0}(\rho_{k})$ for some arbitrary
point $\rho_{k}\in E_{k}\cap\cE$. \end{enumerate}
\begin{defn}
\label{def:gamma_1-cone}We say that a Lagrangian leaf $\Lambda\subset\tilde{E}_{k}$
belongs to the \emph{$\gamma_{1}$-cone} if it is represented, in
the chart $\left\{ (y,\eta)\right\} $, as \begin{equation}
\Lambda=\left\{ (y,dS(y)),\:|y|\leq\eps\right\} ,\qquad\text{with}\quad\sup_{|y|\leq\eps}\left\Vert d^{2}S(y)\right\Vert \leq\gamma_{1}.\label{eq:gamma_1}\end{equation}
Fixing some $\gamma_{1}>0$, there exists $\eps_{\gamma_{1}}>0$ such
that, provided the diameters of the $\tilde{E}_{k}$ are all smaller
than $\eps_{\gamma_{1}}$, then the $\gamma_{1}$-cone contains all
the local unstable manifolds $W_{\eps}^{u0}(\rho)$, $\rho\in\tilde{E}_{k}$,
while all local stable manifolds $W_{\eps}^{s}(\rho)$ are uniformly
transverse to this cone. We call such a cone an \emph{unstable $\gamma_{1}$-cone}.
\end{defn}
Let $\cU_{k}$ be a semiclassical Fourier Integral Operator (FIO)
associated with the change of coordinates 
\begin{equation}
(x,\xi)\in T^{*}X\to(y,\eta)\in\IR^{2d},\label{eq:coord-change}
\end{equation}
unitary microlocally near $\tilde{E}_{k}\times D(\eps,\eps)$. This
means that for any $\psi_k\in L^{2}(X)$, $\lVert\psi_k\rVert=\cO(1)$, microlocalized inside $\tilde{E}_{k}$,
we have 
\[
\left\Vert \psi_k\right\Vert_{L^{2}(X)}=\left\Vert \cU_{k}\psi_k\right\Vert_{L^{2}(\IR^{d})}+\cO(\hbar^{\infty}),
\]
and the function $\cU_{k}\psi_{k}(y)$ is then microlocalized inside $D(\eps,\eps)$.
Hence, using the cutoff 
\[
\chi\in C_{0}^{\infty}(\left\{ |y|\leq2\eps\right\} ),\quad\chi=1\quad\mbox{for }\:|y|\leq\eps,\]
 we may construct a family of {}``localized plane waves'' 
\[
e_{\eta}(y)=\chi(y)\exp(i\la\eta,y\ra/\hbar),\qquad(y,\eta)\in D(2\eps,2\eps),
\]
such that the function $\cU_{k}\psi_{k}(y)$ can be Fourier expanded
into\[
\cU_{k}\psi_{k}=(2\pi\hbar)^{-d/2}\int_{|\eta|\leq2\eps}e_{\eta}\,\hat{\psi}_{k}(\eta)\, d\eta+\cO_{L^{2}}(\hbar^{\infty}),\]
where $\hat{\psi}_{k}=\cF_{\hbar}\psi_{k}$ is the $\hbar$-Fourier
transform of $\cU_{k}\psi_{k}(y)$. Each state $e_{\eta}$, $|\eta|\leq2\eps$,
is a Lagrangian state associated with the {}``horizontal'' Lagrangian
leave $\Lambda_{\eta}=\left\{ (y,\eta),\:|y|\leq2\eps\right\} $. 

The change of coordinates (\ref{eq:coord-change}) brings the energy
layer $\left\{ |p-1/2|\leq\hbar^{1-\delta}\right\} $ into the slice
$\left\{ |\eta_{0}|\leq\hbar^{1-\delta}\right\} $. As a result, the
states $\psi_{k}$ which are sharply localized in energy are easy
to characterize.
\begin{lem}
\label{lem:sharp-cutoff-decompo}Assume that for some integer $m\leq C_{\delta}|\log\hbar|$
the state $\psi_{k}$ satisfies \begin{equation}
\chi^{(m)}(P(\hbar))\psi_{k}=\psi_{k}+\cO_{L^{2}}(\hbar^{\infty}).\label{eq:energy-local}\end{equation}
In that case, the state $\psi_{k}$ can be decomposed into\begin{equation}
\psi_{k}=(2\pi\hbar)^{-d/2}\int_{|\eta_{0}|\leq e^{m\delta}\hbar^{1-\delta},|\eta'|\leq2\eps}\cU_{k}^{*}e_{\eta}\,\hat{\psi}_{k}(\eta)\, d\eta+\cO_{L^{2}}(\hbar^{\infty}).\label{eq:decompo-sharp}\end{equation}
\end{lem}
\begin{proof}
The assumption (\ref{eq:energy-local}) and the microlocalization
of $\cU_{k}\psi_{k}$ inside $D(2\eps,2\eps)$ imply that its Fourier
transform $\hat{\psi}_{k}$ satisfies $\hat{\psi}_{k}(\eta)=\cO\left(\left(\frac{\hbar}{\la\eta\ra}\right)^{\infty}\right)$
for $\eta$ outside the strip $\left\{ |\eta_{0}|\leq e^{m\delta}\hbar^{1-\delta},\;|\eta'|\leq2\eps\right\} $. 
\end{proof}
Our aim is to prove (\ref{eq:hyperb-dispers}). We consider
a sequence $\bal$ of length $n$, and an arbitrary $\Psi\in L^{2}(X)$.
We then apply the above decomposition to the case $k=\alpha_{0}$
and the state $\psi_{\alpha_{0}}=\Pi_{\alpha_{0}}\,\chi^{(n)}(P(\hbar))\,\Psi$.
By construction, this state satisfies the energy localization (\ref{eq:energy-local})
if we take $m\geq n+1$.

\subsection{Evolution of individual Lagrangian states}

Our strategy will consist in controlling the states $\Pi_{\bal}\cU_{\alpha_{0}}^{*}e_{\eta}$
individually. For each $|\eta|\leq2\eps$, the state \begin{equation}
\psi^{0}\defeq\cU_{\alpha_{0}}^{*}e_{\eta}\in L^{2}(X)\label{eq:initial-state}\end{equation}
is a Lagrangian state associated with a certain Lagrangian leaf $\Lambda^{0}$
which belongs to some unstable $\gamma_{1}$-cone in $\tilde{E}_{\alpha_{0}}$.
The operator $\Pi_{\bal}$ is a succession of evolutions along the
Schrödinger flow ($U$) and truncations by quasiprojectors $\Pi_{\alpha_{i}}$.
Each quasiprojector $\Pi_{k}=\Oph(\tilde{\pi}_{k})$ is a pseudodifferential
operator, which transforms a Lagrangian state associated with some
Lagrangian leaf $\Lambda$, into another Lagrangian state on the same
$\Lambda$, by modifying its symbol. This modification will generally
reduce the $L^{2}$ norm of the state. In turn, the propagator $U$
is a unitary Fourier Integral Operator associated with the map $g^{1}$,
which transforms a Lagrangian state on $\Lambda$ into a Lagrangian
state on $g(\Lambda)$, keeping the $L^{2}$-norm unchanged.

More precisely, the operator $\Pi_{\alpha_{i}}U$ acts as follows
on Lagrangian states.
\begin{prop}
\label{pro:1-step}Consider a Lagrangian leaf $\Lambda^{0}\subset\tilde{E}_{\alpha_{0}}$
in some unstable $\gamma_{1}$-cone, and a Lagrangian state $\psi^{0}\in L^{2}(X)$
localized on this leaf, of the form 
\[
\cU_{\alpha_{0}}\psi^{0}(y)=a^{0}(y)\, e^{iS^{0}(y)/\hbar},\qquad a^{0}\in C_{c}^{\infty}(D(\eps)).
\]
Then, the state $\psi^{1}=\Pi_{\alpha_{1}}U\psi^{0}$ is a Lagrangian
state associated with the leaf $\Lambda^{1}=g(\Lambda^{0})$. It can
be expressed (in the coordinates attached to $\tilde{E}_{\alpha_{1}}$)
as
\[
\cU_{\alpha_{1}}\psi^{1}(y)=a^{1}(y,\hbar)e^{iS^{1}(y)/\hbar},
\]
where $S^{1}$ is a generating function for $\Lambda^{1}$. The symbol
$a^{1}(y,\hbar)$ admits an expansion 
\begin{equation}
a^{1}(y,\hbar)=\sum_{j=0}^{L-1}\hbar^{j}a_{j}^{1}(y)+\hbar^{L}r_{L}(y,\hbar).\label{eq:expansion1}
\end{equation}
The inverse flow $g_{\rest\Lambda^{1}}^{-1}:\Lambda^{1}\to\Lambda^{0}$
can be expressed in the coordinates $y$ attached respectively to
$\tilde{E}_{\alpha_{1}}$ and $\tilde{E}_{\alpha_{0}}$, through a
map
\begin{equation}
y^{1}\in\pi\Lambda^{1}\subset D(2\eps)\mapsto y^{0}=\pi g^{-1}(y^{1},dS^{1}(y^{1}))\in D(2\eps).\label{eq:y1->y0}
\end{equation}
Then, the principal symbol $a_{0}^{1}(x)$ in (\ref{eq:expansion1})
reads\begin{equation}
a_{0}^{1}(y^{1})=e^{i\beta^{1}}\, a^{0}(y^{0})\,\left|\det\frac{\partial y^{0}}{\partial y^{1}}\right|^{1/2}\,\tilde{\pi}_{\alpha_{1}}(y^{1},dS^{1}(y^{1})),\label{eq:principal}
\end{equation}
with $\beta^{1}$ a constant phase. The higher-order symbols $a_{j}^{1}$
and the remainder $r_{L}$ satisfy the following bounds, for any $\ell\in\IN$:
\begin{align}
&\lVert a_j^1\rVert_{C^l}  \leq C_{l,j}\lVert a^0\rVert _{C^{l+2j}},\qquad0\leq j\leq L-1,\label{eq:higher-order}\\
\begin{split}
&\lVert r_l(\cdot,\hbar)\rVert _{C^l}  \leq C_{l,L}\lVert a^0\rVert _{C^{l+2L+d}},\\
&\hspace{0.2in}\text{$r_l=\cO\Biggl(\biggl(\frac{\hbar}{\hbar+\dist(\bullet,\pi\widetilde{E}_{\alpha_1}}\biggr)^{\infty}\Biggr)$ outside $\pi\widetilde{E}_{\alpha_1}$.}
\end{split}\label{eq:remainder}
\end{align}
The constants $C_{\ell,j}$ depend on the Lagrangian $\Lambda^{0}$.
\end{prop}
The proximity of $\Lambda^{0}$ from the unstable manifold $\Lambda_{\alpha_{0}}\defeq W^{u0}(\rho_{\alpha_{0}})$
has another consequence. The map $y^{1}\mapsto y^{0}$ in (\ref{eq:y1->y0})
, projection of $g_{\rest\Lambda^{1}}^{-1}$ to the coordinates $y$,
is close to the projection of $g_{\rest g(\Lambda_{\alpha_{0}})}^{-1}$.
For this reason, the Jacobian $\det\frac{\partial y^{0}}{\partial y^{1}}$
appearing in (\ref{eq:principal-n}) is close to the Jacobian of $g_{\rest g(\Lambda_{\alpha_{0}})}^{-1}$.
Using our coarse-grained Jacobians (\ref{eq:coarse-grained Jacobian}),
we find
\begin{equation}
\left|\det\frac{\partial y^{0}}{\partial y^{1}}\right|=J^{u}(\alpha_{0})^{-1}(1+\cO(\gamma_{1},\eps^{\gamma})),\label{eq:Jacobians}
\end{equation}
where $\gamma\in(0,1]$ depends on the Hölder regularity of the unstable foliation.

\subsection{$n$-steps evolution}

The above proposition describes the 1-step evolution $\Pi_{\alpha_{i}}U$.
We need to apply many ($n\sim\log(1/\hbar)$) similar steps. To control
these many steps uniformly w.r.to $n$, we first need to analyze the
evolution of the Lagrangian leaf $\Lambda^{0}$ through the classical
evolution corresponding to the operator $\Pi_{\bal}$: for $i=0,\ldots,n-1$,
the leaf $\Lambda^{i+1}$ is obtained by truncating $\Lambda^{i}$
on $\supp\tilde{\pi}_{\alpha_{i}}\subset\tilde{E}_{\alpha_{i}}$,
and then evolving this truncated leaf through $g^{1}$. We will assume
that the sequence $\bal$ is \emph{admissible}, in the sense that
$\Lambda^{i}$ is nonempty for all $i=1,\ldots,n$. Then, the Anosov
structure of the geodesic flow induces the following \emph{inclination
lemma} \cite{KatHas95}, which describes the leaves $\Lambda^{i}$
in the adapted coordinates $(y,\eta)$ on $\tilde{E}_{\alpha_{i}}$:
\begin{lem}
\label{lem:Inclination}Assume the Lagrangian leave $\Lambda^{0}\subset\tilde{E}_{\alpha_{0}}$
belongs to a certain unstable $\gamma_{1}$-cone, as defined in Def.
\ref{def:gamma_1-cone}. Then, the Lagrangians $\Lambda^{i}\subset\tilde{E}_{\alpha_{i}}$,
$i=1,\ldots,n$, also belong to the corresponding unstable $\gamma_{1}$-cones:\[
\Lambda^{i}\subset\left\{ (y,dS^{i}(y)),\:|y|\leq\eps\right\} ,\qquad\sup_{|y|\leq\eps}\left\Vert d^{2}S^{i}(y)\right\Vert \leq\gamma_{1}.\]
We also have a uniform control on the higher derivatives of the generating
functions $S^{i}$. For any $\ell>1$, there exists $\gamma_{\ell}>0$
such that, assuming $dS^{0}$ is in the $\gamma_{1}$-cone and satisfies
$\left\Vert dS^{0}\right\Vert _{C^{\ell}}\leq\gamma_{\ell}$, then
the evolved Lagrangians also satisfy\[
\left\Vert dS^{i}\right\Vert _{C^{\ell}}\leq\gamma_{\ell},\qquad i=1,\ldots,n.
\]
\end{lem}
The above Lemma shows that the evolution $\Lambda^{i}\mapsto\Lambda^{i+1}$
remains uniformly under control at long times. Putting together the
Lemma with Prop. \ref{pro:1-step}, we get the following
\begin{prop}
\label{pro:n-step}For $\hbar<\hbar_{0}$ and $n\leq C\log\hbar^{-1}$,
take any sequence $\bal=\alpha_{0}\cdots\alpha_{n}$ of length $n+1$.
Consider a Lagrangian leaf $\Lambda^{0}\subset\tilde{E}_{\alpha_{0}}\cap\cE_{\eta_{0}}$
in some unstable $\gamma_{1}$-cone, and an associated Lagrangian
state $\psi^{0}\in L^{2}(X)$ localized on this leaf, of the form
\[
\cU_{\alpha_{0}}\psi^{0}(y)=a^{0}(y)\, e^{iS^{0}(y)/\hbar},\qquad a^{0}\in C_{c}^{\infty}(D(\eps)).
\]
We are interested in the evolved state 
\[
\psi^{n}\defeq\Pi_{\alpha_{n}}U\Pi_{\alpha_{n-1}}U\cdots\Pi_{\alpha_{1}}U\psi^{0}.
\]
Then, 

i) If the manifold $\Lambda^{n}$ (obtained from $\Lambda^{0}$ from
the classical evolution) is empty, then $\left\Vert \psi^{n}\right\Vert _{L^{2}}=\cO(\hbar^{\infty})$.

ii) Otherwise, $\psi^{n}$ is a Lagrangian state associated with $\Lambda^{n}$.
It reads\[
\cU_{\alpha_{n}}\psi^{n}(y)=a^{n}(y,\hbar)\, e^{iS^{n}(y)/\hbar}+\cO_{\cS}(\hbar^{\infty}),\]
where the symbol $a^{n}(\bullet,\hbar)$ is supported in $D(\eps)$
and satisfies the bound
\[
\left\Vert a^{n}(\bullet,\hbar)\right\Vert _{C^{0}(D(\eps))} \leq C\, J(\alpha_{0}\cdots\alpha_{n-1})^{-1/2}\left\Vert a^{0}\right\Vert _{C^{0}}.
\]
As a consequence, we obtain the $L^{2}$ estimate 
\begin{equation}
\left\Vert \psi^{n}\right\Vert _{L^{2}(X)}\leq C\, J(\alpha_{0}\cdots\alpha_{n-1})^{-1/2}\left\Vert a^{0}\right\Vert _{C^{0}}.\label{eq:psi^n}\end{equation}
The constant $C$ is uniform when the function $S^{0}$ generating
$\Lambda^{0}$ remains in a bounded set in the $C^{\infty}$ topology. \end{prop}
\begin{proof}
This proposition is proved in \cite[Lemma 3.5]{AnaNo07-2}, but we
will rather use the notations of a similar result valid in a more
general setup in \cite[Prop.4.1]{NoZw09}. The strategy consists in
a tedious but straightforward bookkeeping of the properties of the
symbols \[
a^{i}(y,\hbar)=\sum_{j=0}^{L-1}\hbar^{j}a_{j}^{i}(y)+r_{L}^{i}(y,\hbar)\]
associated with the intermediate states $\cU_{\alpha_{i}}\psi^{i}(y)$.
Namely, one manages to control the $C^{\ell}$ norms of the symbols
and of the remainder, using the equations (\ref{eq:principal-n},\ref{eq:higher-order},\ref{eq:remainder}).

The principal symbol $a_{0}^{n}$ is given by the explicit formula
\begin{equation}
a_{0}^{n}(y^{n})=e^{i\sum_{i=1}^{n}\beta^{n}}\, a^{0}(y^{0})\,\prod_{i=1}^{n}\left|\det\frac{\partial y^{i-1}}{\partial y^{i}}\right|^{1/2}\,\tilde{\pi}_{\alpha_{i}}(y^{i},dS^{i}(y^{i})),\label{eq:principal-n}
\end{equation}
where each $y^{i-1}=\pi g_{S^{i}}^{-1}(y^{i},dS^{i}(y^{i}))$ is the
coordinate of the $(i-1)$-th iterate of the point $(y^{0},dS^{0}(y^{0}))\in\Lambda^{0}$.
This formula shows that $a_{0}^{n}$ results from a transport of the
amplitude (or half-density) $a^{0}$ through the flow, and a multiplication
by successive factors $|\tilde{\pi}_{\alpha_{i}}|\leq1+\cO(\hbar)$.
From this expression and (\ref{eq:Jacobians}) we directly get the
$C^{0}$ bound
\begin{equation}
\left\Vert a_{0}^{n}\right\Vert _{C^{0}(D(\eps))} \leq C\, J(\alpha_{0}\cdots\alpha_{n-1})^{-1/2}\left\Vert a^{0}\right\Vert _{C^{0}}.\label{eq:a^n_0}
\end{equation}
\begin{figure}
\caption{Each symbol $a_{j}^{i}$ is linked to its direct {}``descendents''.
Vertical arrows represent operators of transport+multiplication, while
oblique arrows include a certain number of differentiations, as given
in (\ref{eq:higher-order},\ref{eq:remainder}).\label{cap:recurrence}}

\includegraphics[scale=0.7]{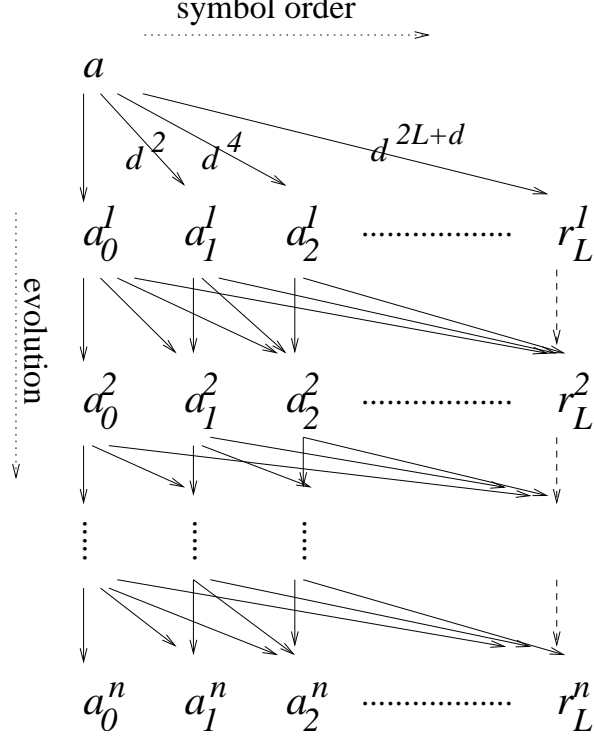}
\end{figure}

The derivatives $\frac{\partial^{\ell}a_{0}^{n}}{\partial(y^{n})^{\ell}}$
are computed by applying the Leibnitz rule to the product (\ref{eq:principal-n})
(which leads to $\cO(n^{\ell})$ terms), and then the chain rule $\frac{\partial f(y^{i})}{\partial y^{n}}=\frac{\partial f(y^{i})}{\partial y^{i}}\frac{\partial y^{i}}{\partial y^{n}}$.
Using the fact that the Jacobian matrices $\frac{\partial y^{i}}{\partial y^{n}}$
are uniformly bounded, one obtains the bounds \[
\left\Vert a_{0}^{n}\right\Vert _{C^{\ell}}\leq C_{\ell}\, n^{\ell}\, J(\alpha_{0}\cdots\alpha_{n-1})^{-1/2}\left\Vert a^{0}\right\Vert _{C^{\ell}},\quad\ell\geq0.\]
The higher-order symbol $a_{j}^{i}$, $1\leq j\leq L-1$, is obtained
by the same transport-and-multiplication of the symbol $a_{j}^{i-1}$,
but also by transporting and differentiating $2(j-j')$ times the
symbols $a_{j'}^{i-1}$, $j'\leq j$. This procedure is sketched in
Fig.\ref{cap:recurrence}. On this figure, each symbol $a_{j}^{n}$
results from the sum of $\cO(n^{j})$ paths starting from $a^{0}$,
each path consisting in a succession of {}``long'' vertical evolutions,
and $j$ {}``oblique'' evolutions, involving altogether $2j$ differentiations
performed at various stages. We have seen above that each differentiation
leads, through Leibnitz's rule, to a factor $\cO(n)$. Taking into
account the number of paths, we obtain the bounds\[
\left\Vert a_{j}^{n}\right\Vert _{C^{0}}\leq C_{j}\, n^{3j}\, J(\alpha_{0}\cdots\alpha_{n-1})^{-1/2}\left\Vert a^{0}\right\Vert _{C^{2j}},\quad1\leq j\leq L-1,\]
and $\ell$ differentiations of the symbol $a_{j}^{n}$ provide, for
the same reasons as above, an additional factor $n^{\ell}$:
\[
\left\Vert a_{j}^{n}\right\Vert _{C^{\ell}}\leq C_{j,\ell}\, n^{\ell+3j}\, J(\alpha_{0}\cdots\alpha_{n-1})^{-1/2}\left\Vert a^{0}\right\Vert _{C^{\ell+2j}},\quad1\leq j\leq L-1.
\]
At each stage, one also gets a remainder $r_{L}^{i}$, which results
from the symbols $a_{j}^{i-1}$ through transport, multiplication
and differentiation, as well as from the \emph{unitary} evolution
of the previous remainder $r_{L}^{i-1}$ (dashed vertical arrow).
Taking the above bounds into account, one easily obtains the $L^{2}$
bound 
\[
\left\Vert r_{L}^{n}\right\Vert _{L^{2}}\leq C_{L}\left\Vert a^{0}\right\Vert _{C^{2L+d}}.
\]
This last estimate shows that the full symbol $a^{n}(\bullet,\hbar)$
is dominated by the principal symbol $a_{0}^{n}$, so that the estimate
(\ref{eq:a^n_0}) also applies to $\left\Vert a^{n}(\bullet,\hbar)\right\Vert _{C^{0}}$,
and hence to $\left\Vert a^{n}(\bullet,\hbar)\right\Vert _{L^{2}}=\left\Vert \psi^{n}\right\Vert _{L^{2}}+\cO(\hbar^{\infty})$.
\end{proof}
We now conclude the proof of Prop. \ref{pro:hyperb-dispers}. We use
the decomposition (\ref{eq:decompo-sharp}) of the state $\psi_{\alpha_{0}}$,
and apply Prop. \ref{pro:n-step} to each state $\psi^{0}=\cU_{\alpha_{0}}^{*}e_{\eta}$,
$|\eta|\leq2\eps$. Notice that the manifolds $\Lambda_{\eta}$ remain
in a bounded cone in $C^{\infty}$: a generating function for $\Lambda_{\eta}$
is simply $S_{\eta}(y)=\la\eta,y\ra$. Therefore, the constant $C$
in the estimate (\ref{eq:psi^n}) is uniform w.r.to $\eta$. The triangle
inequality then implies the bound
\[
\left\Vert \Pi_{\alpha_{n}}U\Pi_{\alpha_{n-1}}\cdots\Pi_{\alpha_{1}}U\psi_{\alpha_{0}}\right\Vert \leq C\,\hbar^{-d/2}J(\alpha_{0}\cdots\alpha_{n-1})^{-1/2}\int\left|\hat{\psi}_{\alpha_{0}}(\eta)\right|\, d\eta+\cO_{L^{2}}(\hbar^{\infty}).
\]
The RHS contains the $L^{1}$ bound for the Fourier transform $\hat{\psi}_{\alpha_{0}}(\eta)$.
Since this function is $\cO(\hbar^{\infty})$ outside the set $\left\{ |\eta_{0}|\leq e^{m\delta}\hbar^{1-\delta},\,|\eta'|\leq2\eps\right\} $,
the Cauchy-Schwarz inequality leads to 
\[
\left\Vert \hat{\psi}_{\alpha_{0}}\right\Vert _{L^{1}}\leq\left(C\,\eps\, e^{m\delta}\hbar^{1-\delta}\right)^{1/2}\,\left\Vert \hat{\psi}_{\alpha_{0}}\right\Vert _{L^{2}}+\cO_{L^{2}}(\hbar^{\infty}).
\]
Since $\left\Vert \hat{\psi}_{\alpha_{0}}\right\Vert _{L^{2}}=\left\Vert \psi_{\alpha_{0}}\right\Vert _{L^{2}}\leq1+\cO(\hbar^{\infty})$,
we obtain (\ref{eq:hyperb-dispers}). \hfill$\square$

\section{Extension to Anosov toral diffeomorphisms\label{sec:Anosov-maps}}

In this section we will give the few modifications necessary to prove
Thm.\ref{thm:bound-general2}, \emph{ii) } dealing with Anosov toral
diffeomorphisms. Although the strategy of proof is exactly the same
as for the Laplacian eigenstates, the quantum setting is slightly
different, and not so well-known as the one used above. One advantage
of dealing with maps instead of flows is that we will not need any
sharp energy cutoff: the torus will take the place of the energy shell
$\cE$, so there won't be any need to localize ourselves on a submanifold.

Let us mention that the bound \ref{eq:bound-maps} is not a new result
in the case of \emph{linear} hyperbolic symplectomorphisms on the 2-dimensional torus: 
as explained
in $\S$\ref{sub:Counterexamples}, it is a consequence of Brooks's
more precise result (Theorem~\ref{thm:Brooks}). For linear symplectomorphisms in higher dimension, an improvement of the lower bound \eqref{eq:bound-maps} has been recently obtained by G.~Rivière \cite{Riv10}.

\subsection{Quantum mechanics on a torus phase space}

We briefly recall the properties of quantum mechanics associated with
the torus phase space. $\mathbb{T}^{2d}=\mathbb{R}^{2d}/\mathbb{Z}^{2d}$
is equipped with the symplectic form $\omega=\sum_{i=1}^{d}d\xi_{i}\wedge dx_{i}$.
Quantum states are defined as the distributions $\psi\in\cS'(\IR^{d})$
which are $\IZ^{d}$-periodic, and the $\hbar$-transforms of which
are also $\IZ^{d}$-periodic:
\begin{equation}
\forall n\in\IZ^{d},\forall x\in\IR^{d},\quad\psi(x+n)=\psi(x),\qquad\forall\xi\in\IR^{d},\quad\left(\cF_{\hbar}\psi\right)(\xi+n)=\left(\cF_{\hbar}\psi\right)(\xi).\label{eq:biperiodic}
\end{equation}
A simple calculation shows that such distributions can be nontrivial
iff
\[
\hbar=\hbar_{N}\defeq(2\pi N)^{-1}\quad\text{for some integer $N>0$.}
\]
Such values of $\hbar$ are called \emph{admissible}. From now on
we will only consider admissible values of $\hbar$. The distributions
(\ref{eq:biperiodic}) then form a $N^{d}$-dimensional subspace of
$\cS'(\R^{d})$, (denoted by $\hn$), which is spanned by the following
basis of {}``Dirac combs'':
\begin{equation}
e_{j}(x)=N^{-d/2}\sum_{\nu\in\IZ^{d}}\delta(x-\frac{j}{N}-\nu),\qquad j\in(\Z/N\Z)^{d}\simeq\left\{ 0,\ldots,N-1\right\} ^{d}.\label{eq:position-basis}
\end{equation}
It is natural to equip $\hn$ with the hermitian norm $\norm{\bullet}_{\hn}$
for which the basis $\left\{ e_{j},\: j\in(\Z/N\Z)^{d}\right\} $
is orthonormal. One can construct the space $\hn$ by {}``projecting''
states $\psi\in\cS(\IR^{d})$:
\[
\forall\psi\in\cS(\R^{d}),\quad\Pi_{N}\psi(x)\defeq\sum_{\mu,\nu\in\Z^{d}}\psi(x-\nu)\, e^{2i\pi\la\mu,x\ra}
\]
belongs to $\hn$, and the map $\Pi_{N}$ is surjective. In general,
there is no obvious link between the norms $\norm{\psi}_{L^{2}}$
and $\norm{\Pi_{N}\psi}_{\hn}$. Yet, imposing some localization for
$\psi$, we get the following relation:
\begin{lem}
\label{lem:projection}Assume that $\psi\in\cS(\IR^{d})$ is microlocalized
inside a set $E\subset\R^{2d}$ of diameter $<1/2$. Then, its projection
$\Pi_{N}\psi\in\hn$ satisfies 
\[
\norm{\Pi_{N}\psi}_{\hn}=\norm{\psi}_{L^{2}}+\cO(\hbar^{\infty}).
\]
\end{lem}
Let us now describe observables on $\T^{2d}$. Any smooth functon
on $\mathbb{T}^{2d}$ is also a $\mathbb{Z}^{2d}$-periodic function
on $\IR^{2d}$. It is natural to introduce symbol classes \[
S_{\nu}^{k}(\mathbb{T}^{2d})=\left\{ f(\hbar)\in C^{\infty}(\T^{2d}),\;\left|\partial_{x}^{\alpha}\partial_{\xi}^{\beta}f(\hbar)\right|\leq C_{\alpha,\beta}\hbar^{-k-\nu|\alpha+\beta|}\right\} ,\quad\nu\in[0,1/2),\; k\in\IR\]
(due to periodicity, we cannot have growth or decay in the variable
$\xi$ as in the classes (\ref{eq:S_nu})).

Observables $f=(f_{\hbar})\in S_{\nu}^{k}(\T^{2d})$ can be Weyl-quantized
as operators $\Oph(f)$ acting on $\cS(\IR^{d})$, but also on $\cS'(\IR^{d})$
by duality. We already know that any observable $f\in S_{\nu}^{0}(\T^{2d})$
satisfies \[
\norm{\Oph(f)}_{L^{2}\to L^{2}}\leq\norm{f}_{\infty}+\cO(\hbar^{1-2\nu}).\]

\begin{prop}
\cite{BDB96}Take $f\in C^{\infty}(\T^{2d})$. For any admissible
$\hbar=\hbar_{N}$, the operator $\Oph(f)$ leaves invariant the subspace
$\hn\subset\cS'(\IR^{d})$. Let us call $\OpN(f)$ its restriction
on $\hn$. These two operators satisfy

\begin{equation}
\norm{\OpN(f)}_{\hn\circlearrowleft}\leq\norm{\Oph(f)}_{L^{2}\circlearrowleft}.\label{eq:hn-L2}\end{equation}

\end{prop}
This property allows to carry the pseudodifferential calculus on $\IR^{2d}$
down to the torus. For instance, for any two observables $f,g\in S_{0}^{0}(\T^{2d})$,
the product $\OpN(f)\OpN(g)$ is the restriction of $\Oph(f)\Oph(g)=\Oph(f\sharp g)$.
One can check that $f\sharp g$ is a periodic function, and that each
term in the expansion $f\sharp g=\sum_{j=0}^{L-1}\hbar^{j}(f\sharp g)_{j}+\hbar^{L}R_{L}(f,g,\hbar)$
(including the remainder) is also periodic. Hence, the estimates on
$\norm{\Oph((f\sharp g)_{j})}_{L^{2}}$ and $\norm{\Oph(R_{L}(f,g,\hbar))}_{L^{2}}$
can be directly translated to estimates of their restrictions on $\hn$.

\subsection{Quantum maps on the torus}

We now give a brief overview of what we mean by a quantum propagator
associated with a smooth symplectic diffeomorphism $\kappa:\T^{2d}\circlearrowleft$.
We will not try to provide a general recipe to {}``quantize'' all
possible $\kappa$, but only give some relevant examples. 
\begin{enumerate}
\item consider the flow $g^{t}$ generated by some Hamilton function $p\in C^{\infty}(\T^{2d})$.
If we quantize $p$ into $P(\hbar)=\Oph(p)$, then the quantum propagator
quantizing $g^{t}$ is naturally $U^{t}=e^{-itP(\hbar)/\hbar}$. Since
$P(\hbar)$ leaves $\hn$ invariant, so does the propagator. Hence,
the map $\kappa\defeq g^{1}$ is quantized by 
\[
U_{N}(\kappa)=e^{-iP(\hbar)/\hbar}\restriction\hn=\exp\left(-i2\pi N\OpN(p)\right),
\]
which is unitary on $\hn$. The Egorov property (\ref{eq:Egorov})
can be directly translated to the torus setting:
\begin{equation}
U_{N}(\kappa)^{-1}\OpN(f)U_{N}(\kappa)=\OpN(f\circ\kappa)+\cO_{\hn\to\hn}(\hbar).\label{eq:Egorov-torus}
\end{equation}

\item If $\kappa$ is a linear symplectomorphism of $\T^{2d}$ associated
with the symplectic matrix $S_{\kappa}=\left(\begin{array}{cc}
A & B\\
C & D\end{array}\right)\in Sp(2d,\Z)$, then it can be quantized as a metaplectic transformation $U_{\hbar}(\kappa)$,
which acts unitarily over $L^{2}(\IR^{d})$. Provided $S_{\kappa}$
satisfies some {}``checkerboard conditions''\cite{BDB96}, the extension
of $U_{\hbar}(\kappa)$ to $\cS'(\R^{d})$ leaves $\hn$ invariant;
its restriction $U_{N}(\kappa)=U_{\hbar}(\kappa)\restriction\hn$
is unitary on $\hn$. If $S_{\kappa}$ has no eigenvalues on the unit
circle (meaning that the matrix is hyperbolic), the map $\kappa$
is Anosov. Such a map is often called a {}``generalized quantum cat
map'', by reference to {}``Arnold's cat map''. It satisfies an
exact Egorov property: $U_{N}(\kappa)^{-1}\OpN(f)U_{N}(\kappa)=\OpN(f\circ\kappa)$.
\item Let us combine these two types of maps, namely a linear symplectomorphism
$\kappa_{0}$ (satisfying the checkerboard condition), and the flow
$g^{t}$ generated by some $p\in C^{\infty}(\T^{2d})$, to get 
\begin{equation}
\kappa=g^{1}\circ\kappa_{0}.\label{eq:map-admissible}
\end{equation}
If $S_{\kappa_{0}}$ is hyperbolic and $p$ is small enough in the
$C^{2}$ topology, then it is known that $\kappa$ still has the Anosov
property (and it is topologically conjugate with $\kappa_{0}$). The
quantum propagator can be defined as \[
U_{N}(\kappa)\defeq\exp\left(-i2\pi N\OpN(p)\right)\circ U_{N}(\kappa_{0}).\]
It obviously satisfies the Egorov property (\ref{eq:Egorov-torus}).
\end{enumerate}
The long-time Egorov theorem (Prop. \ref{sub:Egorov-Ehrenfest}) can
also be brought to the torus framework:
\begin{prop}
Choose $\eps>0$ small and $\nu\in(\frac{1-\eps}{2},\frac{1}{2})$.
Take $f\in S_{0}^{0}(\T^{2d})$. Then, for any admissible $\hbar=\hbar_{N}$
and any time $n=n(\hbar)$ in the range $|n|\leq T_{\eps,\hbar}$,
we have 
\begin{multline}\label{eq:Egorov-long-torus}
U_N^{-n}\OpN(f)U_N^n=\OpN(\tilde f_n)+\cO(\hbar^{\infty}),\\
\text{with $\tilde f_n\in S_{\nu}^0(\IT^{2d}),\tilde f_n-f\circ\kappa^n\in S_{\nu}^{-(1+\eps)/2}(\IT^{2d})$.}
\end{multline}
\end{prop}

\subsection{Quantum partitions on the torus}

Through the quantization $f\mapsto\OpN(f)$, we can associate to any
semiclassical sequence of normalized states $(u_{N}\in\hn)_{N\to\infty}$
one or several semiclassical measures, as in §\ref{sub:Semiclassical-measures}.
Starting from an Anosov map $\kappa$ of the form (\ref{eq:map-admissible}),
we consider sequences $(\psi_{N}\in\hn)_{N\to\infty}$ where each
$\psi_{N}$ is an eigenstate of $U_{N}(\kappa)$. After possibly extracting
a subsequence, this sequence admits a single semiclassical measure
$\mu_{sc}$, which is a probability $\kappa$-invariant measure on
$\T^{2d}$. To analyze the KS entropy of this measure (defined as
in §\ref{sub:KSentropy} after replacing $g^{1}$ by $\kappa$), we
set up a partition $\cP=\bigsqcup_{k=1}^{K}E_{k}$ of $\T^{2d}$ of
diameter $\epsilon$, such that $\mu_{sc}(\partial\cP)=0$. The partition
$\cP$ is smoothed into $\cP_{sm}=\left\{ \pi_{k}\right\} $, where
$\pi_{k}\in C^{\infty}(\T^{2d},[0,1])$, each $\pi_{k}$ is supported
near $E_{k}$, and $\sum_{k}\pi_{k}=1$. The quantum partition is
defined similarly as in $\S$\ref{sub:quantum-partition}, except
that $\Pi_{k}=\OpN(\tilde{\pi}_{k})$, $\tilde{\pi}_{k}\in S(\T^{2d})$
satisfy 
\begin{equation}
\sum_{k=1}^{K}\Pi_{k}^{2}=Id_{\hn}+\cO(\hbar^{\infty}).\label{eq:Quantum Partition-torus}
\end{equation}
The refined quasiprojectors $\Pi_{\bal}$ are defined similarly as
in (\ref{eq:quantum-refined}), after setting $U=U_{N}(\kappa)$.
The symbolic measures $\mu_{N},\tilde{\mu}_{N}$ associated with $\psi_{N}$
are defined as in (\ref{eq:mu^u-definition},\ref{eq:tilde-mu^u}).

\subsection{Hyperbolic dispersive estimate}

One needs to adapt the proof of the hyperbolic dispersion estimate
(\ref{eq:hyperb-dispers-torus}) to the torus setting. We can prove
the following
\begin{prop}
\label{pro:hyperb-dispers-torus}Fix a constant $C_{1}\gg1$. $ $Then
there exists $C>0$ such that for any $N\geq1$, any $0\leq n\leq C_{1}\log N$
and any sequence $\bal$ of length $n$, the following estimate holds:
\begin{equation}
\left\Vert \Pi_{\bal}\right\Vert \leq C\, N^{\frac{d}{2}}\, J_{n}^{u}(\bal)^{-1/2}.\label{eq:hyperb-dispers-torus}
\end{equation}
\end{prop}
Here $J_{n}^{u}(\bal)=\prod_{j=0}^{n-1}J^{u}(\alpha_{j})$ is the
coarse-grained unstable Jacobian of the map $\kappa$. Notice that
the power of $N\sim\hbar^{-1}$ is $d/2$ instead of $(d-1+\delta)$
in (\ref{eq:hyperb-dispers}).

\subsubsection*{Proof of the proposition.}

The proof of this estimates proceeds along the same lines as in $\S$\ref{sec:Hyperbolic dispersive estimate},
that is by explicitly computing the action of the operator $\Pi_{\bal}$
on an arbitrary normalized state $\Psi\in\hn$. To do this, we first
expand each localized piece $\psi=\Pi_{k}\Psi$, $k=1,\ldots,K$,
into an well-chosen orthonormal basis $\left\{ f_{j},\: j\in(\Z/N\Z)^{d}\right\} $
obtained from the original basis (\ref{eq:position-basis}) through
a well-chosen quantized linear symplectomorphism $U(\tilde{\kappa}_{k})$.
Let us recall a few facts about linear symplectomorphisms of the torus.
Any $\tilde{\kappa}\in Sp(2d,\IR)$ acting on $\R^{2d}$ maps a {}``position
Lagrangian'' $\Lambda_{x_{0}}^{\IR}=\left\{ x=x_{0},\:\xi\in\IR\right\} ,\, x_{0}\in\R^{d}$
into another {}``linear'' Lagrangian%
\footnote{We use the representation $S_{\tilde{\kappa}}=\left(\begin{array}{cc}
A & B\\
C & D\end{array}\right)$ and assume for simplicity that $B$ is nonsingular.%
}\begin{equation}
\tilde{\Lambda}_{x_{0}}^{\IR}=\tilde{\kappa}(\Lambda_{x_{0}}^{\IR})=\left\{ \xi=DB^{-1}x+(C-DB^{-1}A)x_{0},\: x\in\IR^{d}\right\} .\label{eq:Lagran-R}\end{equation}
Each position Lagrangian $\Lambda_{x_{0}}^{\R}$ is associated with
a {}``position eigenstate'' $e_{x_{0}}^{\R}(x)=(2\pi\hbar)^{d/2}\,\delta(x-x_{0})$.
The metaplectic operator $U_{\hbar}(\tilde{\kappa})$ maps this position
state into the Lagrangian state 
\begin{equation}\label{eq:f_x_0}
f_{x_0}^{\IR}\defeq U_{\hbar}(\tilde\kappa)e_{x_0}^{\IR},\qquad
f_{x_0}^{\IR}(x)=\frac{1}{\sqrt{\det(B)}} e^{(i/2\hbar)(\la DB^{-1}x,x\ra-2\la x,B^{-1}x_0\ra+\la B^{-1}Ax_0,x_0\ra)},
\end{equation}
which is associated with the Lagrangian $\tilde{\Lambda}_{x_{0}}^{\IR}$.
If $\tilde{\kappa}$ has integer coefficients and satisfies the {}``checkerboard
condition'', then this construction can be brought down to the torus.
Each basis state $e_{j}$ in (\ref{eq:position-basis}) is associated
with the Lagrangian $\Lambda_{j/N}^{\T}$, projection on $\T^{2d}$
of $\Lambda_{j/N}^{\IR}$. The state $f_{j}\defeq U_{N}(\tilde{\kappa})e_{j}\in\hn$
is a {}``Lagrangian state on the torus'' associated with the projected
Lagrangian $\tilde{\Lambda}_{j/N}^{\T}=\tilde{\kappa}(\Lambda_{j/N}^{\T})$. 

For each partiton component $E_{k}$, we select an appropriate linear
transformation $\tilde{\kappa}_{k}$. In each neighbourhood $\tilde{E}_{k}$
of $\supp\pi_{k}$ (assumed to have diameter $\leq\epsilon$), we
use an adapted coordinate chart $\left\{ (y,\eta)\right\} $ as in
$\S$\ref{sub:Decomposition-Lagrangian}. Using these coordinates,
we consider the family of $\gamma_{1}$-cones defined in Def. \ref{def:gamma_1-cone}.
\begin{lem}
Provided the diameter of $\tilde{E}_{k}$ is small enough, we can
choose the automorphism $\tilde{\kappa}_{k}$ such that each connected
component of $\tilde{\Lambda}_{x_{0}}^{\IT}\cap\tilde{E}_{k}$ belongs
to the unstable $\gamma_{1}$-cone. \end{lem}
\begin{proof}
Indeed, the pieces of unstable manifolds inside $\tilde{E}_{k}$ are
then {}``almost flat'' and {}``almost parallel'', so they can
be approached by a family of (local) linear Lagrangians. Besides,
any linear Lagrangian $\tilde{\Lambda}_{0}^{\T}$ can be approached
by a {}``rational'' Lagrangian $\tilde{\kappa}(\Lambda_{0}^{\R})$,
$\tilde{\kappa}\in Sp(2d,\IZ)$.
\end{proof}
For each $k=1,\ldots,K$ we expand $\psi=\Pi_{k}\Psi$ in the o.n.
basis $\left\{ f_{j}=U_{N}(\tilde{\kappa}_{k})e_{j},\; j\in(\Z/N\Z)^{d}\right\} $.
Using cutoffs $\bbbone_{\tilde{E}_{k}}\succ\pi_{k}^{\sharp}\succ\pi_{k}$
we can localize the Lagrangian states $f_{j}$ inside $\tilde{E}_{k}$:
\[
\psi=\sum_{j\in(\Z/N\Z)^{d}}\psi_{j}\, f_{j}=\sum_{j\in(\Z/N\Z)^{d}}\psi_{j}\,\tilde{f}_{j}+\cO(\hbar^{\infty}),\qquad\tilde{f}_{j}\defeq\OpN(\pi_{k}^{\sharp})f_{j}.\]
We can thus proceed as in §\ref{sec:Hyperbolic dispersive estimate},
namely compute separately each $\Pi_{\bal}\tilde{f}_{j}$. To be able
to use $\S$\ref{sec:Hyperbolic dispersive estimate}, we will switch
back from states in $\hn$ to states in $\cS(\IR^{d})$. 
\begin{lem}
Each localized Lagrangian state $\tilde{f}_{j}=\OpN(\pi_{k}^{\sharp})f_{j}\in\hn$
is equal (up to $\cO_{\hn}(\hbar^{\infty})$) to the projection on
$\hn$ of finitely many Lagrangian states $\tilde{f}_{j,n}^{\IR}\in\cS(\IR^{d})$
microlocalized inside a single representative $\tilde{E}_{k}^{\R}$
of $\tilde{E}_{k}$ in $\R^{2d}$.
\end{lem}
\begin{proof}
The localized Lagrangian state $\tilde{f}_{j}=\OpN(\pi_{k}^{\sharp})f_{j}$
is microlocalized on $\tilde{\Lambda}_{j/N}^{\IT}\cap\tilde{E}_{k}$,
which is a finite collection of Lagrangian leaves. Each of these leaves
is the projection on the torus of a leaf of the form $\tilde{\Lambda}_{j/N+n}^{\IR}\cap\tilde{E}_{k}^{\R}$,
where $\tilde{\Lambda}_{\bullet}^{\IR}$ is given in (\ref{eq:Lagran-R}),
$n\in\IZ^{d}$ and $\tilde{E}_{k}^{\R}$ is a certain (arbitrary)
representative of $\tilde{E}_{k}$ in $\R^{2d}$. Accordingly, each
$\tilde{f}_{j}$ can be split into a finite linear combination of
Lagrangian states $\tilde{f}_{j,n}\in\hn$ supported on these individual
leaves. If we call $\pi_{k}^{\sharp\R}\in C_{c}^{\infty}(\IR^{2d})$
the component of $\pi_{k}^{\sharp}$ supported on the representative
$\tilde{E}_{k}^{\R}$, then the Lagrangian state $\tilde{f}_{j,n}^{\IR}=\Oph(\pi_{k}^{\sharp\R})f_{j/N+n}^{\IR}$
is microlocalized inside $\tilde{E}_{k}^{\R}$. We claim that \[
\tilde{f}_{j,n}=\Pi_{N}\tilde{f}_{j,n}^{\R}+\cO_{\hn}(\hbar^{\infty}).\]
We will only give the proof in the case where $f_{j}$ is a momentum
state $\la e_{l},f_{j}\ra=N^{-d/2}\, e^{2i\pi\la l,j\ra/N}$ associated
with the momentum Lagrangian $\left\{ \xi=\xi_{j}=\frac{j}{N}\right\} $,
and the cutoff is the operator $\OpN(\chi_{1})\OpN(\chi_{2})$, where
$\chi_{1}(x)$ (resp. $\chi_{2}(\xi)$) is obtained by periodizing
$\chi_{1}^{\R}\in C_{c}^{\infty}((0,1))$ (resp. $\chi_{2}^{\R}\in C_{c}^{\infty}((0,1))$).
In that case, the Lagrangian state $\tilde{f}_{j}=\OpN(\chi_{1})\OpN(\chi_{2})f_{j}$
admits the components 
$$\la e_{l},\tilde{f}_{j}\ra=N^{-d/2}\,\chi_{1}(\frac{l}{N})\,\chi_{2}(\frac{j}{N})\, e^{2i\pi\la l,j\ra/N}.
$$
On the other hand, the state $\tilde{f}_{j}^{\R}=\Oph(\chi_{1}^{\R})\Oph(\chi_{2}^{\R})f_{j}^{\R}$
can be expressed as $\tilde{f}_{j}^{\R}(x)=\chi_{1}^{\R}(x)\,\chi_{2}^{\R}(\xi_{j})\, e^{i\la\xi_{j},x\ra/\hbar}$.
The projection on $\hn$ of that state gives 
\begin{align*}
\Pi_{N}\tilde{f}_{j}^{\R}(x) & =\chi_{2}^{\R}(\xi_{j})\sum_{\nu,\mu}\chi_{1}^{\R}(x-\nu)\, e^{2i\pi N\la\xi_{j}+\mu,x\ra}\\
 & =\chi_{2}^{\R}(\xi_{j})\chi_{1}^{\R}(\left[x\right])\, e^{2i\pi N\la\xi_{j},x\ra}\,\delta_{\Z^{d}}(Nx)\\
 & =N^{-d/2}\chi_{2}^{\R}(\frac{j}{N})\sum_{l\in(\Z/N\Z)^{d}}e_{l}(x)\,\chi_{1}(\frac{l}{N})\, e^{2i\pi\la j,l\ra}=\tilde{f}_{j}(x).
\end{align*}

\end{proof}
This Lemma allows us to use our control of the evolution of the Lagrangian
states $\tilde{f}^{\R}\defeq\tilde{f}_{j,n}^{\IR}$ through the sequence
of operators $U_{\hbar}(\kappa)$ and $\Pi_{k}=\Oph(\tilde{\pi}_{k})$.
We have to be a little careful when applying Props. \ref{pro:1-step}
and \ref{pro:n-step}, because $\tilde{\pi}_{\alpha_{j}}$ are now
periodic symbols. Still, the Lagrangian state $\tilde{f}^{\IR}$ is
localized inside a single copy $\tilde{E}_{\alpha_{0}}^{\R}$ (of
diameter $\leq\eps$), so its image $U_{\hbar}(\kappa)\tilde{f}^{\R}$
is a Lagrangian state microlocalized in a set of diameter $\cO(\eps)$,
which can intersect at most one copy $\tilde{E}_{\alpha_{1}}^{\R}+n_{1}$.
As a result, the state $\Pi_{\alpha_{1}}U_{\hbar}(\kappa)\tilde{f}^{\R}$
is also localized in this single copy, and is of the form of the state
$\psi_{1}$ in Prop. \ref{pro:1-step}. Importantly, the estimates
we have on the remainder $r_{L}(\bullet,\hbar)$ in the expansion
$ $(\ref{eq:expansion1}) show that $e^{iS_{1}/\hbar}r_{L}$ is microlocalized
in the single copy $\tilde{E}_{\alpha_{1}}+n_{1}$, and decays fast
away from it. As a result, Lemma \ref{lem:projection} implies that
the projection on $\hn$ of that remainder has a norm comparable with
$\norm{r_{L}}_{L^{2}(\R^{d})}$.

Since the initial Lagrangian piece $\tilde{\Lambda}$ lies in some
$\gamma_{1}$-unstable cone, we can iterate the evolution as in Prop.
\ref{pro:n-step}. At each step we get a Lagrangian state and some
remainder $e^{iS_{t}(y)/\hbar}r_{L}^{t}(y,\hbar)$, which can be projected
to $\hn$ with a control on its norm. When we act on this remainder
through operators $\Pi_{\alpha_{j}}U_{N}(\kappa)$, its norm can increase
at most by a factor $(1+\cO(\hbar^{\infty}))$. Finally, we obtain
a Lagrangian state microlocalized in some $\tilde{E}_{\alpha_{n}}^{\R}+n_{n}$
(and a sum of remainders). The projection in $\hn$ of that state
satisfies the same bound as \ref{eq:psi^n}: 
\[
\left\Vert \Pi_{\bal}\tilde{f}_{j}\right\Vert _{L^{2}(X)}\leq C\, J^{u}(\bal)^{-1/2}.
\]
Summing over all the states $\tilde{f}_{j}$ and taking into account
$\norm{\psi}_{\hn}=\sqrt{\sum_{i\in(\Z/N\Z)^{d}}|\psi_{i}|^{2}}\leq1$,
we obtain (invoking Cauchy-Schwarz) 
\[
\norm{\Pi_{\bal}\Psi}\leq\sum_{j\in(\Z/N\Z)^{d}}|\psi_{j}|C\, J^{u}(\bal)^{-1/2}\leq C\, N^{\frac{d}{2}}\, J^{u}(\bal)^{-1/2}.
\]
\hfill$\square$

\subsection{EUP and subadditivity}

The Hyperbolic dispersive estimate can now be injected in some form
of Entropic Uncertainty Principle as in §\ref{sub:application-of-EUP}.
Since we don't need any energy cutoff, the setting we need is actually
simpler than Prop. \ref{pro:EUP,microlocal}: we can take $S_{c_{1}}=S_{c_{2}}=Id_{\hn}$.
In the application, we also take $\vareps=\hbar^{L}$ for any large
$L>0$. We obtain a lower bound of the form (\ref{eq:applied-EUP})
on the pressures associated with the symbolic measures $\mu_{\hbar},\:\tilde{\mu}_{\hbar}$
and the Ehrenfest time $n=\left\lfloor 2T_{\epsilon,\hbar}\right\rfloor $,
with a constant $C_{cone}(\hbar)=C\hbar^{-d/2}$. Using the Egorov
theorem up to $T_{\eps,\hbar}$, one shows that the quantum pressures
satisfy an approximate subadditivity property similar with (\ref{eq:pressure-approx-subadd}),
which allows to prove \[
\frac{p_{0}^{n_{0}-1}(\mu_{sc},v,\cP)}{n_{o}}\geq-\frac{d\lambda_{\max}}{2}+\cO(\eps).\]
The rest of the proof is unchanged.

\end{document}